\documentclass[reqno,11pt, a4paper]{amsart}

\usepackage[T1]{fontenc}

\usepackage{amssymb}
\usepackage{amsmath}
\usepackage{amsfonts}
\usepackage{amsthm}
\usepackage{mathtools}
\usepackage{mathabx}
\usepackage{fullpage}
\usepackage{comment}
\usepackage[dvipsnames]{xcolor}
\usepackage{graphicx}
\usepackage{float}

\usepackage{mdframed}
\usepackage{enumitem} 

\usepackage[export]{adjustbox}
\usepackage{bbold}
\usepackage[percent]{overpic}
\usepackage{units}

\usepackage{booktabs,array}
\usepackage{url} 
\usepackage{hyperref}

\hypersetup{
    colorlinks=true,
    citecolor=purple!70!blue,
    linkcolor=black,
    filecolor=black,      
    urlcolor=black,
    pdftitle={Wilson lines in the lattice Higgs model at strong coupling},
    }

\usepackage{pgfplots}
\pgfplotsset{compat=1.6}

\usepackage{dashrule}

\usetikzlibrary{calc,quotes,angles,arrows.meta,positioning}
\usetikzlibrary{decorations.pathreplacing,decorations.markings}

\usepackage{float}

\usepackage{caption}
\usepackage{subcaption}

\usepackage[nocompress]{cite}

\allowdisplaybreaks

\theoremstyle{plain}%
\newtheorem{theorem}{Theorem}[section]
\newtheorem{lemma}[theorem]{Lemma}
\newtheorem{proposition}[theorem]{Proposition}
\newtheorem{corollary}[theorem]{Corollary}

\newtheorem*{conjecture*}{Conjecture} 

\newtheorem{sublemma}[theorem]{\normalfont Claim}
 \numberwithin{equation}{section}

\theoremstyle{definition}
\newtheorem{definition}[theorem]{Definition}

\theoremstyle{remark}
\newtheorem{remark}[theorem]{Remark}

\newenvironment{subproof}[1][\proofname]{%
  \proof[Proof of claim]%
}{\endproof}

 \let \leq \leqslant
 \let \geq \geqslant
 \let \ge \geqslant

\DeclareMathOperator{\sgn}{sgn}

\DeclareMathOperator{\support}{supp}

\DeclareMathOperator{\dist}{dist}

\usepackage{fancyhdr}
\usepackage{graphics}
\usepackage{graphicx}
\usepackage{soul}

\usepackage{scalerel,stackengine}
\newcommand\pig[1]{\scalerel*[5.5pt]{\Big#1}{%
  \ensurestackMath{\addstackgap[1.5pt]{\big#1}}}}
\newcommand\pigl[1]{\mathopen{\pig{#1}}}
\newcommand\pigr[1]{\mathclose{\pig{#1}}}

\title{Wilson lines in the lattice Higgs model \\ at strong coupling}

\author{Malin P. Forsstr\"om, Jonatan Lenells, Fredrik Viklund}
\address[Malin P. Forsstr\"om]{Mathematical Sciences
Chalmers University of Technology and University of Gothenburg, 412 96 G\"oteborg, Sweden.}
\email{palo@chalmers.se}
\address[Jonatan Lenells]{Department of Mathematics, KTH Royal Institute of Technology, 100 44 Stockholm, Sweden.}
\email{jlenells@kth.se}
\address[Fredrik Viklund]{Department of Mathematics, KTH Royal Institute of Technology, 100 44 Stockholm, Sweden.}
\email{frejo@kth.se}

\begin{document}

\maketitle 


\begin{abstract}
    We consider the 4D fixed length lattice Higgs model with Wilson action for the gauge field and structure group $\mathbb{Z}_n$. We study Wilson line observables in the strong coupling regime and compute their asymptotic behavior with error estimates. Our analysis is based on a high-temperature representation of the lattice Higgs measure combined with Poisson approximation. We also give a short proof of the folklore result that Wilson line (and loop) expectations exhibit perimeter law decay whenever the Higgs field coupling constant is positive.
\end{abstract}

\section{Introduction}

\subsection{Background}
One of the major open problems of mathematical physics is that of defining in a mathematically rigorous way quantum Yang-Mills theories in four dimensions. Let $G$ be a compact Lie group and let $A: \mathbb{R}^4 \ni x \mapsto (A_1(x), \dots, A_4(x))$ be a smooth map where each $A_j$ takes values in $\mathfrak{g}$, the Lie algebra of $G$.
The Yang-Mills action on $\mathbb{R}^4$ is defined by
\[
S_{YM}(A) = - \frac{1}{4} \int \sum_{i,j} \textrm{tr}(F_{ij}(x)^2)dx.
\]
Here $F_{ij} = \partial_i A_j - \partial_j A_i + [A_i,A_j]$ is the curvature 2-form (field strength) of the connection one-form $A$ (gauge potential). The Yang-Mills quantum field theory (QFT) with structure group $G$ can be thought of as being defined via a formal path integral measure $e^{-\beta S_{YM}(A)} \mathcal{D}A$, where $\mathcal{D}A$ is thought of as a uniform measure on the space of $G$-valued connection one-forms. Yang-Mills QFTs with structure group $U(1)$, $SU(2)$, and $SU(3)$ play a key part in the Standard Model of particle physics, yet, except when $G=U(1)$, it is not known how to define such theories rigorously in four dimensions~\cite{c-survey}. 

Motivated by this problem, Wilson~\cite{w1974} constructed a regularized model by discretizing space-time using a lattice. In Wilson's model, the Yang-Mills action is discretized (see below) and used to define a Gibbs measure on discrete gauge field configurations. The resulting lattice model is mathematically well-defined and can be analyzed as a statistical mechanics system. In a formal scaling limit, the continuum Yang-Mills action is recovered, and the discrete model maintains exact gauge symmetry. In fact, prior to this, in 1971, Wegner \cite{w1971} had discovered an Ising-like model which is a special case of the one given in~\cite{w1974}, when choosing $G = \mathbb{Z}_2$. In contrast to Wilson, Wegner's motivation came primarily from a statistical mechanics perspective, namely, from a desire to understand the phase structure of a model exhibiting a \emph{local} (spin-flip) symmetry (see also \cite{e1975}). Wegner identified a class of non-local observables, now known as Wilson loops, which allowed him to detect a phase transition in the Ising lattice gauge theory. More precisely, he observed that there is a high-temperature ``confining'' phase and a low temperature ``non-confining'' phase distinguished by the decay rate of the Wilson loop expectation. These phases were later interpreted by Wilson as a criterion for whether or not ``quark confinement'' occurs. (See \cite{c2021} for recent progress on this problem.) A phase transition is rigorously known to occur in several instances, e.g., in the case $G=U(1)$ as well as in the case $G = \mathbb{Z}_n$, see 
\cite{os1978, frolich-spencer82, seiler1982}, but it is an open problem to determine whether or not it occurs in general. (It is expected that there is no phase transition for, e.g., $G=SU(N)$ when $N \ge 2$, see, e.g., the recent survey \cite{c-survey} for more on this and other open problems in the area.)

The models of main relevance in physics, however, are interacting theories where there are matter fields coupled to the gauge field, and it is desirable to better understand the phase structure of such models. This question is also natural from the point of view of statistical mechanics.
In this paper, we consider Wilson's lattice gauge theory coupled to a complex scalar bosonic matter field known as the Higgs field (see below for the precise definition). We shall only consider the structure groups $G=\mathbb{Z}_n$ for $n \ge 2$. In general, the action of this model includes, in addition to the Wilson action, a Dirichlet energy term that couples the Higgs field to the gauge field, as well as a potential that is strongly localized near the unit circle. We shall work with a simplified model, known as the fixed-length limit of the Higgs model, in which the Higgs field is completely localized to the unit circle, see, e.g.,~\cite{seiler1982, s1988}. (In fact, all mathematical treatments we are aware of consider this version of the model.) Because the Higgs model has two fields, there are additional interesting gauge invariant observables besides Wilson loops; in particular Wilson lines can be included in the analysis. While little is known rigorously, the phase structure of this model has been investigated quite thoroughly in the physics literature, see for instance~\cite{fs1979, jsj1980, ks1984}. 

Starting with Chatterjee's paper~\cite{c2019}, several authors have recently rigorously analyzed the behavior of Wilson loops in pure gauge theory (i.e., only the gauge field is considered)~\cite{sc2019,flv2020} as well as Wilson loops~\cite{flv2021,a2021} and Wilson lines~\cite{f2021b} in the lattice Higgs model at very low temperature for various choices of finite gauge groups. These papers all consider a particular limit as the size of the support of the observable grows with the inverse temperature. The temperature is tuned in such a way that the limiting expectation is non-trivial.  The paper \cite{fv2023} takes a different approach and uses the cluster expansion which allows analysis of the pure Ising gauge theory at low but fixed temperature. See also \cite{gs2021} for recent progress on the $U(1)$ pure gauge theory in the low-temperature setting using yet another approach. In this paper, we continue the analysis of the lattice Higgs model started in \cite{flv2021}. Instead of studying the model at very low temperature we shall focus instead on the high-temperature regime with positive Higgs coupling. Our main result is a description of the asymptotic behavior of Wilson loop and Wilson line expectations in the limit as the inverse temperature $\beta \to 0$, with error term estimates. The starting point is a high-temperature representation combined with a detailed analysis of the leading order contributions for small $\beta$. See below for an overview of the proof strategy. Along the way, we obtain a very short proof of the folklore statement that Wilson loops exhibit perimeter law decay at all temperatures as long as the Higgs coupling constant is non-zero.

In order to describe the model and state our results, we need to introduce some notation.

\subsection{Preliminary notation}
For \( m \geq 2 \), consider the graph naturally associated to $\mathbb{Z}^m$ which has a vertex at each point \( x \in \mathbb{Z}^m \) with integer coordinates, and a non-oriented edge between nearest neighbors. We will work with oriented edges throughout this paper, and for this reason we associate to each non-oriented edge \( \bar e \)  two oriented edges \( e_1 \) and \( e_2 = -e_1 \) with the same endpoints as \( \bar e \) but with opposite orientations. 

Let \( {e}_1 ,\) \( {e}_2 , \) \dots, \( {e}_m  \) be oriented edges corresponding to the unit vectors in \( \mathbb{Z}^m \). We say that an oriented edge \( e \) is \emph{positively oriented} if it is equal to a translation of one of these unit vectors, i.e.,\ if there is a \( v \in \mathbb{Z}^m \) and a \( j \in \{ 1,2, \dots, m\} \) such that \( e = v + {{e}}_j \). 
If \( v \in \mathbb{Z}^m \) and \( j_1 <   j_2 \), then \( p = (v +  {e}_{j_1}) \land  (v+ {e}_{j_2}) \) is a positively oriented 2-cell, also known as a  \emph{positively oriented plaquette}. We let \( C_0(\mathbb{Z}^4) \), \( C_1(\mathbb{Z}^4)\), and \( C_2(\mathbb{Z}^4) \) denote the sets of oriented vertices, edges, and plaquettes.
Next, we let \( B_N \) denote the set \(   [-N,N]^m \cap \mathbb{Z}^m \), and we let \( C_0(B_N) \), \( C_1(B_N)\), and \( C_2(B_N) \) denote the sets of oriented vertices, edges, and plaquettes, respectively, whose end-points are all in \( B_N \).

Whenever we talk about a lattice gauge theory we do so with respect to some (abelian) group \( (G,+)  \), referred to as the \emph{structure group}. We also fix a unitary and faithful representation \( \rho \) of \( (G,+) \). In this paper, we will always assume that \( G = \mathbb{Z}_n \) for some \( n \geq 2 \) with the group operation \( + \) given by standard addition modulo \( n \). Also, we will assume that \( \rho \) is a one-dimensional representation of \( G \). We note that a natural such representation is given by \( j\mapsto e^{j \cdot 2 \pi i/n} \).

Now assume that a structure group \( (G,+) \), a one-dimensional unitary representation \( \rho \) of \( (G,+) \), and an integer \( N\geq 1 \) are given.
We let \( \Omega^1(B_N,G) \) denote the set of all  \( G \)-valued  1-forms \( \sigma \) on \( C_1(B_N) \), i.e., the set of all \( G \)-valued functions \(\sigma \colon  e \mapsto \sigma(e) \) on \( C_1(B_N) \) such that \( \sigma(e) =  -\sigma(-e) \) for all \( e \in C_1(B_N) \). 
%
Similarly, we let \( \Omega^0(B_N,G) \) denote the set of all \( G\)-valued functions \( \phi \colon x \mapsto \phi(x)\) on \( C_0(B_N) \) such that \( \phi(x) = - \phi(-x) \) for all \( x \in C_0(B_N). \) 
When \( \sigma \in \Omega^1(B_N,G) \) and \( p \in C_2(B_N) \), we let \( \partial p \) denote the formal sum of the four edges \( e_1,\) \( e_2,\) \( e_3,\) and \( e_4 \) in the oriented boundary of \( p \) (see Section~\ref{sec: cell boundary}), and define
\begin{equation*}
    d\sigma(p) \coloneqq \sigma(\partial p) \coloneqq \sum_{e \in \partial p} \sigma(e) \coloneqq \sigma(e_1) + \sigma(e_2) + \sigma(e_3) + \sigma(e_4).
\end{equation*} 
Similarly, when \( \phi \in \Omega^0(B_N,G) \) and \( e \in C_1(B_N) \) is an edge from \( x_1 \) to \( x_2 \), we let \( \partial e \) denote the formal sum \( x_2-x_1, \) and define \( d\phi(e) \coloneqq \phi(\partial e) \coloneqq \phi(x_2) - \phi(x_1). \)

\subsection{The fixed length lattice Higgs model}\label{sec: measure in intro}
The Wilson action for pure gauge theory is given as follows
\[
S_{N}^W(\sigma) =  - \sum_{p \in C_2(B_N)}  \rho\bigl( d\sigma(p)\bigr), \qquad \sigma \in \Omega^1(B_N,G).
\]
For each plaquette \( p \in C_2(B_N), \) we have    \( \rho\bigl(d\sigma(-p) \bigr) =  \overline{\rho\bigl(d\sigma(p) \bigr)},\) so \( S_N^W(\sigma) \) is real.
With $(\sigma, \phi) \in \Omega^1(B_N,G) \times  \Omega^0(B_N,G)$, we also define a coupling term
\[
S_N^H(\sigma, \phi) = -\sum_{\substack{e\in C_1(B_N)\mathrlap{\colon}\\ \partial e = y-x}} 
        \rho\bigl(\sigma(e)-\phi(\partial e)\bigr).
\]
For each edge \( e \in C_1(B_N), \) we have    \( \rho\bigl(\sigma(-e)-\phi(\partial (-e))\bigr) =  \overline{\rho\bigl(\sigma(e)-\phi(\partial e)\bigr)},\) so \( S_N^H(\sigma, \phi) \) is real.
The action for the fixed length Higgs model is then defined by
\[
 S_{N,\beta,\kappa}(\sigma, \phi) = \beta S_{N}^W(\sigma)  + \kappa S_N^H(\sigma, \phi).
\]
Elements \( \sigma \in \Omega^1(B_N,G) \) are referred to as \emph{gauge field configurations} and functions \(\phi  \in \Omega^0(B_N,G) \) are referred to as \emph{Higgs field configurations}.
The quantity \( \beta \) is the \emph{inverse temperature}; it is related to the gauge theory \emph{coupling constant} $g$ by the relation $\beta =1/2g^2$. In particular, \emph{strong coupling} corresponds to small values of $\beta$ (and to high temperature). 
The parameter~\( \kappa \) is known as the \emph{hopping parameter} or as the \emph{Higgs field coupling constant}.

We next consider the probability measure \(\mu_{N,\beta, \kappa, \infty}\) on \(\Omega^1(B_N,G) \times \Omega^0(B_N,G)\) given by
\begin{equation}\label{eq: London limit measure}
    \mu_{N,\beta, \kappa, \infty}(\sigma, \phi)  \coloneqq
    Z_{N,\beta,\kappa, \infty}^{-1} e^{-S_{N,\beta,\kappa}(\sigma, \phi)} , \qquad \sigma \in \Omega^1(B_N,G) ,\, \phi \in \Omega^0(B_N,G),
\end{equation}
where \( Z_{N,\beta,\kappa, \infty}\) is a normalizing constant. This is the \emph{fixed length lattice Higgs model}. (We include the subscript 
$\infty$ in the notation for $\mu_{N,\beta, \kappa, \infty}$ in order to keep our notation consistent with that of~\cite{flv2021}.) We let \( \mathbb{E}_{N,\beta,\kappa,\infty} \) denote the corresponding expectation.

Whenever \( f \colon \Omega^1(B_M,G) \times \Omega^0(B_M,G) \to \mathbb{R}\) for some \( M \geq 1,\) then, as a consequence of the Ginibre inequalities (see Section~\ref{sec: ginibre}), the following limit exists
\begin{equation}\label{infinitevolumelimit}
    \bigl\langle f(\sigma,\phi) \bigr\rangle_{\beta,\kappa,\infty} \coloneqq \lim_{N \to \infty} \mathbb{E}_{N,\beta,\kappa,\infty} \bigl[f(\sigma,\phi) \bigr]
\end{equation}
and is translation invariant. Note that we use free boundary conditions here. It is this limit that we will consider in  Theorem~\ref{theorem: first theorem Z2}. 

\subsection{Wilson loops and Wilson lines}
 
For \( k \in \{ 0,1,\dots, m\},\) a \( k \)-chain is a formal sum of positively oriented $k$-cells with integer coefficients, see Section~\ref{sec: chains} below. The support of a \(k\)-chain \(c\), written \(\support c\), is the set of positively oriented \(k \)-cells with non-zero coefficient in \( c.\) 

We say that a \(1\)-chain $\gamma$ with finite support is a \emph{closed path} if it has coefficients in \(\{-1,0,1\},\) connected support, and empty boundary (see Section~\ref{sec: cell boundary} for the precise definition of the boundary $\partial \gamma$  of $\gamma$). 
We say that a \(1\)-chain with finite connected support is an \emph{open path} from \( x_1\in C_0^+(B_N) \) to \( x_2 \in C_0^+(B_N)\) if it has coefficients in \(\{-1,0,1\}\) and boundary \( \partial \gamma \coloneqq x_2 - x_1. \) 
%
If \( \gamma \) is either an open path or a closed path, we refer to \( \gamma \) as a \emph{path}. 

Let \( R \) be an axis-parallel rectangle with corners in \( \mathbb{Z}^m.\) If all of the edges in the support of a path \( \gamma \) lie in the boundary of \( R, \) then we say that \( \gamma \) is a path along the boundary of \( R , \) or equivalently, that \( \gamma\) is a rectangular path (see Section~\ref{sec: rectangular}).

Given a path \( \gamma \), the \emph{Wilson line observable} \( L_\gamma(\sigma,\phi) \) is defined by 
\begin{equation*}
    L_\gamma(\sigma,\phi)
    \coloneqq  \rho \bigl( \sigma(\gamma) - \phi(\partial \gamma) \bigr),\qquad \sigma\in \Omega^1(B_N,G),\, \phi\in \Omega^0(B_N,G),
\end{equation*}
where \( \sigma(\gamma) \coloneqq \sum_{e \in  \gamma} \sigma(e), \) and \( \phi(\partial \gamma) = \phi(x_2) - \phi(x_1) \) if \( \gamma \) is an open path from \( x_1 \) to \( x_2 \), and \(\phi(\partial \gamma) = 0 \) if the boundary of \( \gamma \) is empty. If \( \gamma \) is a closed path, then \(  L_\gamma(\sigma,\phi)\) is referred to as a \emph{Wilson loop observable}. Both these observables are gauge invariant, see below. We let \( |\gamma| \coloneqq |\support \gamma|. \)

\subsection{Functions and parameters}\label{subsec: functions}
The following will appear in the statements of our main results. 
For \( j \in [n] \coloneqq \{ 0,1,\dots, n-1\} \) and \( a \geq 0, \)  define
\begin{equation*}
    \psi_a(j) \coloneqq \sum_{k=0}^\infty \frac{a^{j+kn} }{(j+kn)!} ,
\end{equation*} 
\begin{equation*}
    \hat \varphi_a(j)
    \coloneqq 
    \sum_{\substack{k,k' \in [n] \mathrlap{\colon}\\ k'-k =  j \mod n}}\psi_{a}(k)\psi_a(k')
    =
    \sum_{k = 0}^{n-1-j} \psi_{a}(k)\psi_a(k+j)
    +
    \sum_{k = n-j}^{n-1} \psi_{a}(k)\psi_a(k+j-n),
\end{equation*} 
and
\begin{equation*}\label{varphidef} 
    \varphi_a(j) \coloneqq \hat \varphi_a(j) / \hat \varphi_a(0).
\end{equation*} 
By identifying \(  \mathbb{Z}_n \) with $[n]$ in the natural way, we also view $\psi_a$, $\hat{\varphi}_a$, and $\varphi_a$ as functions on $\mathbb{Z}_n$.
To simplify notation, we extend the definition of  \( \hat \varphi_a \) to \( \mathbb{Z} \) by letting \( \hat \varphi_a(j) \coloneqq \hat \varphi_a(j') \) whenever \( j' \in [n] \) and \( j = j' + kn \) for some \( k \in \mathbb{Z}. \)
We also set
\begin{equation}\label{eq: def eta kappa1}
    \eta_a \coloneqq \min_{j \in [n]} \varphi_a( j+1) /\varphi_a( j) = \min_{j \in [n]} \varphi_a( j\pm 1) /\varphi_a( j),
\end{equation} 
where the last identity follows since, by symmetry, we have \( \varphi_a(j) = \varphi_a(-j) \) for all \( j \in [n] \). 

\subsection{Main results}
Our first result is certainly known in the physics community, but we are not aware of a clean and complete proof in the literature. 
\begin{theorem}[Perimeter law for $\kappa>0$]\label{theorem: perimeter law}
    Let \( m \geq 2, \) let \( G = \mathbb{Z}_n, \) suppose \( \beta,\kappa \geq 0, \) and let \( \gamma \) be a path. Then
    \begin{equation}\label{eq: perimeter law}
        \mathbb{E}_{N,\beta,\kappa,\infty} \bigl[L_\gamma(\sigma,\phi)\bigr] \geq \eta_\kappa^{| \gamma|},
    \end{equation}
    where \( \eta_\kappa \) is defined in~\eqref{eq: def eta kappa1}. In particular, if \( n = 2, \) then 
    \begin{equation}\label{eq: perimeter law 2}
        \mathbb{E}_{N,\beta,\kappa,\infty} \bigl[L_\gamma(\sigma,\phi)\bigr] \geq (\tanh 2\kappa)^{| \gamma|}.
    \end{equation}
\end{theorem}

\begin{remark}
    Using Proposition~10.20 in~\cite{f2021b}, we also have the \emph{a priori} upper bound
    \begin{equation*}
        \mathbb{E}_{N,\beta,\kappa,\infty} \bigl[L_\gamma(\sigma,\phi)\bigr] \leq e^{-2 (| \gamma|-4)e^{-4(2\kappa + 12\beta)}}
    \end{equation*}
    which is valid for all \( \beta,\kappa \geq 0.\) We note that there is a substantial gap between this upper bound and the lower bound \( \tanh(2\kappa)^{| \gamma|} = e^{-| \gamma| \ln (\tanh (2\kappa)^{-1})}\) obtained in Theorem \ref{theorem: perimeter law}. Indeed, for $\beta \ge 0$ fixed, the coefficient of $-| \gamma|$ in the exponent of the upper bound is bounded as $\kappa \to 0+$, whereas the corresponding coefficient in the lower bound diverges. Moreover, since the lower bound is independent of $\beta$ it does not provide any information about the phase transition in the $\kappa=0$ pure gauge theory.
\end{remark}

When \( \beta = 0, \) one can show that, after a gauge transform, the measure \( \mu_{N,\beta,\kappa,\infty} \) is  equivalent to a product measure on \( C^1(B_N,G)^+  , \) where for each \( g \in G \) and \( e \in C_1(B_N)^+, \)  \( \sigma(e) = g \) with probability \( e^{2\kappa\Re \rho(g)}/\sum_{g \in G} e^{2\kappa\Re \rho(g)}. \) Consequently, if we let
\begin{equation}\label{hatetadef}
    \hat \eta_\kappa \coloneqq \sum_{g \in G}\rho(g) e^{2\kappa\Re \rho(g)}/\sum_{g \in G} e^{2\kappa \Re \rho(g)},
\end{equation}
then, for any fixed $N < \infty$, $\gamma$, and $\kappa$, 
\begin{equation}\label{eq: trivial limit}
        \lim_{\beta \to 0}\mathbb{E}_{N,\beta,\kappa,\infty} \bigl[L_\gamma(\sigma)\bigr]  = \mathbb{E}_{N,0,\kappa,\infty} \bigl[L_\gamma(\sigma)\bigr] = 
        \prod_{e \in \gamma} \hat \eta_\kappa
        = \hat \eta_\kappa^{| \gamma|}.
    \end{equation}
    However, it is neither clear what the rate of convergence is, what the sub-leading term of \( \mathbb{E}_{N,\beta,\kappa,\infty}  \bigl[L_\gamma(\sigma)\bigr]\) is (as a function of \( \beta \)), nor whether the convergence is uniform in \( N, \) \( \gamma, \) or \( \kappa. \) 
    Moreover, if \( n = 2, \) then \( \hat \eta_\kappa = \eta_\kappa = \tanh(2\kappa) \) for all \( \kappa \geq 0, \) but if \( n \geq 3, \)  then \( \hat \eta_\kappa  > \eta_\kappa. \) This means that for \( n \geq 3 \) the lower bound in Theorem~\ref{theorem: perimeter law} is not the limit as \( \beta \to 0. \)

    Our main results, Theorem~\ref{theorem: first theorem Z2} and the more general Theorem~\ref{theorem: first theorem}, add to~\eqref{eq: perimeter law} and~\eqref{eq: trivial limit} in several ways. First, they give estimates on the rate of convergence in~\eqref{eq: trivial limit}, and in fact, our upper bound is uniform in \( N.\) Second, they show that~\eqref{eq: trivial limit} holds also in the limit \( N \to \infty. \) Third, they improve~\eqref{eq: perimeter law 2} whenever the error term is smaller than the corresponding estimate.

We now give our main result. For simplicity, we state it here only for \( G = \mathbb{Z}_2,\) and include our more general version, which is valid for $G = \mathbb{Z}_n$ for any $n \geq 2$, as Theorem~\ref{theorem: first theorem} in Section~\ref{sec: main result proof}. To simplify the notation in the statement, we let 
\begin{equation}\label{eq: Pgamma}
	\mathcal{P}_\gamma \coloneqq \bigcup_{e \in \gamma} \bigl\{ p \in C_2(B_N) \colon e \in \partial p \bigr\}
\end{equation}   denote the set of all oriented plaquettes bordering the path $\gamma$ whose orientation is consistent with that of $\gamma$.

\begin{theorem}[Small $\beta$ asymptotics]\label{theorem: first theorem Z2}
    Let $m \geq 2$, \( G = \mathbb{Z}_2, \) \( \beta, \kappa \geq 0 \), and let  \( \gamma \) be a path along the boundary of a rectangle with side lengths \( \ell_1,\ell_2 \geq 7. \) Suppose further that \[ (16m)^2 \tanh{2\beta} < \tanh(2\kappa) \qquad \textrm{and} \qquad  \kappa + \kappa( 1
            +
            2\kappa e^\kappa  ) \bigl( 1
            +
            \frac{\kappa^2 e^\kappa}{2} \bigr)^2 \leq 1.\]
    Then
    \begin{equation}\label{eq: first theorem Z2}
        \begin{split}
            &\Bigl| \langle L_\gamma \rangle_{\beta,\kappa,\infty} - \tanh(2\kappa)^{| \gamma|} \alpha(\beta,\kappa)^{|\mathcal{P}_\gamma|} \Bigr| 
            \leq   
            C^{(0)}_{\gamma,\beta,\kappa,m}\, \tanh(2\kappa)^{| \gamma|} \alpha(\beta,\kappa)^{{|\mathcal{P}_\gamma|}}  \tanh(2\beta),
        \end{split}
    \end{equation} 
    where $\langle L_\gamma \rangle_{\beta,\kappa, \infty}$ is the  limit defined in~\eqref{infinitevolumelimit}, 
    \[ \alpha(\beta,\kappa) \coloneqq \frac{1 + \tanh(2\beta)\tanh(2\kappa)^2}{1 + \tanh(2\beta)\tanh(2\kappa)^4},\]
     and \( C^{(0)}_{\gamma,\beta,\kappa,m} \) is defined in~\eqref{eq: def last constant}.
\end{theorem}

\begin{remark}
    One easily verifies that the assumptions of Theorem~\ref{theorem: first theorem Z2} hold whenever \( {\kappa \leq 0.318 }\) and \( \beta \leq \tanh(\kappa)/(16m)^2\), see also Figure~\ref{figure: parameters}.
    \begin{figure}[tp]
        \centering
        \includegraphics[width=.5\textwidth]{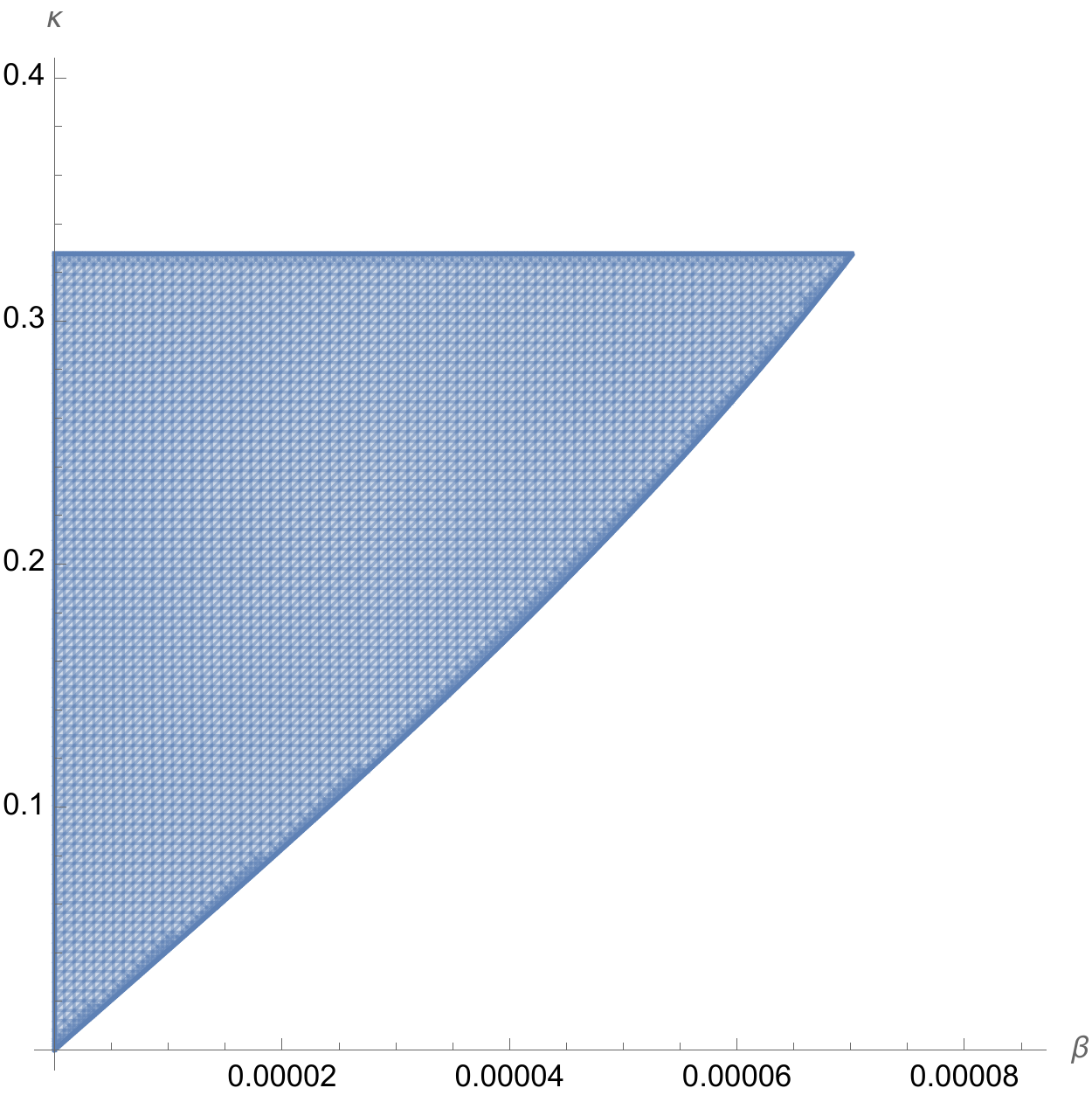}
        \caption{The shaded area represents the set of pairs \( (\beta,\kappa)\) for which the assumptions of Theorem~\ref{theorem: first theorem Z2} hold when \( m = 4.\)}
        \label{figure: parameters}
    \end{figure}
\end{remark}

\begin{remark}
    Theorem~\ref{theorem: first theorem Z2} implies that, for fixed $\gamma$ and $\kappa > 0$, as $\beta \to 0$,
    \[
    \langle L_\gamma \rangle_{\beta,\kappa,\infty} = (\tanh 2\kappa)^{| \gamma|}\left(1 + O(\beta)\right).
    \]
\end{remark}

\begin{remark}
    If \( \kappa\) is fixed, Theorem~\ref{theorem: first theorem Z2} implies that  \[ C_{\gamma,\beta,\kappa,m} = O(1) \pigl( 1 + O\bigl( \tanh 2\beta  | \gamma|\bigr) \pigr)^2 ,\] and hence the right-hand side of~\eqref{eq: first theorem Z2} is of smaller order than \( (\tanh 2\kappa)^{| \gamma|} \alpha(\beta,\kappa)^{|\mathcal{P}_\gamma|} \) when \( | \gamma|(\tanh 2\beta)^2  \) is small. Therefore, for fixed \( \gamma \) and \(\kappa,\) one can show that as \( \beta \to 0 , \) we have \[ C_{\gamma,\beta,\kappa,m} = (\tanh 2\kappa) \, |\mathcal{P}_{\gamma,c}|  \bigl( (8m)^{2}  (\tanh 2\kappa)^{3} +1\bigr) + O(\beta)\] and 
    \begin{equation*}
       ( \tanh 2\kappa)^{| \gamma|} \alpha(\beta,\kappa)^{|\mathcal{P}_\gamma|} \sim (\tanh 2\kappa)^{| \gamma|} e^{(\tanh 2\beta) (\tanh 2\kappa)^2(1 - (\tanh 2\kappa)^2 )} .
    \end{equation*}
    Hence when \( \beta>0\) we get a non-trivial multiplicative correction. 
\end{remark}

\begin{remark}
    In contrast to recent results for very large \( \beta \) (see, e.g.,~\cite{c2019, sc2019, a2021, ac2022, f2021, f2021b, flv2020, flv2021}), whose proofs only work when the dimension $m$ of the lattice $\mathbb{Z}^m$ satisfies \( m \geq 3, \) Theorem~\ref{theorem: first theorem Z2} holds for any \( m \geq 2. \)
    To understand why this is the
case, note that in an \( m \)-dimensional lattice with \( m \geq 3 \) there are non-empty sets of plaquettes \( P \) with \( \partial P = \emptyset, \) whereas no such sets of plaquettes exist in a two-dimensional lattice.  
For large \( \beta \), this geometric difference has major implications.
However, when \( \beta \) is very small, we do not see this difference because any set of plaquettes with \( \partial P = \emptyset \) contains at least six plaquettes, and is therefore given a much smaller weight than any set \( P = \{ p \} \) with \( p \in C_2(B_N)^+. \)
\end{remark}

\subsection{Outline of the proof and comparison with other work}
Let us briefly summarize the main steps in the proof of Theorem~\ref{theorem: first theorem Z2}.

The first step is to pass to what is known as the unitary gauge. 
In this (standard) step, which is implemented in Section \ref{sec: unitary gauge}, a gauge transformation is used to shift the value of the Higgs field $\phi$ to the identity element of the abelian group $G$. 
This allows us to write the expectation value of a gauge invariant observable under $\mu_{N,\beta,\kappa,\infty}$ as the expectation value of the same observable evaluated at $\phi \equiv 0$ under another measure, which we denote by $\mu_{N,\beta,\kappa}$. Whereas $\mu_{N,\beta,\kappa,\infty}$ is a measure on $\Omega^1(B_N,G) \times \Omega^0(B_N,G)$, $\mu_{N,\beta,\kappa}$ is a measure on $\Omega^1(B_N,G)$.
Since the Wilson loop and Wilson line observables are both gauge invariant, this leads to the following formula (see Corollary~\ref{corollary: unitary gauge}):
\begin{equation}\label{firststep}
        \mathbb{E}_{N,\beta,\kappa,\infty}\bigl[L_\gamma(\sigma,\phi)\bigr] =
        \mathbb{E}_{N,\beta,\kappa}\bigl[L_\gamma(\sigma,0)\bigr].
    \end{equation}

The second step is to pass to a high-temperature representation. It allows us to study an equivalent quantity to the Wilson line in a low-temperature model, which is less noisy and therefore easier to understand. This step shows that the expectation value $\mathbb{E}_{N,\beta,\kappa}[L_\gamma(\sigma,0)]$ can be expressed as an expectation value of a closely related observable, $\widehat{L_\gamma}(\omega)$, under a probability measure on the space of 2-forms $\Omega^2(B_N,\mathbb{Z}_n)$. Denoting the expectation with respect to the latter measure by $\mathbb{E}_\varphi$, the outcome of the second step is the identity (see Proposition~\ref{proposition: high-temperature expansion ALHM 3})
\begin{align}\label{secondstep}
\mathbb{E}_{N,\beta,\kappa}\bigl[ L_\gamma(\sigma,0)\bigr]
= \mathbb{E}_\varphi \bigl[ \widehat{L_\gamma}(\omega) \bigr].
\end{align}
While the first two steps are exact, the third step introduces an approximation. The basic observation is that if $(\beta, \kappa)$ belongs to the parameter range specified in the statement of Theorem~\ref{theorem: first theorem Z2} (see Figure \ref{figure: parameters}), then only 2-forms $\omega \in \Omega^2(B_N,\mathbb{Z}_n)$ that lie in a certain subset $\mathcal{E}$ of $\Omega^2(B_N,\mathbb{Z}_n)$ make a substantial contribution to the expectation value $\mathbb{E}_\varphi[\widehat{L_\gamma}(\omega)]$. 
More precisely, we will show that
\begin{align}\label{thirdstep}
\mathbb{E}_\varphi    \pigl[\widehat{L_\gamma}(\omega) \cdot \mathbb{1}( \omega \notin \mathcal{E} ) \pigr] 
            \leq     C^{(1)}_{\gamma,\beta,\kappa,m}\, \tanh(2\kappa)^{| \gamma|}    \tanh(2\beta),
\end{align}   
where $C^{(1)}_{\gamma,\beta,\kappa,m}$ can be explicitly computed, see Proposition~\ref{proposition: useful upper bound forms ii}.
In other words, with a quantified error, the 2-forms which do not lie in $\mathcal{E}$ can be ignored. 
Roughly speaking, the subset $\mathcal{E} \subset \Omega^2(B_N,\mathbb{Z}_n)$ is defined to consist of all 2-forms $\omega$ such that each connected component of the support of $\omega$ that borders $\gamma$ consists of exactly one plaquette $p$ (or two plaquettes $\pm p$ if orientation is counted); a more technical condition relevant at the corners of the rectangular path $\gamma$ is also required, see \eqref{eq: def E2 forms ii} for the exact definition.

The fourth and final step of the proof is to use a Poisson approximation to show that the contribution from configurations $\omega \in \Omega^2(B_N,\mathbb{Z}_n)$ in the set $\mathcal{E}$ can be approximated by $\tanh(2\kappa)^{| \gamma|} \alpha(\beta,\kappa)^{|\mathcal{P}_\gamma|}$. This leads to the following estimate established in Proposition~\ref{proposition: last resampling lemma forms ii}:
\begin{align*}
	&\bigl| \mathbb{E}_\varphi    \pigl[\widehat{L_\gamma}(\omega)\cdot \mathbb{1}(  \mathcal{E} ) \pigr] - \tanh(2\kappa)^{| \gamma|} \alpha(\beta,\kappa)^{|\mathcal{P}_\gamma|}\bigr|
            \leq
     C^{(2)}_{\gamma,\beta,\kappa,m} \, \tanh(2\kappa)^{| \gamma|}
     \alpha(\beta,\kappa)^{|\mathcal{P}_\gamma |} \tanh(2\beta).
\end{align*}
Theorem \ref{theorem: first theorem Z2} is obtained by combining this estimate with~(\ref{firststep})--(\ref{thirdstep}).

Finally, let us comment briefly on how the work here differs from that in the recent papers~\cite{a2021, sc2019, c2019, f2021b, flv2021}.
\begin{enumerate}
    \item In the bulk of the paper, we work with the high-temperature representation of the lattice Higgs model. While the resulting model in a certain sense can be viewed as dual to the low-temperature theories studied in~\cite{a2021, sc2019, c2019, f2021b, flv2021}, it is substantially different. As a result, the analysis is different too, and several of the basic facts need to be established here.  
    
    \item In contrast to the Wilson line observable in the usual lattice Higgs model, which takes values on the unit circle, the high-temperature representation of the Wilson line observable \( \widehat{L_\gamma} \) that we work with here is unbounded. This fact makes the analysis technically more challenging and, for instance, the events that contribute to the leading order behavior need to be determined.
\end{enumerate}

\subsection{A simulation}

As explained in the previous subsection, a key step in the proof of our main result is the introduction of an event $\mathcal{E}$ such that all main contributions to the Wilson loop/line expectation value in the high temperature regime originate from 2-forms in $\mathcal{E}$.
To understand the definition of $\mathcal{E}$ and why 2-forms $\omega$ that do not lie in $\mathcal{E}$ can be ignored to a good approximation when evaluating $\mathbb{E}_\varphi[\widehat{L_\gamma}(\omega)]$, it is useful to consider the simulation, obtained by using Gibbs sampling, presented in Figure \ref{figure: simulations}. Figure~\ref{figure: simulations} displays the result of simulations on a subset of the \( \mathbb{Z}^2 \)-lattice for different choices of the parameters \( \beta \) and \( \kappa, \) and for a $U$-shaped path $\gamma$. The path $\gamma$ consists of the bottom, left, and right edges of a rectangular-shaped path, and a plaquette is drawn black or white depending on whether or not it lies in the support of the $\mathbb{Z}_2$-valued 2-form $\omega$. The expectation value $\mathbb{E}_\varphi[\widehat{L_\gamma}(\omega)]$ is defined by
$$\mathbb{E}_\varphi[\widehat{L_\gamma}(\omega)] = \frac{\sum_{\omega \in \Omega^2(B_N,\mathbb{Z}_n)} \widehat{L_\gamma}(\omega) \varphi(\omega)}{\sum_{\omega \in \Omega^2(B_N,\mathbb{Z}_n)} \varphi(\omega)},$$ 
where $\varphi$ is a function defined in Section \ref{section: high-temperature expansion}. In Figure \ref{figure: simulations}, the 2-form $\omega$ is sampled according to the weight $\widehat{L_\gamma}(\omega) \varphi(\omega)$. In other words, Figure \ref{figure: simulations} shows typical examples of 2-forms that make large contributions to $\mathbb{E}_\varphi[\widehat{L_\gamma}(\omega)]$ for various values of $\beta$ and $\kappa$. The simulations show that for large values of $\beta$, the main contributions stem from 2-forms $\omega$ whose support fill the rectangle spanned by $\gamma$, and that larger values of $\kappa$ are associated with more noise. In the asymptotic regime considered in~Theorem~\ref{theorem: first theorem Z2} (i.e., for small values of $\beta$ and $\kappa$ such that $\kappa \gtrsim \beta$, see Figure \ref{figure: parameters}), the simulations indicate that the main contributions come from 2-forms $\omega$ whose support consists of a small number of plaquettes bordering $\gamma$, together with small scattered islands of plaquettes not bordering $\gamma$ (the number of such islands increases with $\kappa$). This observation motivates our definition of $\mathcal{E}$ as, roughly speaking, the set of 2-forms $\omega$ whose support is a union of isolated plaquettes bordering $\gamma$, together with islands of plaquettes not bordering $\gamma$. 

\begin{figure}[tp]
	\vspace{8ex}
    \renewcommand{\tabcolsep}{1pt}
    \centering
    \begin{tabular}{
    >{\centering\arraybackslash}m{2.8cm} >{\centering\arraybackslash}m{2.8cm} >{\centering\arraybackslash}m{2.8cm} >{\centering\arraybackslash}m{2.8cm} >{\centering\arraybackslash}m{2.8cm}} 
        \includegraphics[width=.2\textwidth, trim=120 40 105 40, clip]{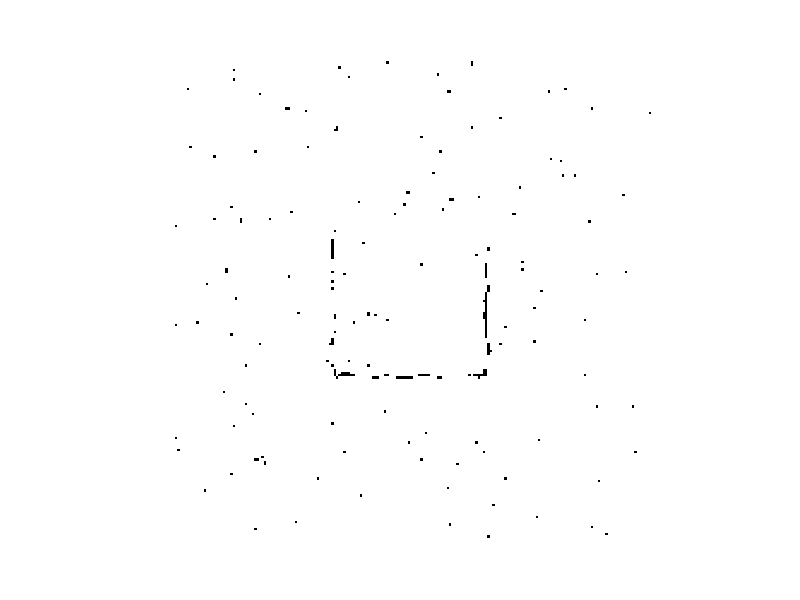}   & \includegraphics[width=.2\textwidth, trim=120 40 105 40, clip]{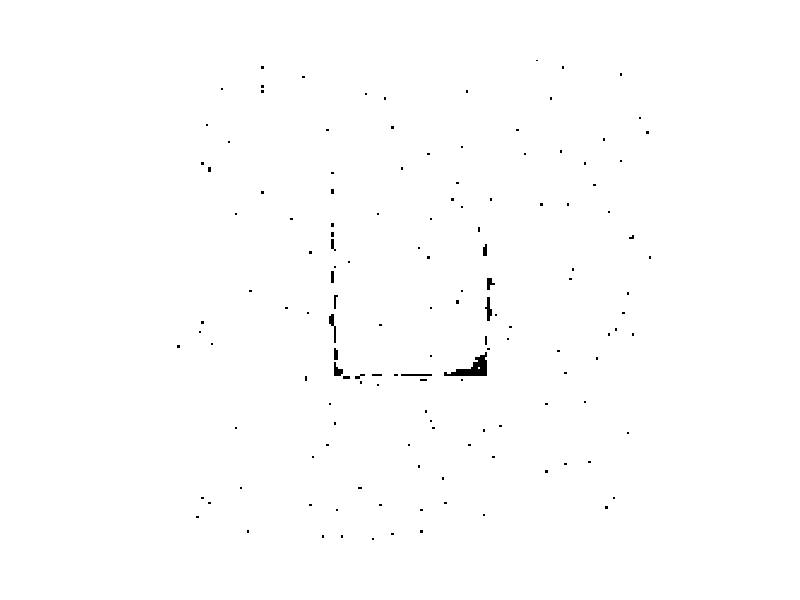}   & \includegraphics[width=.2\textwidth, trim=120 40 105 40, clip]{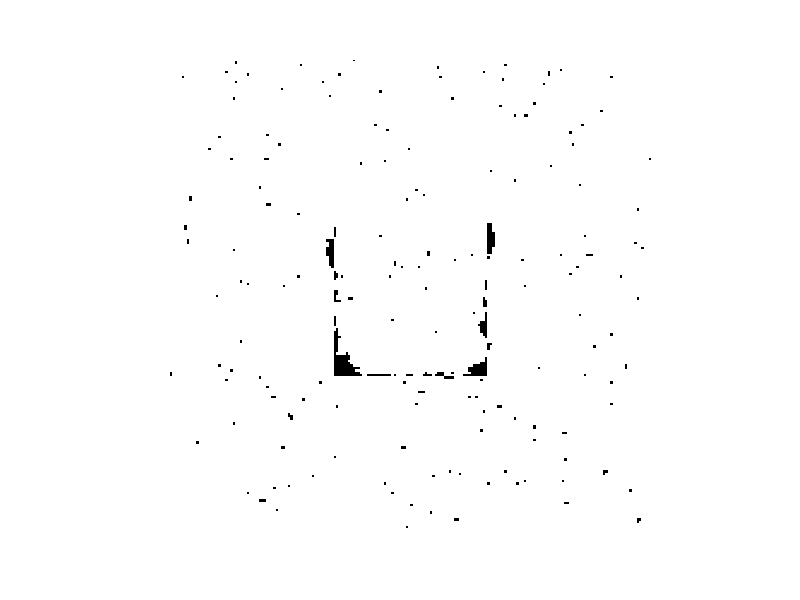}   & \includegraphics[width=.2\textwidth, trim=120 40 105 40, clip]{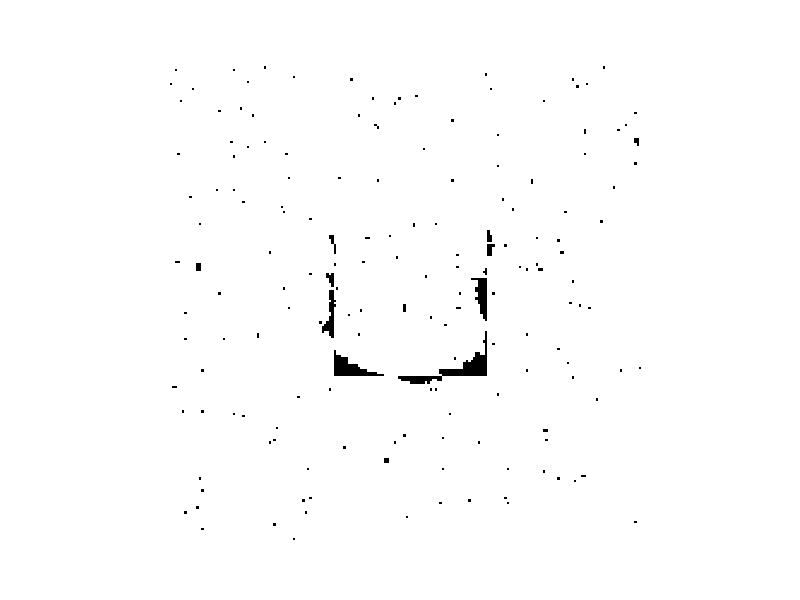}   & \includegraphics[width=.2\textwidth, trim=120 40 105 40, clip]{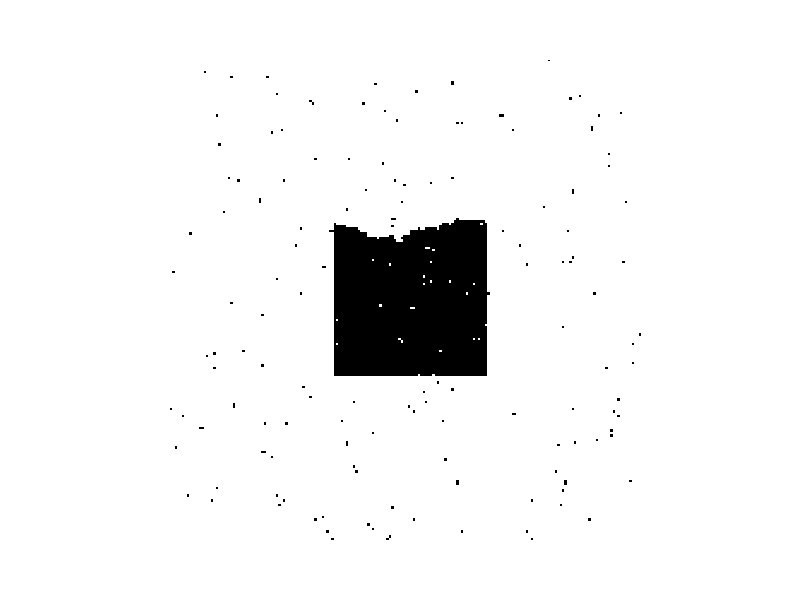} \\[-1ex]
        \includegraphics[width=.2\textwidth, trim=120 40 105 40, clip]{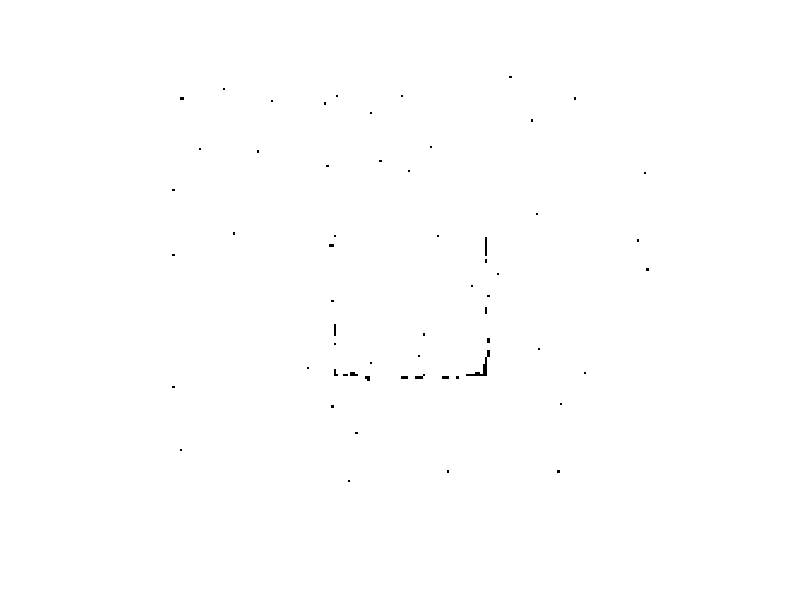}   & \includegraphics[width=.2\textwidth, trim=120 40 105 40, clip]{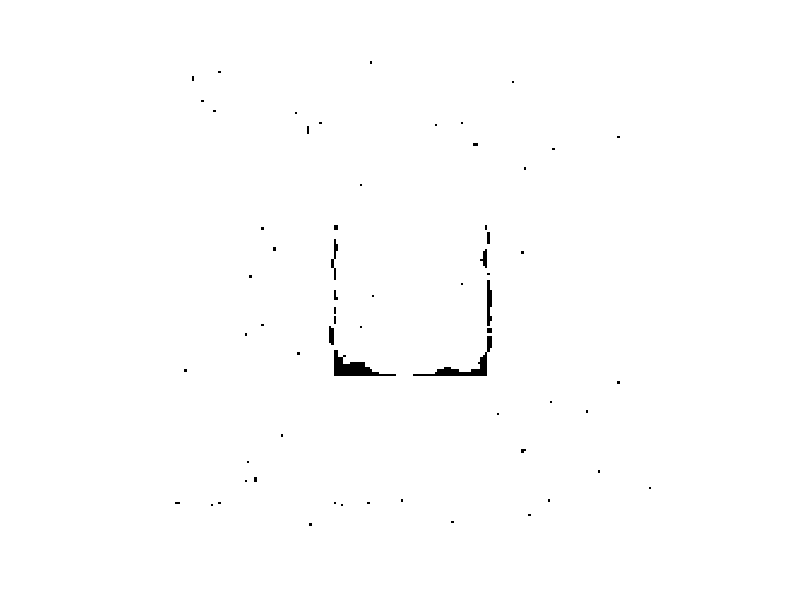}   & \includegraphics[width=.2\textwidth, trim=120 40 105 40, clip]{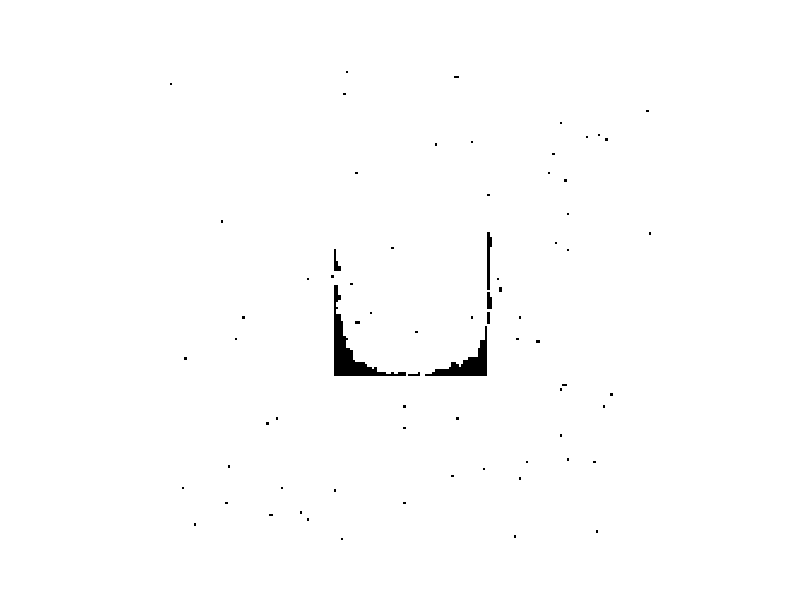}   & \includegraphics[width=.2\textwidth, trim=120 40 105 40, clip]{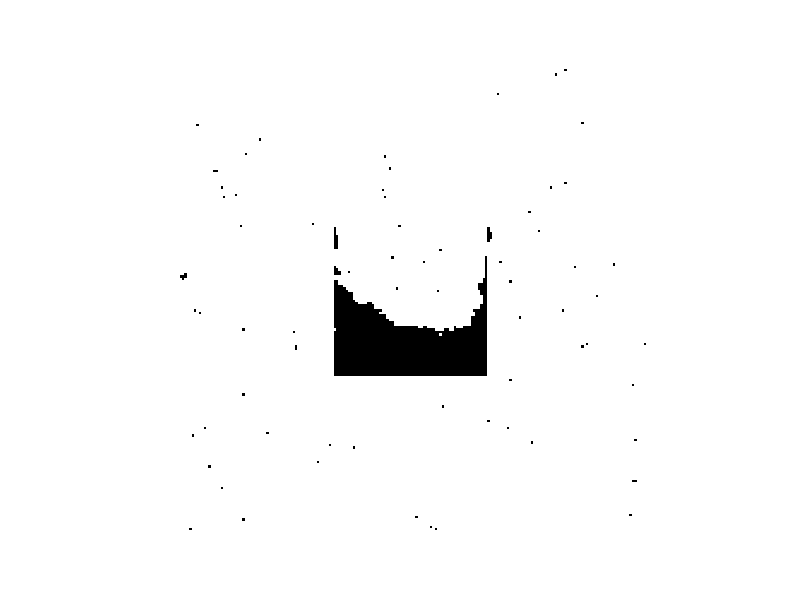}   & \includegraphics[width=.2\textwidth, trim=120 40 105 40, clip]{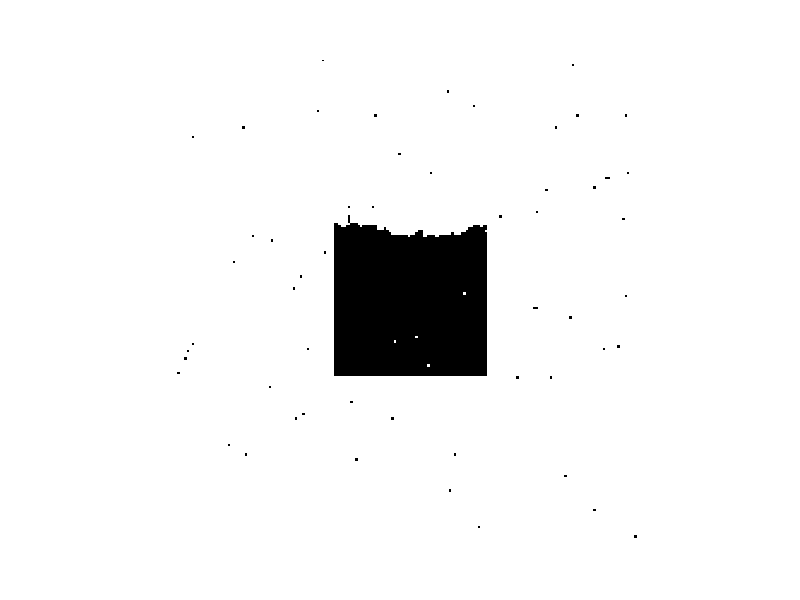}  \\[-1ex]
        \includegraphics[width=.2\textwidth, trim=120 40 105 40, clip]{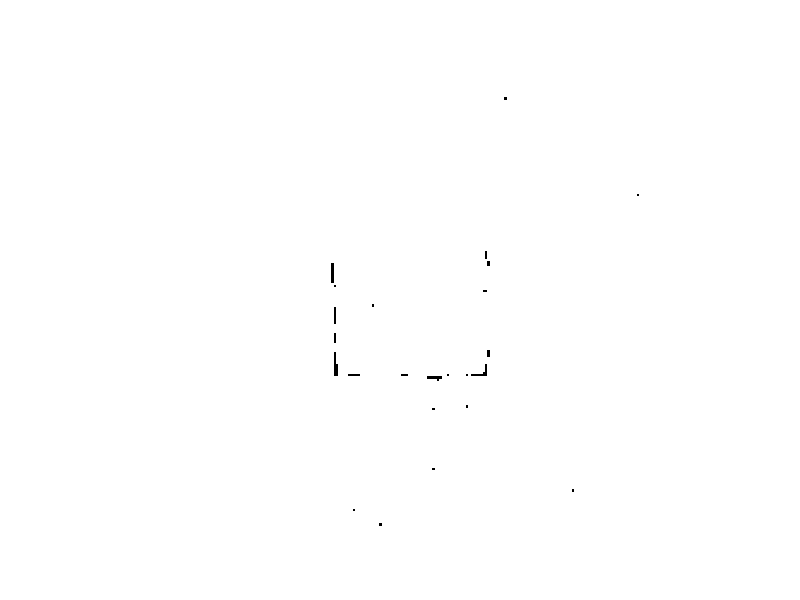}  & \includegraphics[width=.2\textwidth, trim=120 40 105 40, clip]{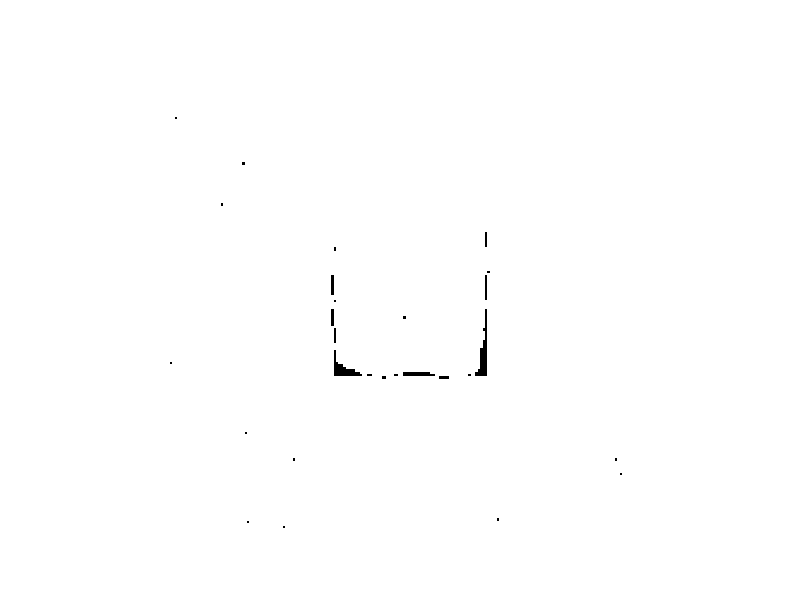}  & \includegraphics[width=.2\textwidth, trim=120 40 105 40, clip]{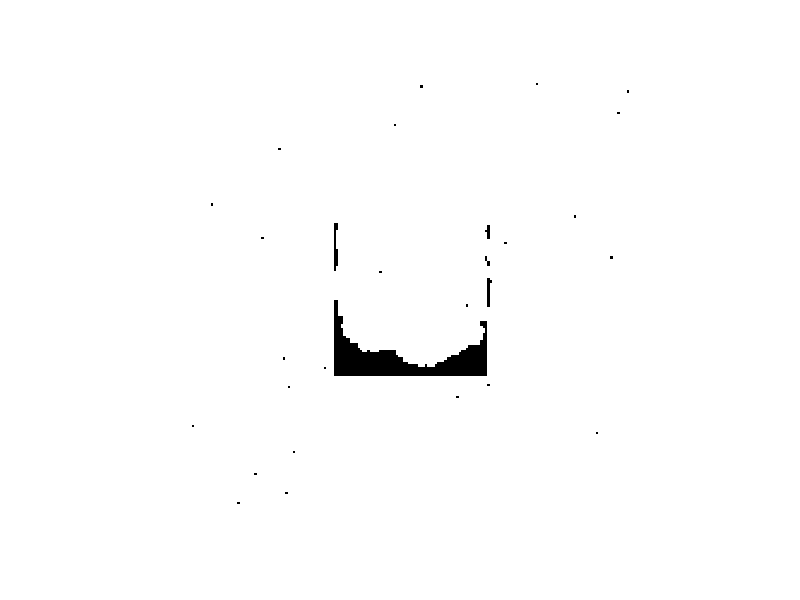}   & \includegraphics[width=.2\textwidth, trim=120 40 105 40, clip]{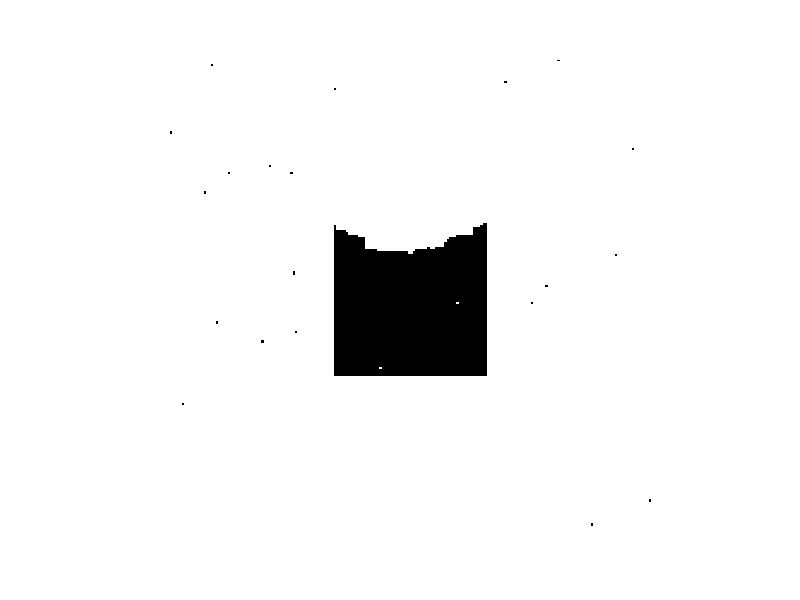}   & \includegraphics[width=.2\textwidth, trim=120 40 105 40, clip]{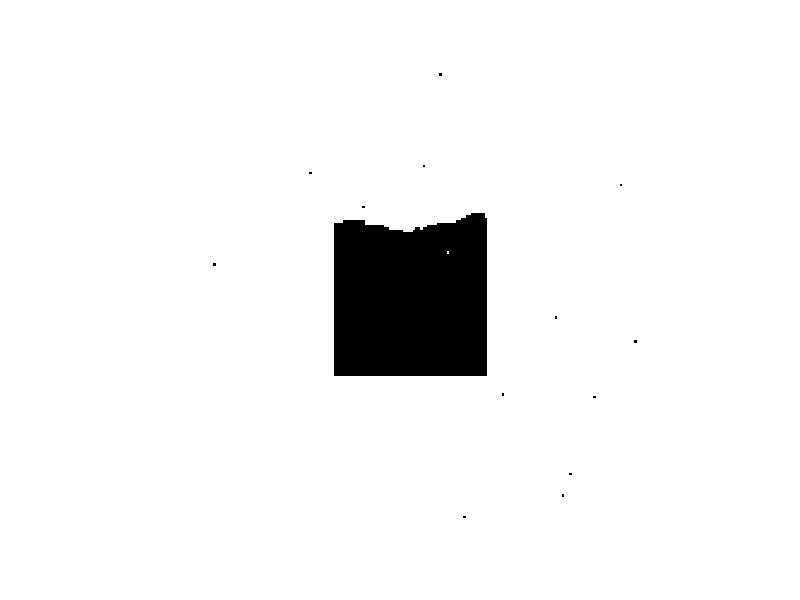}  \\[-1ex]
        \includegraphics[width=.2\textwidth, trim=120 40 105 40, clip]{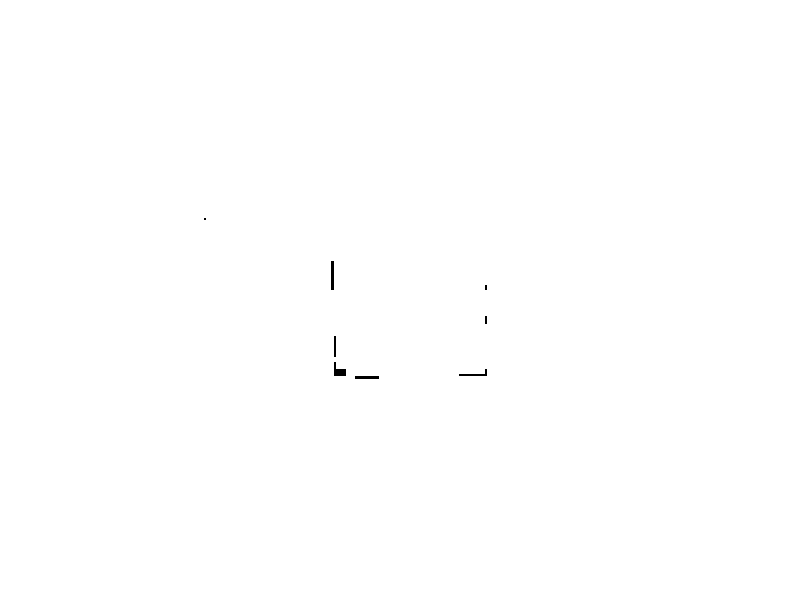} & \includegraphics[width=.2\textwidth, trim=120 40 105 40, clip]{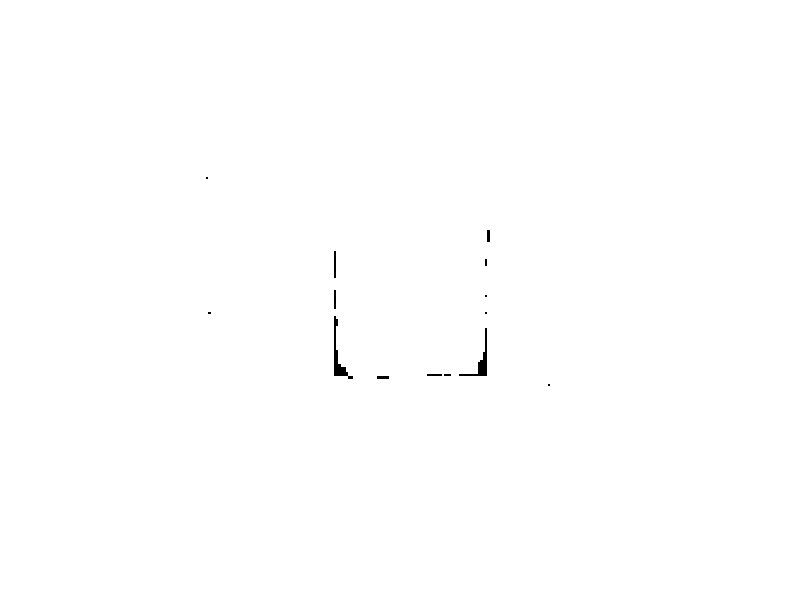} & \includegraphics[width=.2\textwidth, trim=120 40 105 40, clip]{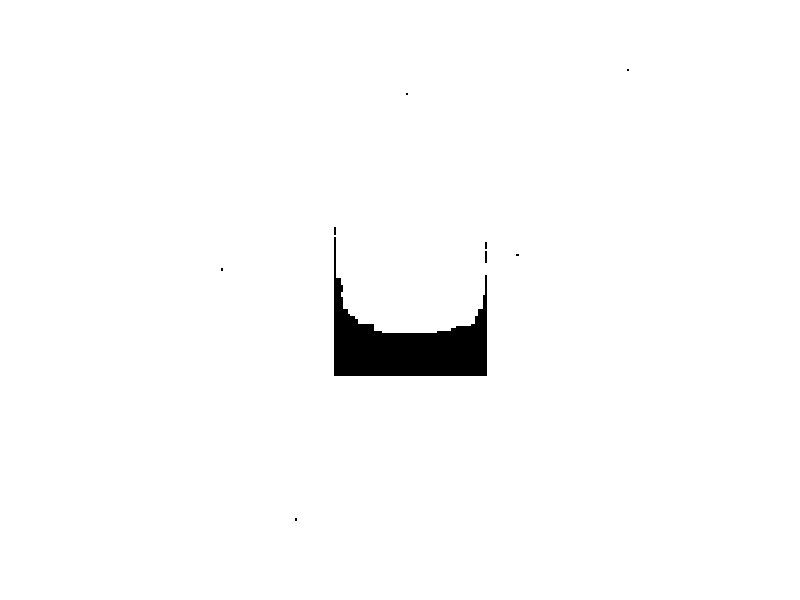}   & \includegraphics[width=.2\textwidth, trim=120 40 105 40, clip]{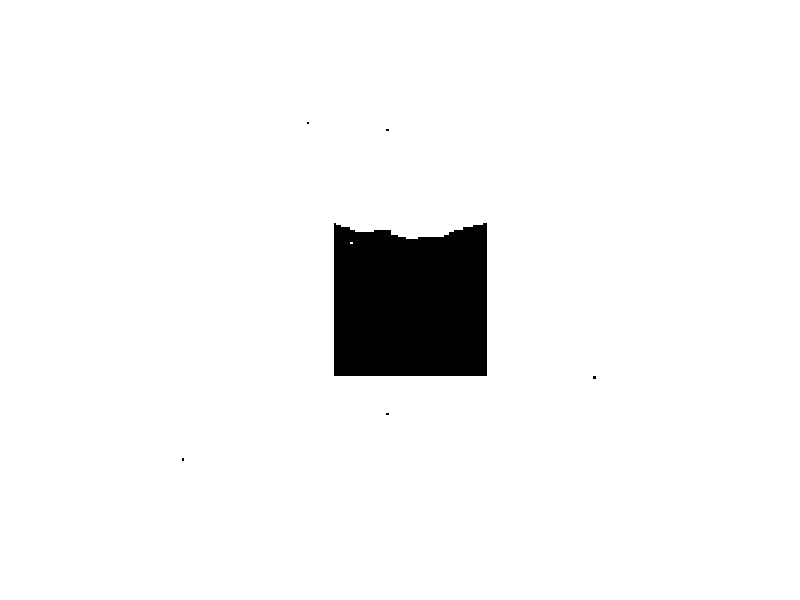} & \includegraphics[width=.2\textwidth, trim=120 40 105 40, clip]{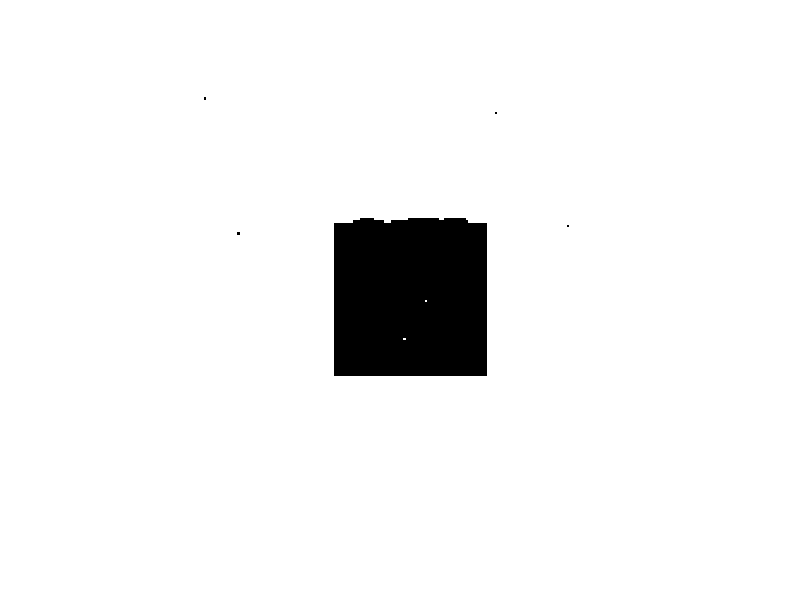} \\[-1ex]
        \includegraphics[width=.2\textwidth, trim=120 40 105 40, clip]{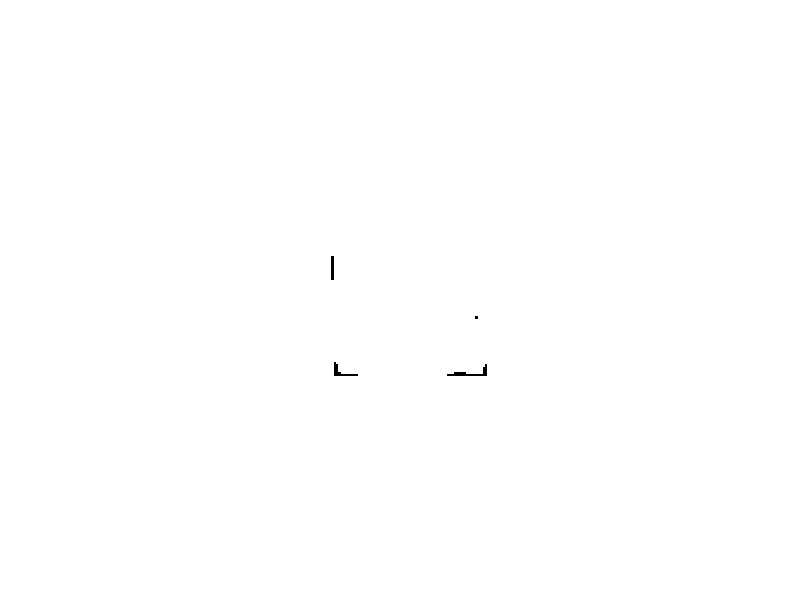} & \includegraphics[width=.2\textwidth, trim=120 40 105 40, clip]{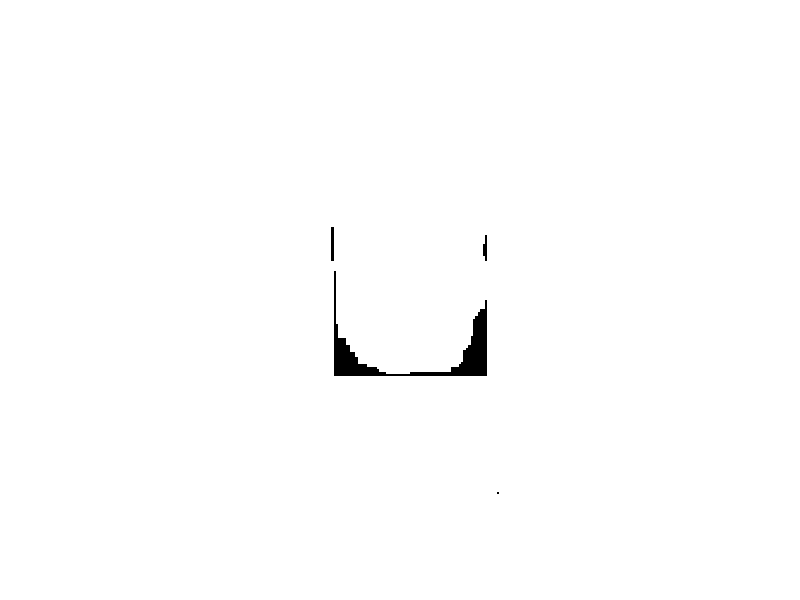} & \includegraphics[width=.2\textwidth, trim=120 40 105 40, clip]{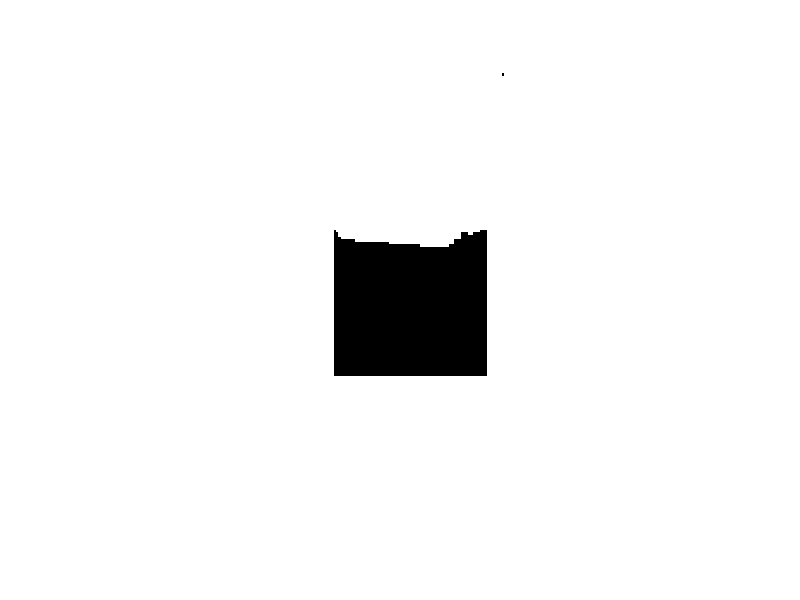}   & \includegraphics[width=.2\textwidth, trim=120 40 105 40, clip]{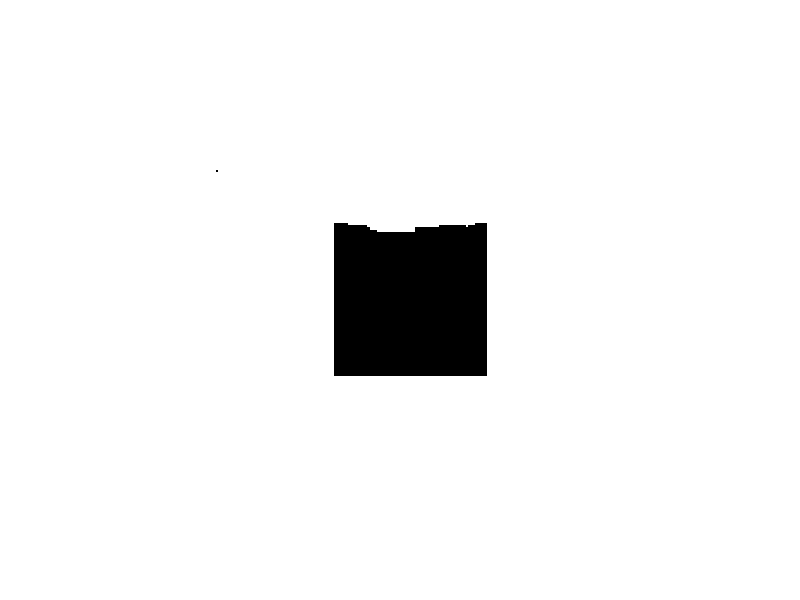} & \includegraphics[width=.2\textwidth, trim=120 40 105 40, clip]{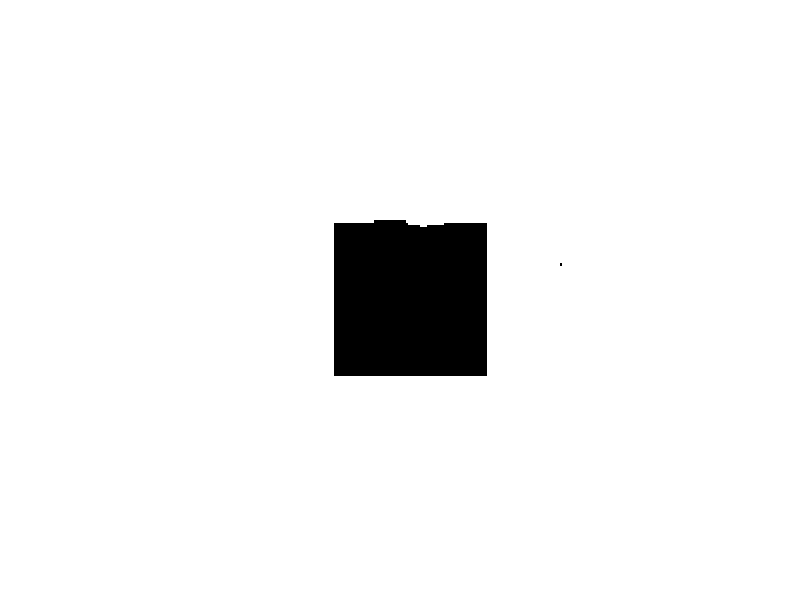} \\[-1ex]
    \end{tabular}
    
    \begin{tikzpicture}[overlay]\footnotesize
    \draw[black!90,->] (-7.7,0) -- (7,0) node[right] {$\beta$};
    \draw[black!90,->] (-7.5,-0.2) -- (-7.5,15) node[above] {$\kappa$};
    \draw (-5.5,0.15) -- (-5.5,-0.15) node[below] {$0.5$};
    \draw (-2.65,0.15) -- (-2.65,-0.15) node[below] {$0.75$};
    \draw (0.2,0.15) -- (0.2,-0.15) node[below] {$0.875$};
    \draw (3.05,0.15) -- (3.05,-0.15) node[below] {$1.0$};
    \draw (5.9,0.15) -- (5.9,-0.15) node[below] {$1.25$};
    
    \draw (-7.35,1.6) -- (-7.65,1.6) node[left] {$0.025$};
    \draw (-7.35,4.8) -- (-7.65,4.8) node[left] {$0.050$};
    \draw (-7.35,8) -- (-7.65,8) node[left] {$0.075$};
    \draw (-7.35,11.2) -- (-7.65,11.2) node[left] {$0.100$};
    \draw (-7.35,14.4) -- (-7.65,14.4) node[left] {$0.125$};
    \end{tikzpicture}
    \vspace{1cm}
    \caption{Illustrations of typical examples of 2-forms $\omega$ that make the largest contributions to the Wilson line expectation value  $\mathbb{E}_\varphi[\widehat{L_\gamma}(\omega)]$ for various values of $\beta$ and $\kappa$, in the case where $\gamma$ is a U-shaped loop and $G = \mathbb{Z}_2$. The simulations are performed on a subset of the $\mathbb{Z}^2$-lattice and a plaquette is drawn black if it lies in the support of the $\mathbb{Z}_2$-valued 2-form $\omega$.} 
    \label{figure: simulations}
\end{figure}
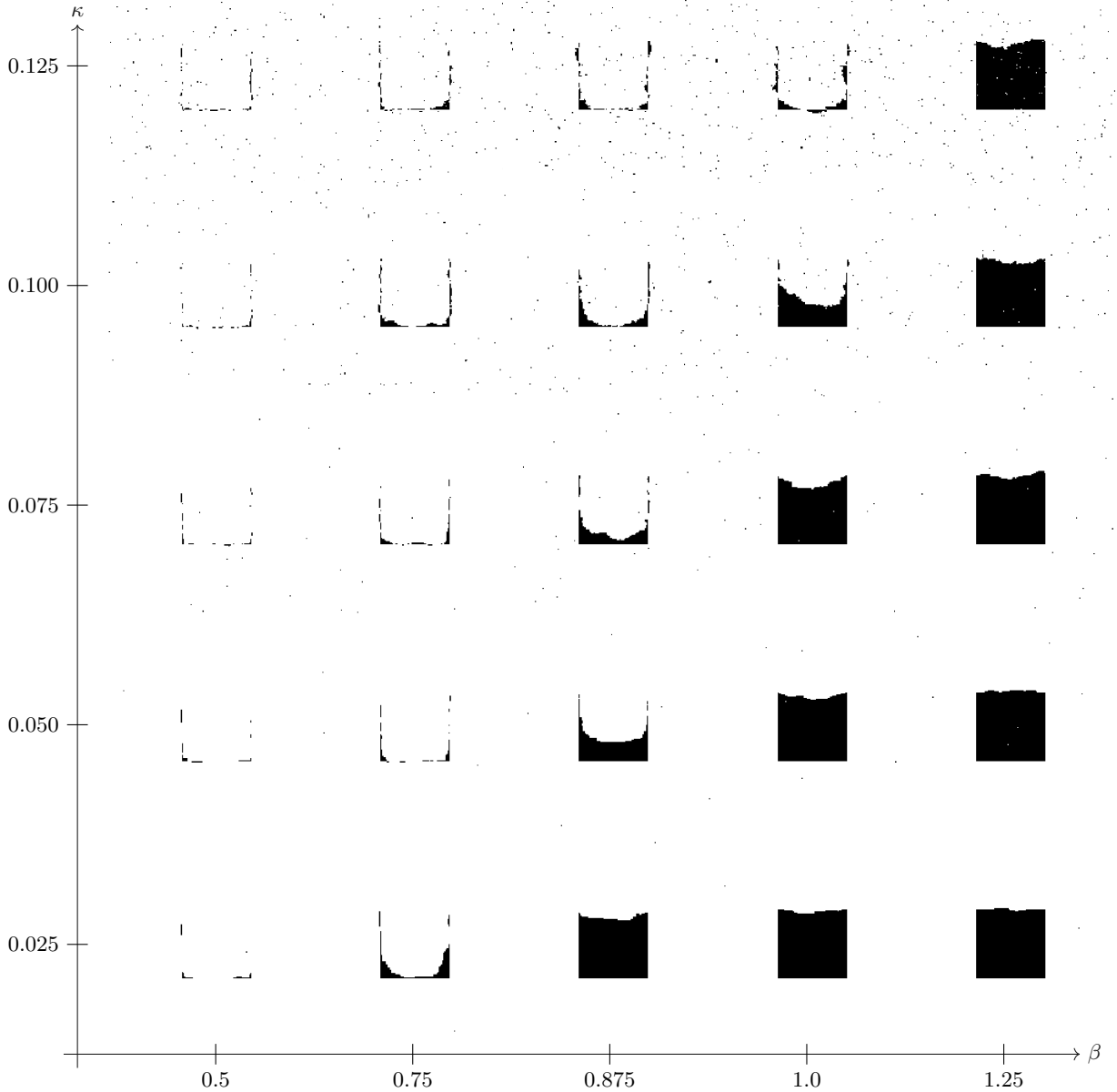

\subsection{Structure of the paper}
In Section~\ref{sec: preliminaries}, we give the necessary background on discrete exterior calculus needed for the rest of the paper.
In Sections~\ref{section: useful functions}, we introduce a few useful functions and discuss their properties.
Next, in Section~\ref{section: high-temperature expansion}, we use the functions of the previous section to describe the high-temperature representation that will be used throughout the rest of the paper. In Section~\ref{section: perimeter law}, we use the high-temperature representation to give a short proof of Theorem~\ref{theorem: perimeter law}.
Next, in Section~\ref{section: activity}, we define a notion of {activity} for the high-temperature representation of our model, and in~Section~\ref{section: useful upper bounds}, we use this concept to give natural upper bounds for certain local events.
In~Section~\ref{section: properties}, we state and prove an inequality concerning geometric properties of \( 2\)-forms. In~Section~\ref{section: bad events}, we use this inequality to give an upper bound on certain ``bad'' events. This bound is then used in~Section~\ref{section: poisson} to compare our model with a related Poisson process.
Finally, in~Section~\ref{sec: main result proof}, we complete the proof of our main result, Theorem~\ref{theorem: first theorem Z2}. Section~\ref{sec: main result proof} also contains a more general version of this result, Theorem~\ref{theorem: first theorem}, which is valid for $G = \mathbb{Z}_n$ for any $n \geq 2$.

\section{Preliminaries}\label{sec: preliminaries}

\subsection{The cell complex}
In this section, we introduce notation for the cell complexes of the lattices \( \mathbb{Z}^m \) and \( B_N \coloneqq [-N,N]^m \cap  \mathbb{Z}^m \) for \( m,N \geq 1 \). This section  closely follows the corresponding section in~\cite{flv2020}, to which we refer the reader for further details.

To simplify notation, we define \( e_1 \coloneqq (1,0,\dots,0) \), \( e_2 \coloneqq (0,1,0,\dots,0) \), \dots, \( e_m \coloneqq (0,\dots,0,1) \).

\subsubsection{Non-oriented cells}

If \( a \in \mathbb{Z}^m \), \( k \in \{ 0,1, \dots, m \} \), and \( \{ j_1,\dots, j_k \} \) is a subset of \( \{ 1,2, \dots, m \} \), we say that the set
\begin{equation*}
    (a; e_{j_1}, \dots, e_{j_k}) \coloneqq \bigl\{ x \in \mathbb{R}^m \colon \exists b_1, \dots, b_k \in [0,1] \text{ such that } x = a + \sum_{i=1}^k b_i e_{j_i}  \bigr\}
\end{equation*}
is a \emph{non-oriented \( k \)-cell}. Note that if \( \sigma \) is a permutation, then both \( (a; e_{j_1}, \dots, e_{j_k}) \) and \( (a; \sigma(e_{j_1}, \dots, e_{j_k})) \) represent the same non-oriented \( k \)-cell.

\subsubsection{Oriented cells}\label{sec: oriented cells}
To each non-oriented $k$-cell \( (a; e_{j_1}, \dots, e_{j_k}) \) with \( a \in \mathbb{Z}^m \), \( k \geq 1 \), and \( 1\leq j_1 < \dots < j_k\leq m \), we associate two \emph{oriented \( k \)-cells}, denoted \(  \frac{\partial}{\partial x^{j_1}}\big|_a \wedge \dots \wedge \frac{\partial}{\partial x^{j_k}}\big|_a\) and \( -\frac{\partial}{\partial x^{j_1}}\big|_a \wedge \dots \wedge \frac{\partial}{\partial x^{j_k}}\big|_a \), with opposite orientation.  
When \( a \in \mathbb{Z}^m \), \( 1\leq j_1 < \dots < j_k\leq m \), and \( \sigma  \) is a permutation of \( \{ 1,2, \dots, k \} \), we define
\begin{equation*}
    \frac{\partial}{\partial x^{j_{\sigma(1)}}}\bigg|_a \wedge \dots \wedge \frac{\partial}{\partial x^{j_{\sigma(k)}}}\bigg|_a 
    \coloneqq 
    \sgn(\sigma) \, 
    \frac{\partial}{\partial x^{j_1}}\bigg|_a \wedge \dots \wedge \frac{\partial}{\partial x^{j_k}}\bigg|_a.
\end{equation*}
If \( \sgn(\sigma)=1 \), then  \( \frac{\partial}{\partial x^{j_{\sigma(1)}}}\big|_a \wedge \dots \wedge \frac{\partial}{\partial x^{j_{\sigma(k)}}}\big|_a \) is said to be \emph{positively oriented}, and if \( \sgn(\sigma)=-1 \), then \( \frac{\partial}{\partial x^{j_{\sigma(1)}}}\big|_a \wedge \dots \wedge \frac{\partial}{\partial x^{j_{\sigma(k)}}}\big|_a  \) is said to be \emph{negatively oriented}. 
Analogously, we define \begin{equation*}
    -\frac{\partial}{\partial x^{j_{\sigma(1)}}}\bigg|_a \wedge \dots \wedge \frac{\partial}{\partial x^{j_{\sigma(k)}}}\bigg|_a 
    \coloneqq 
    -\sgn(\sigma) \, 
    \frac{\partial}{\partial x^{j_1}}\bigg|_a \wedge \dots \wedge \frac{\partial}{\partial x^{j_k}}\bigg|_a,
\end{equation*}
and say that \( -\frac{\partial}{\partial x^{j_{\sigma(1)}}}\bigg|_a \wedge \dots \wedge \frac{\partial}{\partial x^{j_{\sigma(k)}}}\bigg|_a  \) is positively oriented if \( -\sgn(\sigma) = 1 \), and negatively oriented if \( -\sgn(\sigma) = -1. \)

Let $\mathcal{L} = \mathbb{Z}^m$ or $\mathcal{L} = B_N \subseteq \mathbb{Z}^m$. An oriented cell \( \frac{\partial}{\partial x^{j_1}}\big|_a \wedge \dots \wedge \frac{\partial}{\partial x^{j_k}}\big|_a \) is said to be in \( \mathcal{L} \) if all corners of \( (a;e_{j_1},\dots, e_{j_k})\) belong to \( \mathcal{L} \); otherwise it is said to be {\it outside} $\mathcal{L}$. 
The set of all oriented \( k \)-cells in \( \mathcal{L} \) will be denoted by \( C_k(\mathcal{L}). \) The set of all positively and negatively oriented cells in \( C_k(\mathcal{L}) \) will be denoted by \( C_k(\mathcal{L})^+\) and \( C_k(\mathcal{L})^-\), respectively.

A non-oriented 0-cell \( a\in \mathbb{Z}^m \) is simply a point, and to each point we associate two oriented \(0\)-cells \( a^+ \) and \( a^-  \) with opposite orientation. We let \( C_0(\mathcal{L}) \) denote the set of all oriented \( 0 \)-cells.

When \( c \in C_k(\mathcal{L}), \) we let 
\begin{equation}
c^+ \coloneqq \begin{cases}
	c &\text{if } c \in C_k(\mathcal{L})^+ \cr 
	-c &\text{if } c \in C_k(\mathcal{L})^-.
\end{cases}
\end{equation}
In words, for \( c \in C_k(\mathcal{L}), \) \( c^+ \) is the positively oriented cell at the same position as \( c. \) 

Oriented 1-cells will be referred to as~\emph{edges}, and oriented 2-cells will be referred to as~\emph{plaquettes}.

\subsubsection{\( k \)-chains}\label{sec: chains}

The space of finite formal sums of positively oriented \( k \)-cells with integer coefficients will be denoted by \( C_k(\mathcal{L},\mathbb{Z}) \). 
Elements of \( C_k(\mathcal{L},\mathbb{Z}) \) will be referred to as \emph{\( k \)-chains}. 
If \( q \in C_k(\mathcal{L},\mathbb{Z}) \) and \( c \in C_k(\mathcal{L})^+ \), we let \( q[c] \) denote the coefficient of \( c \) in \( q \).
If \( c \in C_k(\mathcal{L})^- \), we let \( q[c]\coloneqq -q[-c]. \)
For \(q,q' \in  C_k(\mathcal{L},\mathbb{Z}) \), we define
\begin{equation*}
    q+q' \coloneqq 
    \sum_{c\, \in C_k(\mathcal{L})^+} \bigl(q[c] + q'[c] \bigr) c.
\end{equation*}
Using this operation, \( C_k(\mathcal{L},\mathbb{Z}) \) becomes a group. We define the \emph{support} of \( q \in C_k(\mathcal{L},G) \) by
\begin{equation*}
    \support q \coloneqq \bigl\{ c \in C_k(\mathcal{L})^+ \colon q[c] \neq 0 \bigr\}.
\end{equation*}
To simplify notation, for \( q \in C_k(\mathcal{L},G) \) and \( c\in C_k(\mathcal{L}) \), we write \( c \in q \) if either
\begin{enumerate}
    \item \( c \in C_k(\mathcal{L})^+ \) and \( q[c]>0\), or
    \item \( c \in C_k(\mathcal{L})^- \) and \( q[-c]<0. \)
\end{enumerate}

\subsubsection{The boundary of a cell}\label{sec: cell boundary}

When \( k \geq 2 \), we define the \emph{boundary} \(\partial c \in C_{k-1}(\mathcal{L}, \mathbb{Z})\) of  \( c = \frac{\partial}{\partial x^{j_1}}\big|_a \wedge \dots \wedge \frac{\partial}{\partial x^{j_k}}\big|_a \in C_k(\mathcal{L})\) by
\begin{align}
    \label{eq: boundary chain}
        &\partial c \coloneqq \sum_{k' \in \{ 1,\dots, k \}}  \biggl(
        (-1)^{k'}  \frac{\partial}{\partial x^{j_1}}\bigg|_a \wedge \dots \wedge \frac{\partial}{\partial x^{j_{k'-1}}}\bigg|_a \wedge  \frac{\partial}{\partial x^{j_{k'+1}}}\bigg|_a \wedge \dots \wedge \frac{\partial}{\partial x^{j_k}}\bigg|_a
        \\\nonumber
        &\qquad + (-1)^{k'+1}   
        \frac{\partial}{\partial x^{j_1}}\bigg|_{a + e_{j_{k'}}} \wedge \dots \wedge \frac{\partial}{\partial x^{j_{k'-1}}}\bigg|_{a + e_{j_{k'}}} \wedge  \frac{\partial}{\partial x^{j_{k'+1}}}\bigg|_{a + e_{j_{k'}}} \wedge \dots \wedge \frac{\partial}{\partial x^{j_k}}\bigg|_{a + e_{j_{k'}}} 
        \biggr). 
\end{align}
When \( c \coloneqq \frac{\partial}{\partial x^{j_1}}\big|_a \in C_1(\mathcal{L})\) we define the boundary \( \partial c \in C_0(\mathcal{L},\mathbb{Z}) \) by
\begin{equation*}
    \partial c = (-1)^1 a^+ + (-1)^{1+1} 
    (a+e_{j_1})^+ = (a+e_{j_1})^+ - a^+.
\end{equation*}
We extend the definition of \( \partial \) to \( k \)-chains \( q \in C_k(\mathcal{L},\mathbb{Z}) \) by linearity.
One verifies, as an immediate consequence of this definition, that if \( k \in \{ 2,3, \dots, m \} \), then \( \partial \partial c = 0 \) for any \( c \in \Omega_k(\mathcal{L}). \)

\subsubsection{The coboundary of an oriented cell}\label{sec: coboundary}

If \( k \in \{ 0,1, \dots, n-1 \} \) and \( c \in C_k(\mathcal{L})\) is an oriented \( k \)-cell, we define the \emph{coboundary} \( \hat \partial c \in C_{k+1}(\mathcal{L})\) of \( c \) as the \( (k+1) \)-chain 
\begin{equation*}
	\hat \partial c \coloneqq \sum_{c' \in C_{k+1}(\mathcal{L})} \bigl(\partial c'[c] \bigr) c'.
\end{equation*}
Note in particular that if \( c' \in C_{k+1}(\mathcal{L}),\) then \( \hat \partial c[c'] = \partial c'[c]. \) We extend the definition of \( \hat{\partial} \) to \( k \)-chains \( q \in C_k(\mathcal{L},\mathbb{Z}) \) by linearity.

\subsubsection{The boundary of a box}

An oriented \( k \)-cell \( c = \frac{\partial}{\partial x^{j_1}}\big|_a \wedge \dots \wedge \frac{\partial}{\partial x^{j_k}}\big|_a \in C_k(B_N)\) is said to be a \emph{boundary cell} of a box \( B = \bigl( [a_1,b_1]\times \dots \times [a_m,b_m] \bigr) \cap \mathbb{Z}^m \subseteq B_N\), or equivalently to be in \emph{the boundary} of \( B \), if the non-oriented cell \( (a;e_{j_1}, \dots, e_{j_k}) \) is a subset of the boundary of 
\(  [a_1,b_1]\times \dots \times [a_m,b_m]. \)

\subsection{Discrete exterior calculus}
In what follows, we give a brief overview of discrete exterior calculus on the cell complexes of \( \mathbb{Z}^m \) and \( B = [a_1,b_1] \times \dots \times [a_m,b_m]  \cap  \mathbb{Z}^m \) for \( m \geq 1 \). As with the previous section, this section will closely follow the corresponding section in~\cite{flv2020}, where we refer the reader for further details and proofs.

All of the results in this subsection are obtained under the assumption that an abelian group \( G \), which is not necessarily finite, has been given. In particular, they all hold for both \( G=\mathbb{Z}_n \) and \( G=\mathbb{Z} \).

\subsubsection{Discrete differential forms}\label{sec: ddf}

A homomorphism from the group \( C_k(\mathcal{L},\mathbb{Z}) \) to the group \( G \) is called a \emph{\( k \)-form}. The set of all such \( k\)-forms will be denoted by \( \Omega^k(\mathcal{L},G) \). This set becomes an abelian group if we add two homomorphisms by adding their values in \( G \).

The set $C_k(\mathcal{L})^+$ of positively oriented $k$-cells is naturally embedded in  $C_k(\mathcal{L},\mathbb{Z})$ via the map  $c \mapsto 1\cdot c$, where \( 1 \cdot c \) refers to the \( k\)-chain \( q \) with \(q[c] = 1 \) and \( q[c']=0\) for all \( c' \in C_k(\mathcal{L})^+\smallsetminus \{ c \}.\)  We will frequently identify $c \in C_k(\mathcal{L})^+$ with the $k$-chain $1\cdot c$ using this embedding. Similarly, we will identify a negatively oriented $k$-cell $c \in C_k(\mathcal{L})^-$ with the $k$-chain $(-1) \cdot (-c)$. 
In this way, a $k$-form $\omega$ can be viewed as a \( G \)-valued function on \( C_k(\mathcal{L}) \) with the property that \( \omega(c) = -\omega(-c) \) for all \( c \in C_k(\mathcal{L}) \). Indeed, if \( \omega \in \Omega^k(\mathcal{L},G) \) and \( q = \sum a_i c_i \in C_k(\mathcal{L},\mathbb{Z}) \), we have
\begin{equation*}
    \omega(q) = \omega \Bigl(\sum a_i c_i \Bigr) = \sum a_i \omega(c_i),
\end{equation*}
and hence a \( k \)-form is uniquely determined by its values on positively oriented \( k \)-cells.

If \( \omega \) is a \( k \)-form, it is useful to represent it by the formal expression 
\begin{equation*}
    \sum_{1 \leq j_1 < \dots < j_k \leq m} \omega_{j_1\dots j_k} dx^{j_1} \wedge \cdots \wedge dx^{j_k}.
\end{equation*}
where \( \omega_{j_1\dots j_k} \) is a \( G \)-valued function on the set of all \( a \in \mathbb{Z}^m \) such that \( \frac{\partial}{\partial x^{j_1}} \big|_a \wedge \dots \wedge \frac{\partial}{\partial x^{j_k}}\big|_a \in C_k(\mathcal{L})\), defined by
\begin{equation*}
    \omega_{j_1 \dots j_k}(a) = \omega \biggl( \frac{\partial}{\partial x^{j_1}} \bigg|_a \wedge \dots \wedge \frac{\partial}{\partial x^{j_k}} \bigg|_a \biggr).
\end{equation*} 

If \( 1\leq j_1 < \dots < j_k\leq m \) and \( \sigma  \) is a permutation of \( \{ 1,2, \dots, k \} \), we define
\begin{equation*}
    dx^{j_{\sigma(1)}}  \wedge \dots \wedge d x^{j_{\sigma(k)}}
    \coloneqq 
    \sgn(\sigma) \, 
    d  x^{j_1} \wedge \dots \wedge d x^{j_k},
\end{equation*} 
and if \( 1 \leq j_1,\dots, j_k \leq n \) are such that \( j_i = j_{i'} \) for some \( 1 \leq i < i' \leq k \), then we let
\begin{equation*}
    d  x^{j_1} \wedge \dots \wedge d x^{j_k} \coloneqq 0.
\end{equation*}

Given a \( k \)-form \( \omega \), we let \( \support \omega \) denote the support of \( \omega \), i.e., the set of all oriented \( k \)-cells \( c \) such that \( \omega(c) \neq 0 \). Note that $\support \omega$ always contains an even number of elements. 

\subsubsection{The exterior derivative}\label{sec: derivative}
Given \( h \colon \mathbb{Z}^m \to G \), \( a \in \mathbb{Z}^m \), and \( i \in \{1,2, \dots, m \} \), we let 
\begin{equation*}
    \partial_i h(a) \coloneqq h(a+e_i) - h(a) .
\end{equation*}
If \( k \in \{ 0,1,2, \dots, m \} \) and \( \omega \in \Omega^k(\mathcal{L},G) \), we define the \( (k+1) \)-form \( d\omega \in \Omega^{k+1}(\mathcal{L},G) \) by
\begin{equation*}
    d\omega = \sum_{1 \leq j_1 < \dots < j_k \leq m} \sum_{i=1}^m \partial_i \omega_{j_1,\dots,j_k} \,  dx^i \wedge (dx^{j_1} \wedge \dots \wedge dx^{j_k}).
\end{equation*} 
The operator \( d \) is called the \emph{exterior derivative.} 
Using~\eqref{eq: boundary chain}, one can show that for \( \omega\in \Omega^k(\mathcal{L},G) \) and \( c \in  C_{k+1}(\mathcal{L},\mathbb{Z}) \), we have \( d\omega(c) = \omega(\partial c). \) This equality is a discrete version of Stokes' theorem. 
Recalling that if \( k \in \{ 2,3,\dots, m-2\}  \) and \( c \in C_{k+2}(\mathcal{L}) \), then \( \partial \partial c = 0, \) it follows from the discrete Stokes theorem that \( dd\omega = 0 \) for any \( \omega \in \Omega^{k}(\mathcal{L},G). \)

\subsubsection{The coderivative}\label{sec: coderivative}

Given \( h \colon \mathbb{Z}^m \to G \), \( a \in \mathbb{Z}^m \), and \( i \in \{ 1,2, \dots, m \} \), we let 
\begin{equation*}
    \bar \partial_i h(a) \coloneqq h(a) - h(a-e_i).
\end{equation*}
When \( k \in \{ 1,2, \dots, m \} \) and \( \omega \in \Omega^k(\mathcal{L},G) \), we define the \( (k-1) \)-form \( \delta \omega \in \Omega^{k-1}(\mathcal{L},G) \)  by
\begin{equation*}
    \delta \omega \coloneqq 
    \sum_{1 \leq j_1 < \dots < j_k \leq m}
    \sum_{i = 1}^k  
    (-1)^{i} \, \bar \partial_{j_i} \omega 
    \, dx^{j_1} \wedge \cdots \wedge dx^{j_{i-1}}\wedge dx^{j_{i+1}} \wedge \cdots \wedge  dx^{j_{k}}.
\end{equation*}
The operator \( \delta \) is called the \emph{coderivative}. 
Analogously as for the discrete derivative, one can show that for any \( \omega\in \Omega^k(\mathcal{L},G) \) and \( c \in C_{k-1} (\mathcal{L},\mathbb{Z}) \), we have
    \begin{equation*}
	    \delta\omega(c) = \omega(\hat \partial c).
    \end{equation*}
In particular, if \( \omega \in \Omega^2(\mathcal{L},G) \) and \( e \in C_1 (\mathcal{L},\mathbb{Z}) \), then
\begin{align}\label{deltaomegae}
    &\delta \omega(e) = \omega(\hat \partial e) = \omega\Bigl( \sum_{p \in \hat \partial e} p \Bigr) = \sum_{p \in \hat \partial e} \omega(p) = \sum_{p \in \support \hat \partial e} \partial p[e] \omega(p).
\end{align}

\subsubsection{Restrictions of forms}

If \( \omega \in \Omega^k(\mathcal{L},G) \), \( C \subseteq C_k(\mathcal{L}) \) and \( c \in C \), we define 
\begin{equation*}
    \omega|_C(c) \coloneqq 
    \begin{cases}
        \omega(c) &\text{if } c \in \pm C, \cr 
        0 &\text{else.}
    \end{cases}
\end{equation*}

 \subsection{Unitary gauge}\label{sec: unitary gauge}
In this section, we introduce gauge transformations, and describe how these can be used to rewrite the Wilson line expectation as an expectation with respect to a slightly simpler probability measure.
 
Before we can state the main results of this section, we need to briefly discuss gauge transformations.  
To this end, for \( \eta \in \Omega^0(B_N,G) \), consider the bijection \( \tau \coloneqq \tau_\eta \coloneqq \tau_\eta^{(1)} \times \tau_\eta^{(2)} \colon \Omega^1(B_N,G)  \times \Omega^0(B_N,G)  \to  \Omega^1(B_N,G) \times \Omega^0(B_N,G) \), defined by
\begin{equation}\label{eq: gauge transform}
    \begin{cases}
     \sigma(e) \mapsto -\eta(x) +\sigma(e) + \eta(y), & e=(x,y)\in C_1(B_N), \cr
     \phi(x) \mapsto  \phi(x) + \eta(x), & x \in C_0(B_N).
    \end{cases}
\end{equation}
Any mapping \( \tau \) of this form is called a \emph{gauge transformation}, and functions \( f: \Omega^1(B_N,G) \times \Omega^0(B_N,G)  \to \mathbb{C} \) which are invariant under such mappings in the sense that \(f= f \circ \tau\) are said to be \emph{gauge invariant}.
For \( \beta, \kappa \geq   0 \) and  \( \sigma \in \Omega^1(B_N,G) \), we define 
\begin{equation}\label{eq: fixed length unitary measure}
    \mu_{N,\beta,\kappa}(\sigma) 
    \coloneqq
    Z_{N,\beta,\kappa}^{-1}\exp\pigl(\beta \sum_{p \in C_2(B_N)}      \rho\bigl(d  \sigma(p)\bigr) + \kappa \sum_{e \in C_1(B_N)}   \rho \bigl( \sigma(e)\bigr)\pigr)  ,
\end{equation} 
where \( Z_{N,\beta,\kappa}^{-1} \) is a normalizing constant which ensures that \( \mu_{N,\beta,\kappa} \) is a probability measure. We let \( \mathbb{E}_{N,\beta,\kappa} \) denote the corresponding expectation. Note that as \( \rho(d\sigma(-p)) = \overline{\rho(d\sigma(p))} \) and \( \rho(\sigma(-e)) = \overline{\rho(\sigma(e))} \), the measure \( \mu_{N,\beta,\kappa} \) is real. 

The main reason that gauge transformations are useful to us is the following result.

\begin{proposition}[Proposition~2.21 in~\cite{flv2021}]\label{proposition: unitary gauge one dim}
Let \( \beta,\kappa \geq 0 \) and assume that the function \( f \colon \Omega^1(B_N,G) \times \Omega^0(B_N,G) \to \mathbb{C} \) is gauge invariant. Then 
\begin{equation*}
    \mathbb{E}_{N,\beta,\kappa,\infty}\bigl[f(\sigma,\phi)\bigr] = 
    \mathbb{E}_{N,\beta,\kappa}\bigl[f(\sigma,0)\bigr],
\end{equation*} 
where \( \mathbb{E}_{N,\beta,\kappa,\infty} \) is the expectation corresponding to the measure \( \mu_{N,\beta,\kappa,\infty} \) introduced in~\eqref{eq: London limit measure}.
\end{proposition}
The main idea of the proof of Proposition~\ref{proposition: unitary gauge one dim} is to perform a change of variables, where we for each pair \( (\sigma,\phi) \) apply the gauge transformation \( \tau_{-\phi}, \) thus mapping \( \phi \) to \( 0 \). After having applied this gauge transformation, we say that we are working in the \emph{unitary gauge}.

Since the function \( (\sigma,\phi) \mapsto  L_\gamma(\sigma,\phi) \) is gauge invariant for any path \( \gamma \), one obtains the following result as an immediate corollary of Proposition~\ref{proposition: unitary gauge one dim}.

\begin{corollary}[Corollary 2.17 in~\cite{f2021b}]\label{corollary: unitary gauge}
    Let \( \beta \in [0,\infty],\) \( \kappa \geq 0 \), and let \( \gamma \) be a path in \( C_1(B_N) \). Then
    \begin{equation*}
        \mathbb{E}_{N,\beta,\kappa,\infty}\bigl[L_\gamma(\sigma,\phi)\bigr] =
        \mathbb{E}_{N,\beta,\kappa}\bigl[L_\gamma(\sigma,0)\bigr] =
        \mathbb{E}_{N,\beta,\kappa}\pigl[\rho\bigl(\sigma(\gamma)\bigr)\pigr]. 
    \end{equation*} 
\end{corollary}

Results analogous to Proposition~\ref{proposition: unitary gauge one dim} are considered well-known in the physics literature.

With the current section in mind, we will work with the measure \( \sigma \sim \mu_{N,\beta, \kappa} \) rather than \( (\sigma,\phi) \sim \mu_{N,\beta,\kappa,\infty}\) throughout the rest of this paper, together with the observable
\begin{equation*}
    L_\gamma(\sigma) \coloneqq L_\gamma(\sigma,0) = \prod_{e \in \gamma} \rho\bigl(\sigma(e)\bigr) = \rho(\sigma(\gamma)).
\end{equation*}

\subsection{Existence of the limiting expectation value}\label{sec: ginibre}

In this section, we recall a result which shows existence and translation invariance of the limit \( \langle L_\gamma(\sigma) \rangle_{\beta,\kappa,\infty} \) defined in the introduction. This result is well-known, and is often mentioned in the literature as a direct consequence of the Ginibre inequalities. A full proof in the special case \( \kappa = 0 \) was included in~\cite{flv2020}, and the general case can be proven completely analogously, hence we omit the proof here. 

\begin{proposition}\label{proposition: limit exists}
    Let \( G = \mathbb{Z}_n \), \( M \geq 1 \), and let \( f \colon \Omega^1(B_M,G)  \to \mathbb{R}\).
    For \( M' \geq M \), we abuse notation and let \( f \) denote the natural extension of \( f \) to \( C_1(B_{M'}) \), i.e., the unique function such that \( f(\sigma) = f(\sigma|_{C_1(B_M)}) \) for all \( \sigma \in \Omega^1(B_{M'},G) \).
    Further, let \( \beta \in [0, \infty] \) and \( \kappa \geq 0 \). Then the following hold.
    \begin{enumerate}[label=\textnormal{(\roman*)}]
        \item The limit \( \lim_{N \to \infty} \mathbb{E}_{N,\beta,\kappa} \bigl[ f(\sigma) \bigr] \) exists.
        \item For any translation \( \tau \) of \( \mathbb{Z}^m \), we have \[ \lim_{N \to \infty} \mathbb{E}_{N,\beta,\kappa}  \bigl[f \circ \tau(\sigma)\bigr] =  \lim_{N \to \infty} \mathbb{E}_{N,\beta,\kappa}\bigl[ f(\sigma) \bigr] .\]
    \end{enumerate} 
\end{proposition}

\subsection{Connected sets of plaquettes}\label{connectedsubsec}

For \( p,p' \in C_2(B_N)^+ ,\) we define the relation \( \sim \) by letting \( p \sim p' \) if and only if \( (\partial p)^+ \cap (\partial p')^+ \neq \emptyset \), i.e., $p \sim p'$ if and only if $p$ and $p'$ are adjacent to each other. If \( (\partial p)^+ \cap (\partial p')^+ = \emptyset, \) we instead write \( p \nsim p'. \)
Given a set \( P \subseteq C_2(B_N)^+, \) we say that two plaquettes \( p,p' \subseteq C_2(B_N)^+ \) are \emph{connected in \( P \)} if there is \( k \geq 1 \) and a sequence \( p_1, p_2, \dots , p_k \in P \) such that \( p = p_1 \sim p_2 \sim \dots \sim p_k = p'. \)
A set \( P \subseteq C_2(B_N)^+ \) is \emph{connected} if all \( p,p' \in P \) are connected in \( P. \)
%
%
A set \( P' \subseteq P \subseteq C_2(B_N)^+ \) is said to be a connected component of \( P \) if it is a maximal connected subset of \( P . \) We let \( \| P \|\) denote the number of connected components of \( P .\)
Similarly, when \( P,P' \subset C_2(B_N)^+, \) we let \( P^{P'} \) denote the union of the connected components of \( P \) which intersect \( P'. \)
We extend these definitions to forms as follows. We say that \( \omega \in \Omega^2(B_N,\mathbb{Z}_n) \) has \emph{connected support} if \( (\support \omega)^+ \) is connected, 
and let
\begin{equation*}
    \| \omega\| \coloneqq \bigl\| (\support \omega)^+\bigr\|.
\end{equation*}

If \( \omega \in \Omega^2(B_N,\mathbb{Z}_n) ,\) then there is a unique decomposition \( \omega = \omega_1 + \dots + \omega_k, \) where for each \( j \in \{1,2, \dots , k \} \) we have  \( \omega_j \coloneqq \omega|_{P_j} ,\) where \( P_1, P_2, \dots, P_m \) are the connected components of \( (\support \omega)^+. \)
Given a set \( E \subseteq C_1(B_N)^+ ,\) we let 
\begin{equation}\label{eq: omegaE def}
    \omega^E \coloneqq \sum_{j \colon \exists e \in E \colon \support \hat \partial e \cap \support \omega_j \neq \emptyset} \omega_j.
\end{equation}
When \( \gamma \) is a path, we let \( \omega^{\gamma,2} \coloneqq \omega^{\support \gamma} ,\) and define
\begin{equation}\label{eq: omegagammadef}
     \omega^\gamma \coloneqq \omega^{\support \delta \omega \cap \support \gamma} = \sum_{j \colon \support \delta \omega_j \cap \support \gamma \neq \emptyset} \omega_j .
\end{equation}

\subsection{A partial ordering of \(\Omega^k(\mathcal{L},G)\)}\label{sec: the partial ordering}

The following definition will be natural to use to compare 2-forms when we work with high-temperature expansions.
\begin{definition}
    Let \( k \geq 1 \) and \( \omega,\omega' \in \Omega^k(\mathcal{L},G). \) If \( \omega|_{\support \omega'} = \omega' \),  and the sets \( \support \delta \omega' \) and \( \support \delta (\omega-\omega') \) are disjoint, then we write \( \omega' \lhd \omega. \)
\end{definition}
Intuitively, if \( \omega' \lhd \omega, \) then \( \omega = \omega' + (\omega-\omega') ,\) and the supports of \( \omega' \) and \( \omega-\omega' \) are well separated, and we later show that the above characterisation is exactly what is required to guarantee that a certain function, which will be important later, factorizes under this assumption (see Lemma~\ref{lemma: action factorization forms}).

\begin{remark}
    We mention that \( \omega \lhd \omega' \) if and only if \( *\omega \leq *\omega', \) where \( \leq \) is the partial order on \( k \)-forms introduced in~\cite[Definition 2.6]{flv2021} and \(* \) is the Hodge operator (see, e.g.,~\cite[Definition 2.1.11]{flv2021}). Using~\cite[Lemma~2.7]{flv2021}, this shows that \( \lhd \) is a partial order on \( \Omega^k(B_N,G). \)
\end{remark}

\begin{lemma}\label{lemma: connected components and lhd}
    Let \( \omega \in \Omega^2(\mathcal{L},G), \)  let \( P_0,P_1,\dots, P_{k-1} \) be the connected components of \( (\support \omega)^+,\) and let  \( J \subseteq J' \subseteq [k].\)
    Define \(\omega^J \coloneqq \omega|_{ \bigsqcup_{j \in J} P_j}\) and \(\omega^{J'} \coloneqq \omega|_{ \bigsqcup_{j \in J'} P_j}.\)
    Then,  \(  \omega^J \lhd \omega^{J'}. \)
\end{lemma}

\begin{proof}
    Since the sets \( P_j \) are disjoint subsets of \( (\support \omega)^+, \)  we have \( (\support \omega^J)^+ = \bigsqcup_{j \in J} P_j \) and \( (\support \omega^{j'})^+ = \bigsqcup_{j \in J'} P_j. \)
    Since \( J \subseteq J', \) it follows that
    \begin{equation}\label{eq: first step of lhd}
        \omega^{J'} = \omega|_{\bigsqcup_{j \in J'} P_j} 
        =
        \bigl(\omega|_{\bigsqcup_{j \in J} P_j}\bigr)\pigr|_{\bigsqcup_{j \in J'} P_j}
        =
        \omega^J|_{\bigsqcup_{j \in J'} P_j}
        =
        \omega^J|_{\support \omega^{J'}}.
    \end{equation} 
    
    Since the sets \( P_j \) are disjoint subsets of \( (\support \omega)^+, \)  for any \( j,j' \in [k] \) the sets \( \support \omega|_{ P_j} \) and \( \support \omega|_{ P_{j'}}  \) are disjoint. Consequently, we have
    \begin{equation}\label{eq: delta decomposed}
        \delta \omega^J
        = 
        \delta \omega|_{ \sqcup_{j \in J} P_j}
        =
        \delta \sum_{j \in J} \omega|_{ P_j}
        =
        \sum_{j \in J}  \delta \omega|_{ P_j}
    \end{equation}
    and, using~\eqref{eq: delta decomposed}, we obtain
    \begin{equation}\label{eq: delta decomposed ii}
        \delta (\omega^{J'}-\omega^J)
        =
        \delta \omega^J-\delta \omega^J
        =
        \sum_{j \in J'}  \delta \omega|_{ P_j}-\sum_{j \in J}  \delta \omega|_{ P_j}
        =
        \sum_{j \in J'\smallsetminus J}  \delta \omega|_{ P_j}.
    \end{equation}
    Since the sets \( P_j \) are the connected components of \( (\support \omega)^+, \)  for any \( j,j' \in [k] \) the sets \( \support \delta \omega|_{ P_j} \) and \( \support \delta \omega|_{ P_{j'}}  \) are disjoint. 
    Consequently, it follows from~\eqref{eq: delta decomposed ii} that \( \support \delta \omega^J \) and \( \support \delta(\omega^{J'}-\omega^J) \) are disjoint.
    Since, by~\eqref{eq: first step of lhd}, we also have \( \omega^{J'} = \omega^J|_{\support \omega^{J'}}, \) it follows that \( \omega^J \lhd \omega^{J'},\) which is the desired conclusion.
\end{proof} 

Recalling the definition of \( \omega^E \) from~\eqref{eq: omegaE def}, we have the following lemma.
\begin{lemma}\label{lemma: action factorization forms iib}
    Let \( E_1 \subseteq E_2 \subseteq C_1(B_N)^+ \) and let \( \omega \in \Omega^2(\mathcal{L},G). \) Then \( \omega^{E_1} \lhd \omega^{E_2}. \)
\end{lemma}

\begin{proof} 
    Let \( P_0, P_1, \dots, P_{k-1} \) be the connected components of \( \omega. \) For \( E \subseteq C_1(B_N)^+, \) define
    \begin{equation*}
        J_E \coloneqq \bigl\{ j \in [k] \colon \exists e \in E \text{ such that } \support  \omega|_{ P_j} \cap \support \hat \partial e \neq \emptyset \bigr\}.
    \end{equation*}
    Since \( E_1 \subseteq E_2, \) we have \( J_{E_1} \subseteq J_{E_2}. \) Applying Lemma~\ref{lemma: connected components and lhd} with \( \omega, \) \( J_{E_1} \) and \( J_{E_2},\) we thus obtain the desired conclusion.
\end{proof}

\subsection{Rectangular paths}\label{sec: rectangular}

If \( \gamma \) is a closed path in a 2-dimensional hyperplane in \( \mathbb{Z}^m \) with 
\begin{equation*}
    \pigl|\bigl\{ e \in \support \gamma \colon \exists p \in \support \hat \partial e \colon |\support \partial p \cap \support \gamma| \geq 2 \bigr\} \pigr| = 4, 
\end{equation*}
then \( \gamma \) is said to be a \emph{rectangular loop}. (Recall that the support of a $1$-chain is defined as the set of {\it positively} oriented $1$-cells with non-zero coefficient, so that both $\support \gamma$ and $\support \partial p$ consist of only positively oriented edges.) 

If \( \gamma \) is a path and there is a rectangular loop \( \gamma_R \) with \( \support \gamma \subseteq \support \gamma_R \) and \( \gamma_R[e] = \gamma[e] \) for all \( e \in \support \gamma, \) then \( \gamma \) is said to be a \emph{rectangular path}. We note that given a rectangular path \( \gamma, \) the path \( \gamma_R \) is not necessarily unique.

\subsection{Corner plaquettes}

Assume that a path \( \gamma \) is given.  We let
\begin{equation}\label{eq: Pgammac}
    \mathcal{P}_{\gamma,c} \coloneqq \pigl\{ p \in C_2(\mathcal{L})^+ \colon \bigl|\support \partial p \cap \support \gamma \bigr| \geq 2 \pigr\}.
\end{equation}
The plaquettes in \( \mathcal{P}_{\gamma,c} \) will be called the \emph{corner plaquettes} of \( \gamma. \)
Given \( \omega \in \Omega^2(\mathcal{L},G),\) we define
\begin{equation}\label{Pomegagammacdef}
    P_{\omega,\gamma,c} \coloneqq \bigl\{ p \in (\support \omega)^+ \colon |\support \partial p  \cap \support \gamma \cap \support \delta \omega| =2 \bigr\}.
\end{equation} 

We note that we always have \( |P_{\omega,\gamma,c}| \leq |\mathcal{P}_{\gamma,c}|, \) and that if \( \gamma \) is a path along the boundary of some rectangle \( R \), then \( |\mathcal{P}_{\gamma,c}| \leq 4. \)

\begin{figure}[tp]
        \centering 
        \begin{subfigure}[b]{0.3\textwidth}\centering
    \begin{tikzpicture}[scale=0.6]
        
        \draw[line width=0.1mm,ForestGreen!60!black] (0,4) -- (0,0) -- (7,0);

        \foreach \x in {0}
            \foreach \y in {0}
                \fill[BlueViolet] (\x+0.1,\y+0.1) rectangle ++(0.3,0.3);
         
    \end{tikzpicture}
    \caption{}
    \end{subfigure}
        \begin{subfigure}[b]{0.3\textwidth}\centering
    \begin{tikzpicture}[scale=0.6]
        
        \draw[line width=0.1mm,ForestGreen!60!black] (0,4) -- (0,0) -- (7,0) -- (7,4);

        \foreach \x in {0}
            \foreach \y in {0}
                \fill[BlueViolet] (\x+0.1,\y+0.1) rectangle ++(0.3,0.3);
        
        \foreach \x in {6.5}
            \foreach \y in {0}
                \fill[BlueViolet] (\x+0.1,\y+0.1) rectangle ++(0.3,0.3); 
    \end{tikzpicture}
    \caption{}
    \end{subfigure}
        \begin{subfigure}[b]{0.3\textwidth}\centering
    \begin{tikzpicture}[scale=0.6]
        
        \draw[line width=0.1mm,ForestGreen!60!black] (0,4) -- (0,0) -- (7,0) -- (7,4) -- cycle;

        \foreach \x in {0,6.5}
            \foreach \y in {0,3.5}
                \fill[BlueViolet] (\x+0.1,\y+0.1) rectangle ++(0.3,0.3);
         
    \end{tikzpicture}
    \caption{}
    \end{subfigure}
    \caption{In the three figures above, we draw the set \( \mathcal{P}_{\gamma,c} \) for three different rectangular paths~\( \gamma \).}
    \label{figure: corner plaquettes}
\end{figure}
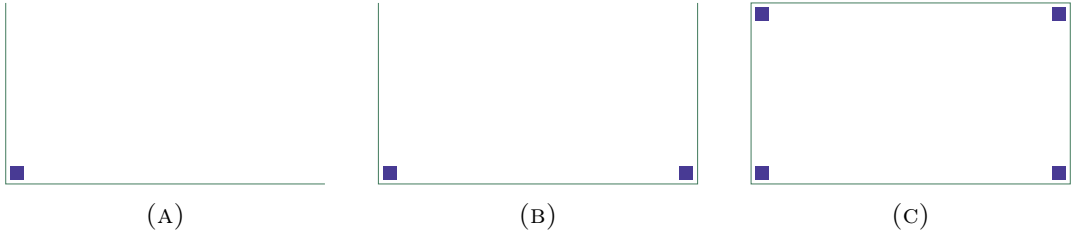

\subsection{Distance between cells}

When \( k,k' \in \{ 0,1,\dots, m \} ,\) \( c \in C_k(\mathcal{L}) \) and \( c' \in C_{k'}(\mathcal{L}) \) we define the distance between \( c \) and \( c' \) to be the smallest euclidean distance between any pair of points \( x \in \hat c \) and \( x' \in \hat c', \) where \( \hat c \) and \( \hat c' \) are non-oriented cells corresponding to \( c \) and \( c' .\)

\subsection{Notation and standing assumptions}
\label{section: notation}

We assume that \( m,n \geq 2 \) and \( N \geq 1 \) are given. 
Further, we assume that \( \beta,\kappa \geq 0 \) are given, and that \( G = \mathbb{Z}_n. \)

For any \( k \in \mathbb{N}, \) we define \( [k] \coloneqq \{ 0,1, \dots, k-1 \}. \)

We assume that arbitrary orderings of the vertices in \( C_0(B_N), \) the edges in \( C_1(B_N), \) and the plaquettes in \( C_2(B_N) \) have been fixed.

\section{Some useful functions and their properties}\label{section: useful functions}

Recall the definitions of \( \psi_a,\) \( \hat \varphi_a,\) and \( \varphi_a \) from~Section~\ref{subsec: functions}.
To simplify notation, we extend the definition of  \( \hat \varphi_a \) to \( \mathbb{Z} \) by letting \( \hat \varphi_a(j) \coloneqq \hat \varphi_a(j') \) whenever \( j' \in [n] \) and \( j = j' + kn \) for some \( k \in \mathbb{Z}. \)
We use these definitions also for \(  \mathbb{Z}_n \) by identifying \( \mathbb{Z}_n\) with \( [n]\) in the natural way.
Next, define 
\begin{equation}\label{eq: def eta kappa}
    \eta_a \coloneqq \min_{j \in [n]} \varphi_a( j+1) /\varphi_a( j),
\end{equation}
\begin{equation}\label{eq: zeta def} 
    \zeta_a \coloneqq \sum_{j \in \mathbb{Z}_n \smallsetminus \{ 0 \}}  \varphi_a(j) \quad \text{and} \quad  \xi_a  \coloneqq \max_{j \in \mathbb{Z}_n \smallsetminus \{ 0 \}}  \varphi_a(j),
\end{equation} 
and let
\begin{equation}\label{eq: def alpha}
    \alpha(\beta,\kappa) \coloneqq
        \frac{\sum_{j \in [n]} \varphi_\beta(j)\varphi_\kappa(j)^4  \cdot \frac{\varphi_\kappa(  j +1)}{ \varphi_\kappa(j)\varphi_\kappa(1)}}{ \sum_{j \in [n]} \varphi_\beta(j)\varphi_\kappa(j)^4 } .
\end{equation}

\begin{remark}\label{remark: Z2}
    When \( G = \mathbb{Z}_2, \) we have \( \psi_a(0) = \cosh(a), \) \( \psi_a(1) = \sinh(a), \) \( \hat \varphi_a(0) = \cosh(2a), \) \( \hat\varphi_a(1) = \sinh(2a), \) \( \varphi_a(0)=1, \) and \( \varphi_a(1) = \tanh(2a), \) and consequently, \( \eta_a = \zeta_a = \xi_a =  \tanh(2a)\) and \( \alpha(\beta,\kappa) = (1 + \tanh(2\beta)\tanh(2\kappa)^2)/(1 + \tanh(2\beta)\tanh(2\kappa)^4). \)
    
    We remark that the assumption, \( \kappa\bigl( 1+\varepsilon_\kappa \bigr) \leq 1, \) which is made in some of the lemmas of this section, can easily be shown to not be needed when \( G = \mathbb{Z}_2. \)  
\end{remark}

We now state and prove a few lemmas that describe properties of the above functions.

\begin{lemma}\label{lemma: alpha when beta is zero}
    For any \( a \geq 0, \) we have  \( \alpha(0,a) = 1. \)
\end{lemma}

\begin{proof}
    For any \(j \in [n] \) we have 
    \( \psi_0(j) = \mathbb{1}(j=0), \) and hence \( \hat \varphi_0(j) = \varphi_0(j) = \mathbb{1}(j=0). \) This implies in particular that 
    \begin{equation*}
        \alpha(0,a) = 
        \frac{\sum_{j \in [n]} \mathbb{1}(j=0) \cdot \varphi_a(j)^4  \cdot \frac{\varphi_a(  j +1)}{ \varphi_a(j)\varphi_a(1)}}{ \sum_{j \in [n]} \mathbb{1}(j=0) \cdot \varphi_a(j)^4 }
        =
        \frac{\varphi_a(0)^4 \cdot \frac{\varphi_a(1)}{\varphi_a(0) \varphi_a(1)}}{\varphi_a(0)^4}
        =
        1
    \end{equation*} 
    as desired. 
\end{proof} 

The next expresses an exponential function as a sum, and using this lemma will later be the first step of the high-temperature expansion.
\begin{lemma}\label{lemma: first part of expansion}
    Let \( a \geq 0 \) and \( g \in \mathbb{Z}_n. \) Then 
    \begin{equation*}
        e^{2a \Re \rho(g) } = \sum_{j \in [n]} \rho(g)^{j} \hat \varphi_a(j).
    \end{equation*}
\end{lemma}

\begin{proof}
    We have
    \begin{equation*}
        e^{a \rho(g)} 
        = \sum_{k = 0}^\infty  \frac{a^k \rho(g)^k}{k!} 
        = \sum_{k' = 0}^{n-1}\sum_{k'' = 0}^\infty  \frac{a^{k'+k''n} \rho(g)^{k'+k''n}}{(k'+k'' n)!}.
    \end{equation*}
    Since \( \rho(g)^n = 1, \) we have
    \begin{equation*}
        \sum_{k' = 0}^{n-1}\sum_{k'' = 0}^\infty  \frac{a^{k'+k''n} \rho(g)^{k'+k''n}}{(k'+k'' n)!}
        = \sum_{k' = 0}^{n-1} \rho(g)^{k'}\sum_{k'' = 0}^\infty  \frac{a^{k'+k''n} }{(k'+k'' n)!}
        = \sum_{k' = 0}^{n-1} \rho(g)^{k'}\psi_a(k'),
    \end{equation*}
    and hence
    \begin{equation*}
        e^{a \rho(g)} = \sum_{k' = 0}^{n-1} \rho(g)^{k'}\psi_a(k').
    \end{equation*}
    Since \( \psi_a \) is a real-valued function and \( \rho \) is unitary, it follows that
    \begin{align*}
        &e^{2a \Re \rho(g)} 
        =
        \bigl|e^{a  \rho(g)} \bigr|^2
        = 
        \Bigl|\sum_{k' = 0}^{n-1} \rho(g)^{k'}\psi_a(k')\Bigr|^2
        =
        \sum_{k = 0}^{n-1}\rho(g)^k\psi_a(k)
        \sum_{k' = 0}^{n-1}\rho(g)^{-k'}\psi_a(k')
        \\&\qquad=
        \sum_{k = 0}^{n-1}\sum_{k' = 0}^{n-1} \rho(g)^{k-k'}\psi_a(k)\psi_a(k')
        =
        \sum_{j \in [n]} \rho(g)^{j} \sum_{\substack{k,k' \in [n] \mathrlap{\colon}\\ k-k' =  j \mod n}} \psi_a(k)\psi_a(k')
        \\&\qquad=
        \sum_{j \in [n]} \rho(g)^j\hat \varphi_a(j),
    \end{align*}
    which is the desired conclusion.    
\end{proof}

The next lemma shows that \( \hat \varphi \) is symmetric in a certain sense. 
\begin{lemma}\label{lemma: symmetry}
    Let \( a \geq 0 \) and \( j \in [n]. \) Then
    \(  \hat \varphi_a(n-j) = \hat \varphi_a(j). \)
\end{lemma}

\begin{proof}
    For any \( k,k' \in [n], \) we have
    \begin{equation*}
        k-k' = n-j \mod n\Leftrightarrow k'-k = j-n \mod n\Leftrightarrow k'-k = j \mod n.
    \end{equation*}
    Using this observation, it follows that
    \begin{equation*}
        \hat \varphi_a(n-j) 
        = \sum_{k,k' \in [n] \colon k-k' = n-j \mod n} \psi_a(k)\psi_a(k') 
        = \sum_{k,k' \in \mathbb{Z}_n \colon k'-k = j} \psi_a(k)\psi_a(k') = \hat \varphi_a(j)
    \end{equation*}
    as desired.
\end{proof}

\begin{lemma}\label{lemma: 0 first}
    Let \( a \geq 0 \) and  \( j \in [n] \smallsetminus \{ 0 \}.\)  Then  \( \hat \varphi_a(j) < \hat \varphi_a(0), \) and hence \( \varphi_a(j) < \varphi_a(0) = 1. \)
\end{lemma}

\begin{proof}
    For \( m \in [n], \) define
    \begin{equation*}
        \mathbf{v}^m \coloneqq \bigl(\psi_a(m), \psi_a(m+1), \dots, \psi_a(n-1), \psi_a(0), \dots, \psi_a(m-1) \bigr),
    \end{equation*} 
    and note that for any \( m,m' \in [n] \) with \( m \geq m', \) we have \( \langle \mathbf{v}^m,\mathbf{v}^{m'} \rangle= \hat \varphi_a(m-m'). \) 
    Using the Cauchy-Schwarz inequality, we deduce that
    \begin{equation*}
        \hat \varphi_a(j)^2 = \langle \mathbf{v}^0, \mathbf{v}^j \rangle^2 <  \langle \mathbf{v}^0, \mathbf{v}^0 \rangle\langle \mathbf{v}^j, \mathbf{v}^j \rangle
        =
        \langle \mathbf{v}^0, \mathbf{v}^0 \rangle^2 = \hat \varphi_a(0)^2.
    \end{equation*}
    Since \( \hat \varphi_a(j) \geq 0 \) and \( \hat \varphi_a(0) \geq 0, \)  it follows that \( \hat \varphi_a(j) < \hat \varphi_a(0). \)
    Using the definition of \( \varphi_a , \) it immediately follows that  \(  \varphi_a(j) <  \varphi_a(0) = 1. \) This concludes the proof.
\end{proof}

The next lemma gives the leading order term of \( \hat \varphi_a(j) \) when \( a \) is small.
\begin{lemma}\label{lemma: upper and lower bounds}
    Let  \( a \in [0,1] \) and \( j \in [n]\) be such that \( j \leq n/2. \) Then
    \begin{equation*}
        \begin{split}
            &0<  \hat \varphi_a(j) - \bigl(1 + \mathbb{1}(j = n/2 )\bigr) a^{j}/j!  
            \leq a^{j}/j! \cdot \varepsilon_a,
        \end{split}
    \end{equation*}
    where
    \begin{equation}\label{eq: def epsilon}
        \varepsilon_a \coloneqq ( 1
            +
            2ae^a  ) \Bigl( 1
            +
            \frac{a^{n}e^a}{n!} \Bigr)^2-1.
    \end{equation}
\end{lemma}

\begin{proof}
    By definition, we have 
    \begin{equation*}
        \psi_a(j) = \sum_{k=0}^\infty \frac{a^{j+kn}}{(j+kn)!} > \frac{a^j}{j!}.
    \end{equation*}
    Consequently, for $0 \leq j \leq n/2$,
    \begin{equation*}
        \begin{split}
            &\hat \varphi_a(j) 
            = \sum_{\substack{k,k' \in [n] \colon\\ k-k' = j \mod n}} \psi_a(k)\psi_a(k') 
            > 
            \sum_{\substack{k,k' \in [n] \colon\\ k-k' = j \mod n}} \frac{a^{k+k'}}{k!k'!}
            \geq
            \frac{a^{j+0}}{j!0!}
        + \mathbb{1}(j > 0) \frac{a^{0+(n-j)}}{0!(n-j)!}
            \\&\qquad \geq
            \frac{a^{j}}{j!} \cdot \bigl( 1 
            + \mathbb{1}(j = n/2) \bigr).
        \end{split}
    \end{equation*}
    This completes the proof of the lower bound.
    For the upper bound, note that
    \begin{equation*}
        \begin{split}
            &\psi_a(j) = \sum_{k=0}^\infty \frac{a^{j+kn}}{(j+kn)!}
            = 
            \frac{a^j}{j!} 
            +
            \sum_{k=1}^\infty \frac{a^{j+kn}}{(j+kn)!}
            \\&\qquad= 
            \frac{a^j}{j!} 
            +
            \frac{a^{j+n}}{(j+n)!}\sum_{k=1}^\infty \frac{a^{(k-1)n}}{(j+n + (k-1)n)!/(j+n)!}
            \\&\qquad\leq
            \frac{a^j}{j!} 
            +
            \frac{a^{j+n}}{(j+n)!}\sum_{k=1}^\infty \frac{a^{(k-1)n}}{( (k-1)n)!}
            \leq
            \frac{a^j}{j!} 
            +
            \frac{a^{j+n}}{(j+n)!}\sum_{k'=0}^\infty \frac{a^{k'}}{k'!}
            =
            \frac{a^j}{j!} 
            +
            \frac{a^{j+n}}{(j+n)!} \cdot e^a.
        \end{split}
    \end{equation*}
    As a consequence,
    \begin{equation*}
        \begin{split}
            &\hat \varphi_a(j) 
            = 
            \sum_{\substack{k,k' \in [n] \colon\\ k-k' = j \mod n}} \psi_a(k)\psi_a(k') 
            \leq
            \sum_{\substack{k,k' \in [n] \colon\\ k-k' = j \mod n}}
            \Bigl( \frac{a^k}{k!} 
            +
            \frac{a^{k+n}}{(k+n)!} \cdot e^a \Bigr) \Bigl(\frac{a^{k'}}{k'!}
            +
            \frac{a^{k'+n}}{(k'+n)!} \cdot e^a \Bigr)
            \\&\qquad = 
            \sum_{\substack{k,k' \in [n] \colon\\ k-k' = j \mod n}}
            \frac{a^{k+k'}}{k!k'!}\Bigl( 1
            +
            \frac{a^{n}k!}{(k+n)!} \cdot e^a \Bigr) \Bigl(1
            +
            \frac{a^{n} k'!}{(k'+n)!} \cdot e^a \Bigr)
            \leq
            \sum_{\substack{k,k' \in [n] \colon\\ k-k' = j \mod n}}
            \frac{a^{k+k'}}{k!k'!}\Bigl( 1
            +
            \frac{a^{n}e^a}{n!} \Bigr)^2.
        \end{split}
    \end{equation*}
    Now fix any \( j' \geq 0. \) Then
    \begin{equation}\label{eq: first three equalities}
        \begin{split}
            &\sum_{\substack{k,k' \in [n] \colon\\ k-k' = j \mod n}}
            \frac{a^{k+k'}}{k!k'!}
            =
            \sum_{k = 0}^{j-1}
            \frac{a^{k+(k-j+n)}}{k!(k-j+n)!}
            +
            \sum_{k = j}^{n-1}
            \frac{a^{k+(k-j)}}{k!(k-j)!}
            \\&\qquad
            = 
            \bigl( 1 + \mathbb{1}(j = n/2) \bigr) \frac{a^j}{j!} 
            +
            \sum_{k = 0+ \mathbb{1}(j=n/2)}^{j-1}
            \frac{a^{k+(k-j+n)}}{k!(k-j+n)!}
            +
            \sum_{k = j+1}^{n-1}
            \frac{a^{k+(k-j)}}{k!(k-j)!}.
        \end{split}
    \end{equation}
   Moreover, if \( k \in \{ j+1, \dots, n-1 \}, \) then 
    \begin{enumerate}[label=(\roman*)]
        \item \( k + (k-j)
        \geq k =  (j+1) + (k-j-1), \) and hence, since \( a \leq 1, \) we have \( a^{k+(k-j)} \leq  a^{(j+1) + (k-j-1)},\) and
        \item \( k!(k-j)! \geq k! = ((j+1) + (k-j-1))!. \)
    \end{enumerate}
    Next, note that if  \( 0 \leq j < n/2 \) and 
    \( k \in \{ 0,\dots, j-1 \}, \) then 
    \begin{enumerate}[label=(\roman*)]
        \item \( k + (k-j+n) = (n-j)+2k \geq j+1+ 2 k  \geq (j+1) + k, \) and hence, since \( a \leq 1 \) we have \( a^{k + (k-j+n)} \leq a^{(j+1) + k}, \) and
        \item \( k! (k-j+n)! \geq  \bigl(k +(n-j)\bigr)! \geq \bigl(k+(j+1)\bigr)! = \bigl((j+1)+k \bigr). \)
    \end{enumerate} 
    Finally, note that if \( j = n/2 \) and \( k \in \{ 1,2, \dots, j-1 \} \), then 
    \begin{enumerate}[label=(\roman*)]
        \item \( k + (k-j+n) = 2k + (n-j) = 2k + j = (j+1) + (2k-1) \geq (j+1) + (k-1)  \) and hence, since \( a \leq 1 \) we have \( a^{k + (k-j+n)} \leq a^{(j+1)+(k-1)}, \) and
        \item \( k! (k-j+n)! \geq  (k-j+n)! = \bigl( k + (n-j)\bigr)! = (k+j)! = \bigl( (j+1) + (k-1) \bigr)! . \)
    \end{enumerate}
    Using these observations, we can bound the right-hand side of~\eqref{eq: first three equalities} from above by
    \begin{equation*}
        \begin{split} 
            \bigl( 1 + \mathbb{1}(j = n/2) \bigr) \frac{a^j}{j!} 
            +
            2\sum_{k = j+1}^{\infty}
            \frac{a^{k}}{k!}
            \leq 
            \bigl( 1 + \mathbb{1}(j = n/2) \bigr) \frac{a^j}{j!} 
            +
            \frac{2a^{j+1}e^a}{(j+1)!}.
        \end{split}
    \end{equation*}
    Together with the previous equation, we thus obtain
    \begin{equation*}
        \begin{split}
            &\hat \varphi_a(j) 
            \leq 
            \biggl( \bigl( 1 + \mathbb{1}(j = n/2) \bigr)  \frac{a^{j}}{j!}
            +
            \frac{2a^{j+1}e^a}{(j+1)!} \biggr) \Bigl( 1
            +
            \frac{a^{n}e^a}{n!} \Bigr)^2
            \\&\qquad 
            \leq 
            \bigl( 1 + \mathbb{1}(j = n/2) \bigr)   \frac{a^{j}}{j!} \cdot ( 1
            +
            2ae^a  ) \Bigl( 1
            +
            \frac{a^{n}e^a}{n!} \Bigr)^2.
        \end{split}
    \end{equation*}
    This concludes the proof of the upper bound.
\end{proof}

\begin{lemma}\label{lemma: order}
    Let \( a \geq 0 \) be such that \( a( 1 +\varepsilon_a )  \leq 1. \)
    Then \( \hat\varphi_a(1) \geq \hat \varphi_a(2) \geq \dots \geq \hat \varphi_a\bigl(\lfloor n/2 \rfloor \bigr). \)
\end{lemma}

\begin{proof}
    Fix any integer \( 2 \leq j \leq n/2.\) Then, by Lemma~\ref{lemma: upper and lower bounds}, we have
    \begin{equation*}
        \begin{split}
            & \bigl(1 + \mathbb{1}(j = n/2 )\bigr) a^{j}/j! 
            \leq  
            \hat \varphi_a(j)   
            \leq 
            \bigl(1 + \mathbb{1}(j = n/2 )\bigr) a^{j}/j!  + a^{j}/j! \cdot \varepsilon_a.
        \end{split}
    \end{equation*}
    Consequently, if 
    \begin{equation*}
        \bigl( (1 + \mathbb{1}(j = n/2 ))+\varepsilon_a \bigr) \cdot a/j \leq 1,
    \end{equation*}
    then  \( \hat \varphi_a(j) \leq \hat \varphi_a(j-1). \) Since \( a( 1 + \varepsilon_a ) \leq 1 \) by assumption, and 
    \begin{equation*}
        \pigl( \bigl(1 + \mathbb{1}(j = n/2)\bigr)+\varepsilon_a \pigr) \cdot a/j \leq a( 1 + \varepsilon_a )
    \end{equation*}
    when \( j \in \{ 2,3, \dots, \lfloor n/2 \rfloor\}, \) it follows that \( \hat \varphi_a(1) \geq \hat \varphi_a(2) \geq \dots \geq \hat \varphi_a(\lfloor n/2 \rfloor).  \) This concludes the proof.
\end{proof}

\begin{lemma}\label{lemma: convexity type inequality}
    Let  \( a \geq 0 \) be such that \( a  (1+\varepsilon_a) \leq 1, \) 
    and let \( j \in [n].\)  
    Then
    \begin{equation}\label{eq: convexity type inequality}
        \hat\varphi_a(  j +1) \hat\varphi_a(0)+\hat\varphi_a(  j -1) \hat\varphi_a(0) \geq 2\hat \varphi_a(j) \hat \varphi_a(1).
    \end{equation}
\end{lemma}

\begin{proof}
    We will prove that~\eqref{eq: convexity type inequality} holds by considering different cases.
    
    If \( j = 0, \) then~\eqref{eq: convexity type inequality} reduces to
    \begin{equation*} 
        \hat\varphi_a( 1) \hat\varphi_a(0)+\hat\varphi_a(  -1) \hat\varphi_a(0) \geq 2\hat \varphi_a(0) \hat \varphi_a(1).
    \end{equation*}
    Since, by Lemma~\ref{lemma: symmetry}, we have \( \hat \varphi_a(1) = \hat \varphi_a(-1) \), the desired conclusion follows.
    
    If \( j = n/2, \) then, by Lemma~\ref{lemma: symmetry}, the left hand side of~\eqref{eq: convexity type inequality} is equal to \( 2 \hat \varphi_a(j-1) \hat \varphi(0), \) and hence in this case the desired conclusion follows by combining Lemma~\ref{lemma: order} and Lemma~\ref{lemma: 0 first}.
    
    If \( j = \lfloor n/2 \rfloor < n/2, \) then, by Lemma~\ref{lemma: symmetry}, the left hand side of~\eqref{eq: convexity type inequality} is equal to \( \hat \varphi_a(j) \hat \varphi(0)+\hat \varphi_a(j-1) \hat \varphi(0), \) and hence in this case the desired conclusion follows by combining Lemma~\ref{lemma: order} and Lemma~\ref{lemma: 0 first}.

    If \( j \in \bigl\{ \lceil n/2 \rceil, \dots, n-1 \bigr\}, \) then \( n-j \in \bigl\{ 1,2, \dots, \lfloor n/2 \rfloor \bigr\}. \) Using Lemma~\ref{lemma: symmetry}, the desired conclusion will thus follow in this case if we can show that~\eqref{eq: convexity type inequality} holds for \( j \in \bigl\{ 1,2,\dots, \lfloor n/2 \rfloor \bigr\}. \) %
    
    Assume that \( j \in \bigl\{ 1,2,\dots, ,\lfloor n/2 \rfloor-1 \bigr\}. \) This implies that \( n \geq 4, \) which in turn implies that \( 1 \neq n/2 \) and \( j \neq n/2. \)
    Thus, using Lemma~\ref{lemma: upper and lower bounds}, we have
    \begin{equation*}
        \hat \varphi_a(j)\hat \varphi_a(1) \leq a^{j+1}/j! \bigl( 1  +   \varepsilon_a \bigr)^2
    \end{equation*}
    and
    \begin{equation*}
        \hat\varphi_a(  j +1) \hat\varphi_a(0)+\hat\varphi_a(  j -1) \hat\varphi_a(0)
        \geq\hat\varphi_a(  j -1) \hat\varphi_a(0)
        \geq 
        a^{j-1}/(j-1)! .
    \end{equation*}
    Consequently, if \( a^{j-1}/(j-1)! \geq a^{j+1}/j! \bigl(1+\varepsilon_a\bigr)^2 \) then~\eqref{eq: convexity type inequality} holds. But this is an immediate consequence of the assumption \( a(1 + \varepsilon_a) \leq 1. \) This concludes the proof.
\end{proof}

The next lemma gives an upper bound on \( \alpha(\beta,\kappa) \) in terms of \( \zeta_\beta \) and \( \xi_\kappa, \) and will be used later when giving upper bounds on the error terms we will get from approximations. 
\begin{lemma}\label{lemma: Zn properties of alpha}
    If \( \kappa( 1+\varepsilon_\kappa ) \leq 1, \)
    then 
    \(
        1 \leq \alpha(\beta,\kappa) \leq  {1}/({1 - \zeta_\beta\xi_\kappa^2}).
    \)
\end{lemma}
\begin{proof}
    Assume that \( \kappa( 1+\varepsilon_\kappa ) \leq 1. \)
    By definition, we have
    \begin{equation*}
        \begin{split}
            &\alpha(\beta,\kappa)
            =
            \frac{\sum_{j \in [n]} \hat\varphi_\beta(j) \hat\varphi_\kappa(j)^3 \hat\varphi_\kappa(  j +1) \hat\varphi_\kappa(0) }{ \sum_{j \in [n]} \hat\varphi_\beta(j) \hat\varphi_\kappa(j)^4 \hat\varphi_\kappa(1)}
            \\&\qquad=
            1 + \frac{\sum_{j \in [n]\smallsetminus \{ 0 \}} \hat\varphi_\beta(j) \hat\varphi_\kappa(j)^3 \bigl( \hat\varphi_\kappa(  j +1) \hat\varphi_\kappa(0)  - \hat \varphi_\kappa(j) \hat \varphi_\kappa(1)  \bigr)}{ \sum_{j \in [n]} \hat\varphi_\beta(j) \hat\varphi_\kappa(j)^4\hat\varphi_\kappa(1) }.
        \end{split}
    \end{equation*}
    Since, by Lemma~\ref{lemma: symmetry}, we have \( \hat \varphi_a(j) = \hat \varphi_a(n-j)\) for every \( a \geq 0 \) and \( j \in [n], \) we have
    \begin{align*}
        &\sum_{j \in [n]\smallsetminus \{ 0 \}} \hat\varphi_\beta(j) \hat\varphi_\kappa(j)^3 \bigl( \hat\varphi_\kappa(  j +1) \hat\varphi_\kappa(0)  - \hat \varphi_\kappa(j) \hat \varphi_\kappa(1)  \bigr)
        \\&\qquad=
        \sum_{j \in [n]\smallsetminus \{ 0 \}} \hat\varphi_\beta(n-j) \hat\varphi_\kappa(n-j)^3 \bigl( \hat\varphi_\kappa(  (n-j) +1) \hat\varphi_\kappa(0)  - \hat \varphi_\kappa(n-j) \hat \varphi_\kappa(1)  \bigr)
        \\&\qquad=
        \sum_{j \in [n]\smallsetminus \{ 0 \}} \hat\varphi_\beta(j) \hat\varphi_\kappa(j)^3 \bigl( \hat\varphi_\kappa( j-1) \hat\varphi_\kappa(0)  - \hat \varphi_\kappa(j) \hat \varphi_\kappa(1)  \bigr)
    \end{align*}
    and hence it follows that
    \begin{equation}\label{eq: alpha equality}
        \begin{split}
            &\alpha(\beta,\kappa)
            =
            1 + \frac{\sum_{j \in [n] \smallsetminus \{ 0 \}} \hat\varphi_\beta(j) \hat\varphi_\kappa(j)^3 \bigl( \hat\varphi_\kappa(  j +1) \hat\varphi_\kappa(0)+\hat\varphi_\kappa(  j -1) \hat\varphi_\kappa(0)  - 2\hat \varphi_\kappa(j) \hat \varphi_\kappa(1)  \bigr) }{ 2\sum_{j \in [n]} \hat\varphi_\beta(j) \hat\varphi_\kappa(j)^4\hat\varphi_\kappa(1) }.
        \end{split}
    \end{equation}
    Applying Lemma~\ref{lemma: convexity type inequality}, we obtain \( \alpha(\beta,\kappa) \geq 1 \) as desired.
    
    Next, note that
    \begin{equation}\label{eq: second step}
        1 - \alpha(\beta,\kappa)^{-1}
        =
        \frac{\alpha(\beta,\kappa)-1}{\alpha(\beta,\kappa)}
        =\frac{\sum_{j \in \mathbb{Z}_n \smallsetminus \{ 0 \}} \varphi_\beta(j) \varphi_\kappa(j)^3  \bigl( \varphi_\kappa(  j +1)\varphi_\kappa(0)- \varphi_\kappa(j) \varphi_\kappa(1)\bigr)}{\varphi_\kappa( 1) + \sum_{j \in \mathbb{Z}_n\smallsetminus\{ 0 \}} \varphi_\beta(j) \varphi_\kappa(j)^3  \varphi_\kappa(  j +1)} .
    \end{equation}
    By Lemma~\ref{lemma: 0 first} and the definition of \( \varphi_a, \) for every \( j \in \mathbb{Z}_n \) we have \( 0 < \varphi_\kappa(j) \leq \varphi_\kappa(0)=1. \)  Consequently, we have
    \begin{equation*}
        \varphi_\kappa(  j +1)\varphi_\kappa(0)- \varphi_\kappa(j) \varphi_\kappa(1)  \leq 1 \quad \forall j \in [n],
    \end{equation*}
    and 
    \begin{equation*}
        \sum_{j \in [n]\smallsetminus\{ 0 \}} \varphi_\beta(j) \varphi_\kappa(j)^3  \varphi_\kappa(  j +1) \geq 0.
    \end{equation*}
    Using these observations, we can upper bound the right-hand side of~\eqref{eq: second step} by 
    \begin{equation}
        {\sum_{j \in [n] \smallsetminus \{ 0 \}} \varphi_\beta(j) \varphi_\kappa(j)^3 }/{\varphi_\kappa( 1)}
        \leq \zeta_\beta \xi_\kappa^3/\varphi_\kappa(1).
    \end{equation}
    Since, by combining Lemma~\ref{lemma: symmetry} and Lemma~\ref{lemma: order}, we have \( \xi_\kappa =  \varphi_\kappa(1), \) we obtain the desired conclusion.
\end{proof}

\begin{lemma}\label{lemma: eta relationships}
    Let \( a \geq 0, \) and let $\hat \eta_a$ be defined by~\eqref{hatetadef}. Then \( \hat \eta_a = \xi_a =  \varphi_a(1)\). Moreover, if \( n \in \{ 2,3 \}, \) then \( \hat \eta_a = \eta_a,\) and  if  \( n \geq 4, \) \( a(1+\varepsilon_a) \leq 1, \) and 
    \(  (1 + \varepsilon_a)\bigl(1 + \mathbb{1}(j+1 = n/2)+ \varepsilon_a\bigr) \leq j+1
     \) for some \(j \in \{ 1, \dots, \lfloor n/2 \rfloor -1\}, \) then \(\eta_a  < \hat \eta_a. \)
\end{lemma}

\begin{proof}
    By Lemma~\ref{lemma: first part of expansion}, we have 
    \begin{align*}
        &\sum_{g \in G} e^{2 a \Re \rho(g)} 
        = \sum_{g \in \mathbb{Z}_n}  \sum_{j \in [n]}\hat \varphi_a(j) \rho(g)^{j} 
        \\&\qquad=\sum_{j \in [n]}\hat \varphi_a(j) \sum_{g \in \mathbb{Z}_n}  \rho(g)^{j} 
        = \sum_{j \in [n]}\hat \varphi_a(j)  \cdot n\mathbb{1}(j=0) =  n\hat \varphi_a(0)
    \end{align*}
    and
    \begin{align*}
        \begin{split}
            &\sum_{g \in G} \rho(g) e^{2 a \Re \rho(g)}
            = \sum_{g \in \mathbb{Z}_n} \rho(g) \sum_{j \in [n]}\hat \varphi_a(j) \rho(g)^{j} 
            =\sum_{j \in [n]}\hat \varphi_a(j) \sum_{g \in \mathbb{Z}_n}  \rho(g)^{j+1} 
            \\&\qquad =\sum_{j \in [n]}\hat \varphi_a(j) \cdot n\mathbb{1}(j=-1) 
            = n\hat \varphi_a(-1) = n\hat \varphi_a(1).
        \end{split}
    \end{align*}
    This implies in particular that \begin{equation}\label{hatetaavarphi}
        \hat \eta_a = {\sum_{g \in G} \rho(g) e^{2 a \Re \rho(g)}}\Big/\, {\sum_{g \in G} e^{2 a \Re \rho(g)}} = {\hat \varphi_a(1)}/{\hat \varphi_a(0)} = \varphi_a(1),
    \end{equation} 
    and hence the first assertion of the lemma holds.
    
    If \( n = 2, \) then it follows immediately from Lemma~\ref{lemma: 0 first}
    that
    \begin{equation*}
        \hat \eta_a = {\hat \varphi_a(1)}/{\hat \varphi_a(0)} \leq {\hat \varphi_a(0)}/{\hat \varphi_a(1)},
    \end{equation*}
    and hence \( \hat \eta_a = \eta_a \) in this case.
    
    Next, if \( n = 3, \) then, by Lemma~\ref{lemma: symmetry} and Lemma~\ref{lemma: 0 first}, we have
    \begin{equation*}
         {\hat \varphi_a(2)}/{\hat \varphi_a(1)}
         = 
         {\hat \varphi_a(1)}/{\hat \varphi_a(1)}
         \geq  
         {\hat \varphi_a(1)}/{\hat \varphi_a(0)} = \hat \eta_a
    \end{equation*}
    and
    \begin{equation*}
         {\hat \varphi_a(0)}/{\hat \varphi_a(2)}
         = 
         {\hat \varphi_a(0)}/{\hat \varphi_a(1)}
         \geq  
         {\hat \varphi_a(1)}/{\hat \varphi_a(0)} = \hat \eta_a,
    \end{equation*}
    and hence \( \hat \eta_a = \eta_a \) in this case.

    Finally, assume that \( n \geq 4. \)
    By combining Lemma~\ref{lemma: symmetry}, Lemma~\ref{lemma: 0 first} and Lemma~\ref{lemma: order}, we see that \( \hat \varphi_a(j+1) \leq \hat \varphi_a(j) \) exactly when \( j \in \{ 0,1,\dots, \lfloor n/2 \rfloor-1 \}. \)
    Consequently, we must have
    \begin{equation*}
        \eta_a = \min_{j \in [n]} \varphi_a( j+1) /\varphi_a( j)
        =
        \min_{j \in \{ 0, 1, \dots,\lfloor n/2 \rfloor-1\} } \varphi_a( j+1) /\varphi_a( j).
    \end{equation*}
    Now note that for each  \( j \in \{ 0, \dots,\lfloor n/2 \rfloor-1\},\) by Lemma~\ref{lemma: upper and lower bounds}, we have
    \begin{equation*}
        \frac{a}{j+1} \cdot \frac{ 1 + \mathbb{1}(j+1 = n/2 ) }{1 + \varepsilon_a } \leq 
        \frac{\hat \varphi_a(j+1)}{\hat \varphi_a(j)}
        < 
        \frac{a}{j+1} \cdot \pigl(1 + \mathbb{1}(j+1 = n/2)+ \varepsilon_a\pigr)  .
    \end{equation*}
    Since $\hat \eta_a = \hat \varphi_a(1)/\hat \varphi_a(0)$ by ~\eqref{hatetaavarphi}, it follows that  \( \eta_a < \hat \eta_a \) if
    \begin{equation*}
        \frac{1}{j+1} \cdot \pigl(1 + \mathbb{1}(j+1 = n/2)+ \varepsilon_a\pigr) \leq \frac{1}{0+1} \cdot \frac{ 1 }{1 + \varepsilon_a }
    \end{equation*}
    for some \( j \in \{ 1, \dots,\lfloor n/2 \rfloor-1\}.\) This concludes the proof.
\end{proof}

We shall in several places consider two conditions on $\kappa$ and $\beta$ given by the following inequalities; in particular we shall need them for the proof of Theorem~\ref{theorem: first theorem}. We state the relations now since we will need to refer to them several times before the statement of Theorem~\ref{theorem: first theorem}:
\begin{equation}\label{assumption: 1}
     (16m)^2 \zeta_\beta < \xi_\kappa,
\end{equation}
and
\begin{equation}\label{assumption: 3}
     \kappa \bigl(2 + \varepsilon_\kappa \bigr) \leq 1.
\end{equation}
Here  \( \xi_\kappa, \) \( \zeta_\beta, \)  and \( \varepsilon_\kappa \) are defined in~\eqref{eq: def eta kappa1},~\eqref{eq: zeta def},~and~\eqref{eq: def epsilon} respectively.

\section{The high-temperature representation}\label{section: high-temperature expansion}

The main purpose of this section is to present the high-temperature representation of \( \mu_{N,\beta,\kappa}.\) We remark that this is well known in the literature (see, e.g.,~\cite{mms1979}), but include it here for completeness.
 
To simplify notation, for \( \omega \in \Omega^2(B_N,\mathbb{Z}_n) \), we define
\begin{equation}\label{eq: phi def forms} 
    \varphi(\omega) \coloneqq \varphi_{\beta,\kappa}(\omega) \coloneqq \prod_{e \in C_1(B_N)^+} 
            \varphi_\kappa\bigl( \delta\omega (e) \bigr)
            \prod_{p \in C_2(B_N)^+} 
            \varphi_\beta\bigl(\omega(p)\bigr)
\end{equation}
and
\begin{equation}\label{eq: hat L} 
    \widehat{L_\gamma}(\omega) \coloneqq \prod_{e \in C_1(B_N)^+} 
            \varphi_\kappa\bigl( \delta\omega(e) + \gamma[e] \bigr) \varphi_\kappa\bigl( \delta\omega  (e) \bigr)^{-1} = \prod_{e \in \support \gamma} 
            \varphi_\kappa\bigl( \delta\omega(e) + \gamma[e] \bigr) \varphi_\kappa\bigl( \delta\omega  (e) \bigr)^{-1}
\end{equation}
where the identity follows from the observation that if \( e \notin \support \gamma, \) then \( \gamma[e] = 0. \)
Next, we define the probability measure \( \mathbb{P}_\varphi \) on \( \Omega^2(B_N,\mathbb{Z}_n) \) by 
\begin{equation}\label{Pvarphidef}
    \mathbb{P}_\varphi\bigl(\{ \omega \}\bigr) \coloneqq
    \frac{\varphi(\omega)}{\sum_{\omega' \in \Omega^2(B_N,\mathbb{Z}_n)} \varphi(\omega')}, \quad \omega \in \Omega^2(B_N,\mathbb{Z}_n),
\end{equation}
and let \( \mathbb{E}_\varphi \) denote the corresponding expectation. 
We then have the following result.

\begin{proposition}\label{proposition: high-temperature expansion ALHM 3}
    Let \( \gamma \) be a path. Then
    \begin{equation*} 
        \begin{split}
            &\mathbb{E}_{N,\beta,\kappa}\bigl[ L_\gamma(\sigma)\bigr]
            =
            \mathbb{E}_\varphi \bigl[  \widehat{L_\gamma}(\omega) \bigr]
            .
        \end{split}
    \end{equation*} 
\end{proposition}

\begin{proof}
    Fix any \( \sigma \in \Omega^1(B_N,\mathbb{Z}_n). \) Then, by Lemma~\ref{lemma: first part of expansion}, 
    \begin{align*}
        &e^{ \kappa \sum_{e \in C_1(B_N)} \rho(\sigma( e))}
        =
        \prod_{e \in C_1(B_N)^+} e^{ 2\kappa  \Re \rho(\sigma( e))}
        =
        \prod_{e \in C_1(B_N)^+}  \sum_{j \in [n]} \rho\bigl( \sigma(e)\bigr)^{j} \hat \varphi_\kappa(j)
        \\&\qquad=
        \sum_{\omega' \in C^1(B_N,\mathbb{Z}_n)^+} \prod_{e \in C_1(B_N)^+}  \rho\bigl( \sigma(e)\bigr)^{\omega'(e)} \hat \varphi_\kappa\bigl( \omega'(e) \bigr).
    \end{align*}
    Completely analogously, we obtain
    \begin{equation*}
        \begin{split}
            & e^{ \beta \sum_{p \in C_2(B_N)} \rho(d\sigma( p))}
            = 
            \sum_{\omega \in \Omega^2(B_N,\mathbb{Z}_n)} \prod_{p \in C_2(B_N)^+} \hat \varphi_\beta\bigl(\omega(p)\bigr) \rho\bigl( d\sigma(p) \bigr)^{\omega(p)}. 
        \end{split}
    \end{equation*} 
    Combining the previous two equations, we find
    \begin{equation*}
        \begin{split}
            &
            L_\gamma(\sigma) e^{ \beta \sum_{p \in C_2(B_N)} \rho(d\sigma( p)) + \kappa \sum_{e \in C_1(B_N)} \rho(\sigma( e))}=  \rho\bigl( \gamma(\sigma) \bigr) e^{ \beta \sum_{p \in C_2(B_N)} \rho(d\sigma( p)) + \kappa \sum_{e \in C_1(B_N)} \rho(\sigma( e))}
            \\&\qquad=
            \rho\bigl( \gamma(\sigma) \bigr)  \sum_{\omega' \in \Omega^1(B_N,\mathbb{Z}_n)}
            \sum_{\omega \in \Omega^2(B_N,\mathbb{Z}_n)} \prod_{e \in C_1(B_N)^+} \hat \varphi_{{\kappa}}\bigl( \omega'(e)\bigr) \rho\bigl( \sigma(e) \bigr)^{\omega'(e)}
            \\
 & \qquad\qquad \times             \prod_{p \in C_2(B_N)^+} \hat \varphi_\beta\bigl(\omega(p)\bigr) \rho\bigl( d\sigma(p) \bigr)^{\omega(p)}
            \\&\qquad=
            \sum_{\omega' \in \Omega^1(B_N,\mathbb{Z}_n)}
            \sum_{\omega \in \Omega^2(B_N,\mathbb{Z}_n)} \prod_{e \in C_1(B_N)^+} \hat \varphi_{{\kappa}}\bigl( \omega'(e)\bigr)
             \prod_{p \in C_2(B_N)^+} \hat \varphi_\beta\bigl(\omega(p)\bigr) 
                         \\
 & \qquad\qquad \times  
             \prod_{e \in C_1(B_N)^+}  \rho\bigl( \sigma(e) \bigr)^{\gamma[e] + \omega'(e) + \delta \omega(e)},
        \end{split}
    \end{equation*}
    where we have used the identity
    $$\rho\bigl( \sigma(\gamma) \bigr) \prod_{e \in C_1(B_N)^+}\rho\bigl(\sigma(e)\bigr)^{\omega'(e)}
            \prod_{p \in C_2(B_N)^+}\rho\bigl(d\sigma(p)\bigr)^{\omega(p)}
            =
            \prod_{e \in C_1(B_N)^+}\rho\bigl( \sigma(e)\bigr)^{\gamma[e] +\omega'(e)+ \delta\omega(e) }$$
    to obtain the second equality.
    Now fix any \( e_0 \in C_1(B_N)^+.\) For  \(g \in G, \)  let \( \sigma_g \in \Omega^2(B_N,\mathbb{Z}_n)\) agree with \( \sigma \) everywhere except possibly on the set \( \{e_0,-e_0 \} ,\) where \( \sigma_g(e_0) = g. \)
    Then, for any \( \omega' \in \Omega^1(B_N,\mathbb{Z}_n), \) we have 
    \begin{equation*}
        \begin{split}
            &\prod_{e \in C_1(B_N)^+}\rho\bigl(\sigma_g(e)\bigr)^{\omega'(e)}
            =
            \rho(g)^{\omega'(e_0)}  \prod_{e \in C_1(B_N)^+}\rho\bigl(\sigma_0(e)\bigr)^{\omega'(e)}
        \end{split}
    \end{equation*} 
    and hence, for any \( \omega' \in \Omega^1(B_N,\mathbb{Z}_n) \) and \( \omega \in \Omega^2(B_N,\mathbb{Z}_n) \),
    \begin{equation*} 
        \begin{split}
            &
            \prod_{e \in C_1(B_N)^+}\rho\bigl( \sigma_g(e)\bigr)^{\gamma[e] +\omega'(e)+ \delta\omega(e) }
            \\&\qquad=  
            \rho( g)^{\gamma[e_0] +\omega'(e_0)+ \delta\omega(e_0) } \prod_{e \in C_1(B_N)^+}\rho\bigl( \sigma_0(e)\bigr)^{\gamma[e] +\omega'(e)+ \delta\omega(e) }
            .
        \end{split} 
    \end{equation*}
    If \( \gamma[e_0] +\omega'(e_0)+ \delta\omega(e_0) \neq 0, \) then \( \sum_{g \in G} \rho( g )^{\gamma[e_0] +\omega'(e_0) + \delta \omega(e_0)}  = 0 , \) and hence 
    \begin{align}  
        &\label{eq: the important term}
        \sum_{\sigma \in \Omega^1(B_N,\mathbb{Z}_n)}
        \prod_{e \in C_1(B_N)^+}\rho\bigl( \sigma(e)\bigr)^{\gamma[e] +\omega'(e)+ \delta\omega(e) }
        \\&\qquad=\nonumber
        \sum_{\substack{\sigma \in \Omega^1(B_N,\mathbb{Z}_n) \mathrlap{\colon}\\ \sigma(e_0)=0}}\,\, \sum_{g \in G}
        \prod_{e \in C_1(B_N)^+}\rho\bigl( \sigma_g(e)\bigr)^{\gamma[e] +\omega'(e)+ \delta\omega(e) }
        \\&\qquad=\nonumber
        \sum_{\substack{\sigma \in \Omega^1(B_N,\mathbb{Z}_n) \mathrlap{\colon}\\ \sigma(e_0)=0}}\,\,\sum_{g \in G}
        \rho( g )^{\gamma[e_0] +\omega'(e_0) + \delta \omega(e_0)}  \prod_{e \in C_1(B_N)^+}\rho\bigl(\sigma(e)\bigr)^{\gamma[e] +\omega'(e) + \delta \omega(e)} 
        = 0. 
    \end{align}
    Consequently, we get a non-zero contribution to~\eqref{eq: the important term} only from \( \omega' \) and \( \omega \) such that 
    \begin{equation*}
        \gamma[e_0] +\omega'(e_0)+ \delta\omega(e_0)=0 \quad \forall e_0 \in C_1(B_N)^+ .
    \end{equation*}
    But if this holds, then for any \(\sigma\in \Omega^1(B_N,\mathbb{Z}_n) \)  we have
    \begin{equation*}
            \prod_{e \in C_1(B_N)^+}\rho\bigl( \sigma(e)\bigr)^{\gamma[e] +\omega'(e)+ \delta\omega(e) }
            = 1.
    \end{equation*} 
    Combining the above observations, it follows that
        \begin{equation*}
        \begin{split}
            &
            \sum_{\sigma \in \Omega^1(B_N,\mathbb{Z}_n)} L_\gamma(\sigma)  e^{ \beta \sum_{p \in C_2(B_N)} \rho(d\sigma( p)) + \kappa \sum_{e \in C_1(B_N)} \rho(\sigma( e))}
            \\&\qquad= 
            \bigl|\Omega^1(B_N,\mathbb{Z}_n) \bigr|
            \sum_{\omega \in \Omega^2(B_N,\mathbb{Z}_n)} \prod_{e \in C_1(B_N)^+} \hat \varphi_\kappa\bigl( (-\delta \omega - \gamma)(e)\bigr)
             \prod_{p \in C_2(B_N)^+} \hat \varphi_\beta\bigl(\omega(p)\bigr)
            \\&\qquad= 
            \bigl|\Omega^1(B_N,\mathbb{Z}_n) \bigr|
            \sum_{\omega \in \Omega^2(B_N,\mathbb{Z}_n)} \prod_{e \in C_1(B_N)^+} \hat \varphi_\kappa\bigl( (\delta \omega + \gamma)(e)\bigr)
             \prod_{p \in C_2(B_N)^+} \hat \varphi_\beta\bigl(\omega(p)\bigr)
            \\&\qquad= 
            \bigl|\Omega^1(B_N,\mathbb{Z}_n) \bigr|
            \sum_{\omega \in \Omega^2(B_N,\mathbb{Z}_n)} \widehat{L_\gamma}(\omega) \prod_{e \in C_1(B_N)^+} \hat \varphi_\kappa\bigl( \delta \omega (e)\bigr)
             \prod_{p \in C_2(B_N)^+} \hat \varphi_\beta\bigl(\omega(p)\bigr),
            \end{split}
        \end{equation*} 
        where the second equality follows by applying Lemma~\ref{lemma: symmetry}. From this the desired conclusion immediately follows.
\end{proof}

\section{Proof of Theorem~\ref{theorem: perimeter law}}
\label{section: perimeter law}

Using Proposition~\ref{proposition: high-temperature expansion ALHM 3}, we now give a proof of Theorem~\ref{theorem: perimeter law}.
\begin{proof}[Proof of Theorem~\ref{theorem: perimeter law}]
    Fix any \( \omega \in \Omega^2(B_N,\mathbb{Z}_n) . \)
    Then, for any \( e \in \support \gamma, \) we have
    \begin{equation*}
        \varphi_\kappa\bigl( (\delta \omega + \gamma)(e)\bigr) /\varphi_\kappa\bigl( \delta \omega(e)\bigr) \geq \min_{j \in [n]} \varphi_\kappa( j \pm 1) /\varphi_\kappa( j) = \eta_\kappa,
    \end{equation*}
    and hence
    \begin{equation*}
        \widehat{L_\gamma}(\omega) = \prod_{e \in \support \gamma}\varphi_\kappa\bigl( (\delta \omega + \gamma)(e)\bigr)
            /\varphi_\kappa\bigl( \delta \omega(e)\bigr) \geq \eta_\kappa^{| \gamma|}.
    \end{equation*}
    Using Proposition~\ref{proposition: high-temperature expansion ALHM 3}, we thus obtain
    \begin{equation*}
        \mathbb{E}_{N,\beta,\kappa} \bigl[ L_\gamma(\sigma)\bigr]
        =
        \mathbb{E}_\varphi \bigl[ \widehat{L_\gamma}(\omega)\bigr]
        \geq 
        \eta_\kappa^{|\support\gamma|}.
    \end{equation*} 
    Employing Corollary~\ref{corollary: unitary gauge} and Proposition~\ref{proposition: limit exists}, the desired conclusion  immediately follows.  
\end{proof}

\section{The activity of a 2-form}
\label{section: activity}

If \( \omega \in \Omega^2(B_N,\mathbb{Z}_n), \) then \( \varphi(\omega) \) is referred to as the \emph{activity} of \( \omega. \) The main purpose of this section is to collect useful results about the activity \( \varphi.\) 
	The first result of this section, Lemma~\ref{lemma: action factorization forms} below, gives a condition given which the action factorizes.

\begin{lemma}\label{lemma: action factorization forms}
    Let \( \omega,\omega' \in \Omega^2(B_N,\mathbb{Z}_n), \) and assume that  \( \omega' \lhd  \omega. \) Then \( {\varphi(\omega) = \varphi(\omega') \varphi(\omega - \omega').} \)
\end{lemma}

\begin{proof} 
    Since \( \omega' \lhd  \omega, \) we have \(  \omega|_{\support  \omega'} =  \omega'. \) Consequently,
    \begin{itemize}[label=$\cdot$]
        \item if \( p \in \support \omega', \) then \( \omega(p) = \omega'(p) \) and \( (\omega-\omega')(p) = 0, \) and
        \item if \( p \in C_2(B_N) \smallsetminus \support \omega', \) then \( \omega'(p) = 0 \) and \( \omega(p) = (\omega-\omega')(p). \)
    \end{itemize} 
    Since \( \varphi_\beta(0) = 1, \) it follows that 
    \begin{equation*}
         \varphi_\beta\bigl( \omega(p)\bigr) = \varphi_\beta\bigl( \omega'(p)\bigr)\varphi_\beta\bigl( (\omega-\omega')(p)\bigr) \quad \forall p \in C_2(B_N),
    \end{equation*} 
    where $\varphi_\beta$ is defined by~\eqref{varphidef}.
    At the same time, since \(   \omega' \lhd \omega, \) the sets \( \support \delta \omega' \) and \( \support \delta (\omega-\omega') \) are disjoint. Since \( \delta \omega = \delta (\omega' + \omega-\omega' ) = \delta \omega' + \delta(\omega-\omega'), \) it follows that \( \delta\omega|_{\support \delta\omega'} = \delta \omega' \) and \( \delta\omega|_{C_1(B_N)\smallsetminus \support \delta\omega'} = \delta (\omega-\omega'). \)
    Consequently, 
    \begin{itemize}[label=$\cdot$]
        \item if \( e \in \support \delta \omega', \) then \( \delta\omega(e) = \delta\omega'(e) \) and \( \delta(\omega-\omega')(e) = 0, \) and
        
        \item if \( e \in C_1(B_N)\smallsetminus\support \delta \omega', \) then \( \delta\omega'(e) = 0 \) and \( \delta \omega(e)  = \delta(\omega-\omega')(e). \) 
    \end{itemize}
    Since \( \varphi_\kappa(0) = 1, \) we conclude that
    \begin{equation*}
        \varphi_\kappa\bigl(\delta\omega(e)\bigr) = \varphi_\kappa\bigl(\delta\omega'(e)\bigr)\varphi_\kappa\bigl(\delta(\omega-\omega')(e)\bigr) \quad \forall e \in C_1(B_N).
    \end{equation*}
    Recalling the definition~\eqref{eq: phi def forms} of \( \varphi \), the desired conclusion follows. 
\end{proof}

The next lemma shows that when calculating \( \widehat{L_\gamma}(\omega) \) for some \( \omega \in \Omega^2(B_N,\mathbb{Z}_n), \) it suffices to know the spins of \( \omega \) "close to" the support of \( \gamma. \)

\begin{lemma}\label{lemma: action factorization forms iic}
    Let \( \gamma \) be a path and let \( \omega \in \Omega^2(B_N,\mathbb{Z}_n). \) Then \( \widehat{L_\gamma}(\omega) = \widehat{L_\gamma}(\omega^\gamma), \) where \( \omega^\gamma \) was defined in~\eqref{eq: omegagammadef}.
\end{lemma}

\begin{proof} 
    By definition, we have  \( \delta\omega(e) = \delta \omega^\gamma(e) \) for all \( e \in \support \gamma. \) Using the definition~\eqref{eq: hat L} of \( \widehat{L_\gamma}(\omega), \) we thus obtain
    \begin{equation*}
        \begin{split}         &\widehat{L_\gamma}(\omega) = \prod_{e \in \support \gamma} \varphi_\kappa\bigl( \delta \omega(e) + \gamma[e] \bigr)\varphi_\kappa\bigl( (\delta \omega )(e) \bigr)^{-1} 
        \\&\qquad= \prod_{e \in \support \gamma} \varphi_\kappa\bigl( (\delta \omega^\gamma + \gamma)[e] \bigr)\varphi_\kappa\bigl( (\delta \omega^\gamma )(e) \bigr)^{-1} 
        = \widehat{L_\gamma}(\omega^\gamma)
        \end{split}
    \end{equation*}
    as desired.
\end{proof}

The next lemma shows that the action \( \omega' \) of a spin configuration is an upper bound on the probability of having \( \omega' \lhd \omega \) when \( \omega \sim \mathbb{P}_\varphi. \)
\begin{lemma}\label{lemma: action upper bound forms}
    Let \( \omega' \in \Omega^2(B_N,\mathbb{Z}_n). \) Then
    \begin{equation}\label{eq: action upper bound forms}
        \mathbb{P}_{\varphi}\bigl(\{ \omega \in \Omega^2(B_N,\mathbb{Z}_n) \colon  \omega' \lhd \omega \} \bigr) \leq \varphi(\omega').
    \end{equation}
\end{lemma}

\begin{proof}
    Let \( H \coloneqq \bigl\{ \omega \in \Omega^2(B_N,\mathbb{Z}_n) \colon \omega' \lhd \omega \bigr\}, \) and define \( \tau \colon \omega \mapsto \omega - \omega'. \) Then 
    \begin{equation}\label{eq: action upper bound 1b}
        \begin{split}
            &\mathbb{P}_{\varphi}\bigl(\{ \omega \in \Omega^2(B_N,\mathbb{Z}_n) \colon  \omega' \lhd  \omega \} \bigr) 
            =
            \frac{\sum_{\omega \in H} \varphi(\omega)}{\sum_{\omega \in \Omega^2(B_N,\mathbb{Z}_n) }\varphi(\omega)}
            \leq
            \frac{\sum_{\omega \in H} \varphi(\omega)}{\sum_{\omega \in \tau(H)}\varphi(\omega)}.
        \end{split}
    \end{equation}
    By Lemma~\ref{lemma: action factorization forms}, for each \( \omega \in H \) we have \( \varphi (\omega) = \varphi(\omega') \varphi(\omega - \omega'). \) Consequently,
    \begin{equation}\label{eq: action upper bound 2b}
        \begin{split}
            &
            \sum_{\omega \in H} \varphi(\omega)
            =
            \sum_{\omega \in H} \varphi(\omega')\varphi(\omega - \omega')
            =
            \varphi(\omega')\sum_{\omega \in H} \varphi(\omega - \omega')
            =
            \varphi(\omega')\sum_{\omega \in \tau(H)} \varphi(\omega),
        \end{split}
    \end{equation}
    where the last equality follows since \( \tau \) is a bijection from \( H \) to \( \tau(H). \)
    Combining~\eqref{eq: action upper bound 1b}~and~\eqref{eq: action upper bound 2b} we obtain~\eqref{eq: action upper bound forms} as desired.
\end{proof}

The next lemma, which is the last lemma of this section, uses the action to obtain an upper bound on the expectation of a natural observable.
\begin{lemma}\label{lemma: action upper bound forms iii}
    Let \( \omega' \in \Omega^2(B_N,\mathbb{Z}_n). \) Then
    \begin{equation}\label{eq: action upper bound forms ii}
        \begin{split}
            \mathbb{E}_{\varphi} \pigl[ \widehat{L_\gamma}(\omega) \cdot \mathbb{1}( \omega^\gamma \lhd \omega' \lhd \omega) \pigr]
            \leq \widehat{L_\gamma}(\omega') \varphi(\omega') .
        \end{split}
    \end{equation}
\end{lemma}

\begin{proof}
    Let \( H \coloneqq \bigl\{ \omega \in \Omega^2(B_N,\mathbb{Z}_n) \colon \omega^\gamma \lhd \omega' \lhd \omega \bigr\}, \) 
    and define
     \( \tau \colon \omega \mapsto \omega - \omega' . \) 
     By definition, we then have
    \begin{equation}\label{eq: action upper bound 1b ii}
        \begin{split}
            &\mathbb{E}_{\varphi} \pigl[ \widehat{L_\gamma}(\omega) \cdot \mathbb{1}\bigl( \omega^\gamma \lhd \omega' \lhd \omega  )\pigr]
            =
            \frac{\sum_{\omega \in H} \widehat{L_\gamma}(\omega) \varphi(\omega)}{\sum_{\omega \in \Omega^2(B_N,\mathbb{Z}_n) }\varphi(\omega)}
            \leq
            \frac{\sum_{\omega \in H} \widehat{L_\gamma}(\omega) \varphi(\omega)}{\sum_{\omega \in \tau(H)}\varphi(\omega)}.
        \end{split}
    \end{equation}

    Now fix any \( \omega \in H. \)
    Since \( \omega \in H \) we have \( \omega^\gamma \lhd \omega', \) and by Lemma~\ref{lemma: action factorization forms iib} (applied with \( E_1 = \support \delta \omega \cap \support \gamma \) and \( E_2 = C_1(B_N)^+ \)), we also have \( \omega^\gamma \lhd \omega. \)
    Moreover, since \( \omega^\gamma \lhd \omega' \lhd \omega, \) we have \( (\omega')^\gamma = \omega^\gamma. \)
    Consequently, by Lemma~\ref{lemma: action factorization forms iic}, \( \widehat{L_\gamma}(\omega) = \widehat{L_\gamma}(\omega^\gamma) = \widehat{L_\gamma}(\omega'). \) 
    At the same time, since \( \omega \in H,\) we have \( \omega' \lhd \omega \) and hence, by Lemma~\ref{lemma: action factorization forms}, \( \varphi(\omega) = \varphi(\omega')\varphi(\omega-\omega') . \)
    Summing over all \( \omega \in H, \) we thus obtain
    \begin{equation}\label{eq: action upper bound 2b ii}
        \begin{split}
            &
            \sum_{\omega \in H} \widehat{L_\gamma}(\omega)\varphi(\omega)
            =
            \sum_{\omega \in H} \widehat{L_\gamma}(\omega')
            \varphi(\omega')\varphi(\omega - \omega')
            =
            \widehat{L_\gamma}(\omega') \varphi(\omega')\sum_{\omega \in \tau(H)} \varphi(\omega),
        \end{split}
    \end{equation} 
    where the last equality follows from the fact that \( \tau \) is a bijection from \( H \) to \( \tau(H). \) 
    Combining~\eqref{eq: action upper bound 1b ii}~and~\eqref{eq: action upper bound 2b ii}, we obtain~\eqref{eq: action upper bound forms ii} as desired.    
\end{proof}

\section{Useful upper bounds}
\label{section: useful upper bounds}

The purpose of this section is to state and prove two lemmas, which will be crucial in the proofs of subsequent results. Recall from Section~\ref{connectedsubsec} that for $P,P_0 \subset C_2(B_N)^+$, $P^{P_0}$ denotes the union of the connected components of $P$ that intersect $P_0$.

\begin{lemma}\label{lemma: flip a set forms}
    Let \( P_0 \subseteq C_2(B_N)^+  \) and let \( k,k' \geq 1. \) Assume that~\eqref{assumption: 1} holds, and define
    \begin{equation*}
        \begin{split}
            &H \coloneqq \pigl\{ \omega \in \Omega^2(B_N,\mathbb{Z}_n) \colon P_0 \subseteq (\support \omega)^+
            =
            \bigl((\support \omega)^+\bigr)^{P_0},\,
            \\
        & \hspace{4cm}
            |\support \omega| \geq 2k,\, |\support \omega| + |\support \delta \omega| \geq 2k' \pigr\}.
        \end{split}
    \end{equation*} 
    Then
    \begin{equation}\label{eq: flip a set forms}
    \begin{split}
        \mathbb{P}_{\varphi} \pigl(\bigl\{\omega \in \Omega^2(B_N,\mathbb{Z}_n) \colon  \exists \omega' \in H \colon    \omega' \lhd \omega  \bigr\}\pigr)
        \leq   
        \frac{(8m)^{2k} \zeta_\beta^k \xi_\kappa^{k'-k} }{1-(8m)^{2} \zeta_\beta \xi_\kappa^{-1} }. 
    \end{split}
    \end{equation} 
\end{lemma}

\begin{proof}
    Let \( \mathcal{G} \) be the graph with vertex set \( C_2(B_N)^+ \) and an edge between two distinct vertices \( p_1,p_2 \in C_2(B_N)^+\) if either
    \begin{enumerate}[label=(\roman*)]
        \item \( p_1 \sim p_2 \) or
        \item \( p_1,p_2\in P_0 \) and \( \bigl|\{p\in P_0 \colon p < p_1\} \Delta \{p\in P_0 \colon p < p_2\}\bigr|= 1, \)
    \end{enumerate} 
    where $A \Delta B$ denotes the symmetric set difference of two sets $A$ and $B$. Given $p_1 \in C_2(B_N)^+$, there are at most $4\bigl(2(m-2)+1\bigr)$ different plaquettes $p_2 \in C_2(B_N)^+$ such that $p_1 \sim p_2$.
    Hence each vertex of \( \mathcal{G} \) has degree at most \( 4\bigl(2(m-2)+1\bigr)+2 \leq 8m. \)

    For each \( j \geq 1, \) let \( \mathcal{T}_{P_0,j} \) be the set of all walks $T$ of length \( 2j \) in \( \mathcal{G}, \) such that $T$ starts at \( \min P_0 \) and $T$ is a spanning walk of some \( P \subseteq C_2(B_N)^+ \) with \( P^{P_0} = P \) and \( |P| = j. \) 
    Since for any \( \omega \in H  \) the set \( (\support \omega)^+ \)  is a connected set in \( \mathcal{G}, \)  
    the set \( (\support \omega)^+ \) must have a spanning walk of length \( |\support \omega|  \) which starts at \( \min P_0. \)
    Since \( |\support \omega| \geq 2k \) for all \( \omega \in H, \) it follows that
    \begin{equation*}
        \{ (\support \omega)^+\colon  \omega \in H \} \subseteq  \bigcup_{j \geq k} \bigl\{ \support T \colon T \in  \mathcal{T}_{P_0,j} \bigr\}.
    \end{equation*}
    Applying Lemma~\ref{lemma: action upper bound forms},  we thus obtain 
    \begin{equation}\label{eq: flip a set ii 1 forms}
    \begin{split}
        &\mathbb{P}_{\varphi} \pigl(\bigl\{\omega \in \Omega^2(B_N,\mathbb{Z}_n) \colon  \exists \omega' \in H \colon \omega' \lhd \omega \}\bigr)
        \leq
        \sum_{\omega' \in H} \mathbb{P}_{\varphi}\pigl(\bigl\{\omega \in \Omega^2(B_N,\mathbb{Z}_n) \colon \omega' \lhd \omega \bigr\}\pigr) 
        \\&\qquad 
        \leq
        \sum_{\omega' \in H} \varphi(\omega') 
        \leq
        \sum_{j = k}^\infty \sum_{\substack{T \in \mathcal{T}_{P_0,j} }} \sum_{\substack{\omega' \in H \mathrlap{\colon}\\ (\support \omega')^+ = T}}\varphi(\omega') .
    \end{split}
    \end{equation} 
    
    Since each vertex of \( \mathcal{G} \) has degree at most \( 8m, \) for each \( j \geq 1 \) we have  \( |\mathcal{T}_{P_0,j}| \leq (8m)^{2j}.\)
    At the same time, if \( j \geq 1 \) and \( T \in \mathcal{T}_{P_0,j}, \) then (using Lemma~\ref{lemma: 0 first}), we have
    \begin{equation*}
        \sum_{\substack{\omega' \in H \mathrlap{\colon}\\ (\support \omega')^+ = T}}\varphi(\omega')  
        \leq 
        \bigg(\sum_{g \in \mathbb{Z}_n\smallsetminus \{ 0 \}} \varphi_\beta(g) \bigg)^j \max_{g \in \mathbb{Z}_n \smallsetminus \{ 0 \} }\varphi_\kappa(g)^{\max(k'-j,0)}
        =
        \zeta_\beta^j \xi_\kappa^{\max(k'-j,0)}.
    \end{equation*} 
    Combining these observations, we arrive at
    \begin{equation}\label{eq: flip a set ii 2 forms}
        \sum_{j = k}^\infty \sum_{\substack{T \in \mathcal{T}_{P_0,j} }} \sum_{\substack{\omega' \in H \mathrlap{\colon}\\ (\support \omega')^+ = T}}\varphi(\omega')
        \leq  
        \sum_{j = k}^\infty   
        (8m)^{2j} 
        \zeta_\beta^j \xi_\kappa^{\max(k'-j,0)}.
    \end{equation}
    Finally, note that
    \begin{equation}\label{eq: flip a set ii 3 forms}
        \begin{split}
        &\sum_{j = k}^\infty   
        (8m)^{2j} \zeta_\beta^j \xi_\kappa^{\max(k'-j,0)}
        =
        \sum_{j = k}^{k'-1}   
        (8m)^{2j} \zeta_\beta^j \xi_\kappa^{k'-j}
        +
        \sum_{j = k'}^\infty   
        (8m)^{2j} \zeta_\beta^j
        \\&\qquad= 
        \frac{(8m)^{2k} \zeta_\beta^k \xi_\kappa^{k'-k} -(8m)^{2k'} \zeta_\beta^{k'}  }{1-(8m)^{2} \zeta_\beta \xi_\kappa^{-1} } 
        +
        \frac{(8m)^{2k'} \zeta_\beta^{k'} }{1-(8m)^{2} \zeta_\beta} 
        \leq
        \frac{(8m)^{2k} \zeta_\beta^k \xi_\kappa^{k'-k}  }{1-(8m)^{2} \zeta_\beta \xi_\kappa^{-1} } 
        ,
        \end{split}
    \end{equation}
    where the last step uses that $\xi_\kappa \leq 1$ as a consequence of Lemma~\ref{lemma: 0 first}, and that \( (8m)^{2} \zeta_\beta < \xi_\kappa\) as a consequence of the assumption~\eqref{assumption: 1}.
    Combining~\eqref{eq: flip a set ii 1 forms},~\eqref{eq: flip a set ii 2 forms}, and~\eqref{eq: flip a set ii 3 forms}, we obtain~\eqref{eq: flip a set forms} as desired. 
\end{proof}

\begin{lemma}\label{lemma: flip a set forms iii}
    Let \( E \subseteq \support \gamma, \) and let $k,k',k''$ be non-negative integers such that \( k \geq 1 .\)  
    Assume that~\eqref{assumption: 3} holds.
    Let \( H' \) be the set of all \( \omega \in \Omega^2(B_N,\mathbb{Z}_n)  \) such that
    \begin{enumerate}[label=\upshape(\roman*)]
        \item \( \support \delta \omega \cap \support \gamma \subseteq E, \)
        \item \( |\support \omega| = 2k, \)
        \item \(  |\support \delta \omega| = 2k', \) and
        \item \(  |\support \delta \omega \cap \support \gamma | = k''. \)
    \end{enumerate}
    Then
    \begin{equation}\label{eq: flip a set forms iii}
    \begin{split}
        &\mathbb{E}_{\varphi} \bigl[ \widehat{L_\gamma}(\omega) \cdot \mathbb{1}(\omega^E \in H'  )\bigr]
        \leq (16m)^{2k} \zeta_\beta^k \xi_\kappa^{| \gamma| + k'-2k''}.
    \end{split}
    \end{equation} 
\end{lemma}

\begin{proof}
    Let \( \mathcal{G} \) be the graph with vertex set \( C_2(B_N)^+ \) and an edge between two distinct vertices \( p_1,p_2 \in C_2(B_N)^+\) if either
    \begin{itemize}[label=$\cdot$]
        \item \( p_1 \sim p_2, \) or
        \item there are two edges \(  e_1 \in (\partial p_1)^+ \cap E \) and \(  e_2 \in (\partial p_2)^+ \cap E \) such that
        \( \bigl|\{e\in E \colon e < e_1\} \Delta \{e \in E \colon e < e_2\}\bigr|= 1. \)
    \end{itemize} 
    Given $p_1 \in C_2(B_N)^+$, there are at most $4\bigl(2(m-2)+1\bigr)$ choices of $p_2 \in C_2(B_N)^+$ such that $p_1 \sim p_2$.
    Furthermore, since $\gamma$ is a path along the boundary of a rectangle with side lengths $\geq 2$, there are at most $2$ choices of $e_1$ such that $e_1 \in (\partial p_1)^+ \cap E$. For each such $e_1$, there are at most $2$ choices of $e_2$ such that $|\{e\in E \colon e < e_1\} \Delta \{e \in E \colon e < e_2\}|= 1$, and for each such $e_2$, there are at most $2(m-1)$ choices of $p_2 \in C_2(B_N)^+$ such that $e_2 \in (\partial p_2)^+$. Hence each vertex of \( \mathcal{G} \) has degree at most \( 4\bigl(2(m-2)+1\bigr)+2 \cdot 2 \cdot 2(m-1) \leq 16m. \) 
    
    For \( j \geq 1, \) let \( \mathcal{T}_{E,j} \) be the set of all walks $T$ of length \( 2j \) in \( \mathcal{G}, \) such that $T$ starts at some plaquette in \( (\hat \partial \min \support E)^+ \) and is a spanning walk of some set \( P \subseteq C_2(B_N)^+ \) with \( |P|=j. \) %
    Since for any \( \omega \in H' \) the set \( (\support \omega)^+ \)  is a connected set in \( \mathcal{G}, \)  
    the set \( (\support \omega)^+ \) must have a spanning walk of length \( \bigl|\support \omega \bigr| = 2j \) which starts at some plaquette in \( (\hat \partial \min \support E)^+ \).
    Consequently, we have 
    \begin{equation*}
        \{ (\support \omega)^+ \colon  \omega \in H' \} \subseteq  \bigl\{ \support T \colon T \in  \mathcal{T}_{E,k} \bigr\},
    \end{equation*}
    and hence
    \begin{equation}\label{eq: flip a set ii 1 forms iiia}
        \begin{split}
            &\mathbb{E}_{\varphi} \pigl[\widehat{L_\gamma}(\omega) \cdot \mathbb{1}\bigl(\omega^E \in H'  )\pigr]
            \leq
            \sum_{\omega' \in H'} \mathbb{E}_{\varphi} \pigl[\widehat{L_\gamma}(\omega) \cdot \mathbb{1}\bigl(\omega^E = \omega'  )\pigr]
            \\&\qquad \leq 
            \sum_{\substack{T \in \mathcal{T}_{P_0,k} }} \sum_{\substack{\omega' \in H' \mathrlap{\colon}\\ (\support \omega')^+ = T}} \mathbb{E}_{\varphi} \pigl[\widehat{L_\gamma}(\omega) \cdot \mathbb{1}\bigl(\omega^E = \omega'  )\pigr].
        \end{split}
    \end{equation}
    Now note that if \( \omega' \in H' \) and \( \omega \in \Omega^2(B_N,\mathbb{Z}_n) \) are such that \( \omega^E = \omega' \) then, by Lemma~\ref{lemma: action factorization forms iic}, we have \( \omega^\gamma \lhd \omega' \lhd \omega, \)
    and hence we can apply Lemma~\ref{lemma: action upper bound forms iii} to obtain  
    \begin{equation}\label{eq: flip a set ii 1 forms iiib}
        \begin{split}
            &\mathbb{E}_{\varphi} \pigl[\widehat{L_\gamma}(\omega) \cdot \mathbb{1}\bigl(\omega^{E} = \omega'  )\pigr]
            \leq
            \mathbb{E}_{\varphi} \pigl[\widehat{L_\gamma}(\omega) \cdot \mathbb{1}\bigl(\omega^\gamma \lhd  \omega' \lhd \omega )\pigr]
            \leq 
            \widehat{L_\gamma}(\omega')\varphi(\omega')
            .
        \end{split}
    \end{equation}
    Since each vertex of \( \mathcal{G} \) has degree at most \( 16m, \) we have  \( |\mathcal{T}_{P_0,k}| \leq (16m)^{2k} \).
    At the same time, if  \( T \in \mathcal{T}_{P_0,k}, \) then
    \begin{equation*}
        \begin{split}
            &\sum_{\substack{\omega' \in H' \mathrlap{\colon}\\ (\support \omega')^+ = T}} \widehat{L_\gamma}(\omega')\varphi(\omega')
            =
            \sum_{\substack{\omega' \in H' \mathrlap{\colon}\\ (\support \omega')^+ = T}} 
            \prod_{p \in C_2(B_N)^+} 
            \varphi_\beta\bigl(\omega'(p)\bigr)     
            \prod_{e \in C_1(B_N)^+} 
            \varphi_\kappa\bigl( (\delta\omega' + \gamma)(e) \bigr)
             \\
    &   \qquad     \leq 
            \bigg(\sum_{g \in \mathbb{Z}_n\smallsetminus \{ 0 \}} \varphi_\beta(g) \bigg)^k
            \varphi_\kappa(1)^{| \gamma|-k''}
             \max_{g \in \mathbb{Z}_n \smallsetminus \{ 0 \} }\varphi_\kappa(g)^{k'-k''}
            = \zeta_\beta^m \xi_\kappa^{| \gamma| + k'-2k''}.
        \end{split}
    \end{equation*} 
    Combining these observations, and recalling that by Lemma~\ref{lemma: order} (using~\eqref{assumption: 3}), we have \( \varphi_\kappa(1) = \xi_\kappa, \) we obtain
    \begin{equation}\label{eq: flip a set ii 2 forms iii}
        \sum_{\substack{T \in \mathcal{T}_{P_0,k} }} \sum_{\substack{\omega' \in H' \mathrlap{\colon}\\ (\support \omega')^+ = T}}\widehat{L_\gamma}(\omega')\varphi(\omega')
        \leq   
        (16m)^{2k} 
        \zeta_\beta^k \xi_\kappa^{| \gamma| + k'-2k''}.
    \end{equation}
    Combining~\eqref{eq: flip a set ii 1 forms iiia},~\eqref{eq: flip a set ii 1 forms iiib}, and~\eqref{eq: flip a set ii 2 forms iii}, we obtain~\eqref{eq: flip a set forms iii} as desired.
\end{proof}

\section{Properties of the support of spin configurations}\label{section: properties}
{\mbox{}

In this section, we state and prove a number of lemmas that gives properties of spin configurations that will need in the proofs of the main result. In these lemmas, whenever \( S \) is a set, \( |S| \) will denote the cardinality of \( S. \)}

For a path \( \gamma \) along the boundary of a rectangle $R$ and a 2-form \( \omega \in \Omega^2(B_N,\mathbb{Z}_n), \) we let
\begin{equation*}
    V^{\gamma,\omega} \coloneqq \bigl\{  v \in C_0(B_N)^+ \colon
    \delta (\delta \omega^\gamma|_{ \support \gamma_R})(v) \neq 0
    \bigr\},
\end{equation*}
where $\omega^\gamma$ is the $2$-form defined in~\eqref{eq: omegagammadef}. 
We note that if \( \omega^\gamma \neq 0 \) and  \( \support \gamma_R \nsubseteq \support \delta \omega^\gamma, \) then \( |V^{\gamma,\omega}| \geq 2. \) 

The first lemma of this section, Lemma~\ref{lemma: split into components}, relates the cardinality of the supports of a spin configuration to the cardinalities of the supports of its connected components.

\begin{lemma}\label{lemma: split into components}
    Let \( \omega \in \Omega^2(B_N,\mathbb{Z}_n) ,\) let \( P_0, P_1, \dots, P_{k-1} \) be the connected components of \( (\support \omega)^+,  \) and for \( j \in [k], \) let \( \omega_j \coloneqq \omega|_{ P_j }. \) Then the following hold.
    \begin{enumerate}[label=\upshape(\roman*)]
        \item \( \bigl|(\support \omega)^+ \bigr| = \sum_{j \in [k]} \bigl|(\support \omega_j)^+ \bigr| \)
        \label{item: split into components i}
        
        \item \( \bigl|(\support \delta \omega)^+ \bigr| = \sum_{j \in [k]} \bigl|(\support \delta \omega_j)^+ \bigr| \)
        \label{item: split into components ii}
        
        \item \( |\support \delta \omega \cap \support \gamma | = \sum_{j \in [k]} |\support \delta \omega_j \cap \support \gamma | \)
        \label{item: split into components iii}
        
        \item \( |P_{\omega,\gamma,c}| = \sum_{j \in [k]} |P_{\omega_j,\gamma,c}| \)
        \label{item: split into components iv}
        
        \item \( |V^{\gamma,\omega}| \leq \sum_{j \in [k] } |V^{\gamma,\omega_j}|\)
        \label{item: split into components v}
        
    \end{enumerate}
\end{lemma}

\begin{proof}
    By definition, we have \( (\support \omega)^+ = \bigsqcup_{j \in [k]} (\support \omega_j)^+ \) and \( (\support \delta \omega)^+ = \bigsqcup_{j \in [k]} (\support \delta \omega_j)^+,\) and hence~\ref{item: split into components i}--\ref{item: split into components iii} hold.
    
    We now show that~\ref{item: split into components iv} holds. To this end, assume first that \( p \in P_{\omega,\gamma,c}. \) Then there is \( j \in [k] \) such that \( p \in P_j, \) and hence \( \omega(p) = \omega_j(p) \) and \( \delta\omega(e) = \delta \omega_j(e) \) for all \( e \in \partial p. \) Consequently, we have \( p \in P_{\omega_j,\gamma,c}. \) At the same time, if \( p \in P_{\omega_j,\gamma,c}, \) for some \( j \in [k], \) then \( p \in P_j , \) and hence \( \omega(p) = \omega_j(p) \) and \( \delta\omega(e) = \delta\omega_j(e) \) for all \( e \in \support \partial p. \) Consequently, we have \( p \in P_{\omega,\gamma,c}. \) Since \( \omega_0, \dots, \omega_{k-1} \) have disjoint supports, it follows that \( P_{\omega,\gamma,c} = \bigsqcup_{j\in [k]} P_{\omega_j,\gamma,c}, \) and hence~\ref{item: split into components iv} holds.
    
    We now show that~\ref{item: split into components v} holds. To this end, assume that \( v \in V^{\gamma,\omega}. \) Then there are two edges \( \{ e,e' \} \subseteq \support \hat \partial v \cap \support \gamma_R, \) 
    such that \( \delta(\delta \omega|_{ \{ e,e' \}})(v) \neq 0 \) and \( \delta \omega(e) \neq 0.
    \) Since \( \delta \omega(e) \neq 0, \) there is \( j \in [k] \) such that \( e \in \support \delta \omega_j. \) For this \( j, \) we must have either \( e' \in \support \delta \omega_j, \) in which case \( \delta(\delta \omega_j|_{ \{ e,e' \}})(v) = \delta(\delta \omega|_{ \{ e,e' \}})(v)  \neq 0, \) or \( e' \notin \support \delta \omega_j, \) in which case
    \( \delta(\delta \omega_j|_{ \{ e,e' \}})(v) = \delta(\delta \omega|_{ \{ e\}})(v)  \neq 0, \) and hence \( v \in V^{\omega_j,\gamma}. \) Summing over all \( v \in V^{\omega,\gamma},\) we obtain~\ref{item: split into components v} as desired.
\end{proof}

The next three lemmas of this section provide inequalities that compare the cardinalities of sets naturally associated with a spin configuration. These inequalities will be important in later sections where we use them to give upper bounds on certain useful events.

\begin{lemma}\label{lemma: basic inequality 1 forms}
    Let \( \gamma \) be a path along the boundary of a rectangle with side lengths \( \ell_2 \geq \ell_1 \geq 2, \) and let \( \omega \in \Omega^2(B_N,\mathbb{Z}_n). \) Then
    \begin{equation}\label{eq: basic inequality 1}
        \bigl|(\support \omega)^+\bigr| + |P_{\omega,\gamma,c}| \geq  \bigl|\support \delta \omega \cap \support \gamma  \bigr|.
    \end{equation}
\end{lemma}

\begin{proof} 
    By definition, for any \( e \in \support \delta \omega \cap \support \gamma \) the set \( \support \hat \partial e \cap \support \omega \) must be non-empty.
    Moreover, since \( \ell_1 \geq 2, \) for any distinct \( e,e' \in \support \delta \omega \cap \support \gamma , \) the intersection of the sets \( \support \hat \partial e \cap \support \omega \) and \( \support \hat \partial e' \cap \support \omega \) can be non-empty only if their intersection is equal to \( \{ p \} \) for some corner plaquette \( p \in \mathcal{P}_{\gamma,c}. \)
    At the same time, since \( \gamma \) is a  rectangular path, there can be at most \( |P_{\omega,\gamma,c}|\) pairs of distinct edges  \( e,e' \in \support \delta \omega \cap \support \gamma \) such that \( (\support \hat \partial e \cap \support \omega) \cap (\support \hat \partial e' \cap \support \omega) = \{ p \}\) for some \( p \in \mathcal{P}_{\gamma,c}. \)
    From this the desired conclusion  follows.
\end{proof}

\begin{lemma}\label{lemma: general technical lemma forms}
    Let \( \gamma \) be a path along the boundary of a rectangle \( R \) with side lengths \( \ell_2 \geq \ell_1 \geq 8. \)  Let \( \omega \in \Omega^2(B_N,\mathbb{Z}_n). \) Then 
    \begin{equation}\label{eq: general technical lemma forms}
        \begin{split}
            &\bigl| (\support \omega)^+\bigr| + \bigl|(\support \delta \omega)^+ \bigr|
            \geq 3\bigl| \support \delta \omega  \cap \support \gamma\bigr| +  |V^{\gamma,\omega}| - 3|P_{\omega,\gamma,c}| .
        \end{split} 
    \end{equation}
\end{lemma}

\begin{proof}
    Let \( \hat P_0, \dots, \hat P_{k-1} \) be the connected components of \( (\support \omega)^+ , \) and for each \( j \in [k] \) let \( \omega_j \coloneqq \omega|_{ \hat P_j}. \) Then, using Lemma~\ref{lemma: split into components}, we see that~\eqref{eq: general technical lemma forms} holds for \( \omega \) if we can show that it holds for \( \omega_0, \dots, \omega_{k-1}. \) Consequently, we can without loss of generality assume that \( (\support \omega)^+ \) is connected. Note that if \( (\support \omega)^+ \) is connected, then either \( \omega^\gamma = 0 \) or \( \omega = \omega^\gamma. \)
    If \( \omega^\gamma = 0, \) then the right-hand side of~\eqref{eq: general technical lemma forms} is equal to zero, and hence~\eqref{eq: general technical lemma forms} holds.
    Finally, assume that \( \omega = \omega^\gamma. \) Then \( \omega^\gamma = \omega^{\gamma_R}, \) implying in particular that \( V^{\gamma,\omega} = V^{\gamma_R,\omega} . \) On the other hand, since \( \support \gamma \subseteq \support \gamma_R, \) we have \( P_{\omega,\gamma,c} \subseteq P_{\omega,\gamma_R,c}, \) and if \( p \in P_{\omega,\gamma_R,c} \smallsetminus P_{\omega,\gamma,c}, \) then there must exist \( e \in \support \partial p \cap \support \delta \omega \cap \support \gamma_R \smallsetminus \support \gamma. \) Consequently, we have
    \begin{equation*}
        |\support \delta \omega \cap \support \gamma_R| - |P_{\omega,\gamma_R,c}|
        \geq 
        |\support \delta \omega \cap \support \gamma| - |P_{\omega,\gamma,c}|.
    \end{equation*}
    This shows that it is sufficient to show that~\eqref{eq: general technical lemma forms} holds when \( \gamma = \gamma_R. \)
    Motivated by the above argument, we assume without loss of generality that \( (\support \omega)^+ \) is connected, that \( \omega = \omega^\gamma \neq 0, \) and that \( \gamma = \gamma_R. \) 
    
    We now define a number of sets that will be useful in the remainder of the proof. Define 
    \begin{equation*}
         \begin{cases}
             E_0 \coloneqq \support \delta \omega \cap \support \gamma,\cr
              P_1 \coloneqq \{ p \in (\support \omega)^+ \colon \support \partial p \cap \support \gamma \neq \emptyset \},\text{ and }\cr
              E_1 \coloneqq \bigl\{ e \in (\support \delta \omega)^+ \smallsetminus E_0 \colon \exists e' \in \support \gamma \colon \support \hat \partial e \cap \support \hat \partial e' \neq \emptyset \bigr\}.
         \end{cases}
    \end{equation*} 
    Next, define
    \begin{equation*}
        \begin{cases}
            \tilde E_2 \coloneqq \bigl(\support \delta (\omega|_{ P_1 }) \bigr)^+ \smallsetminus (E_0 \cup E_1), \cr 
        E_2 \coloneqq \tilde E_2 \cap \support \delta \omega, \text{ and} \cr
        P_2 \coloneqq \bigl\{ p \in (\support \omega)^+ \smallsetminus P_1 \colon \support \partial p \cap (\tilde E_2 \smallsetminus E_2) \neq \emptyset \bigr\}.
        \end{cases}
    \end{equation*}
    Finally, define 
    \begin{equation*}
        \begin{cases}
            \tilde E_3 \coloneqq \bigl(\support \delta (\omega|_{P_1\cup P_2}) \bigr)^+ \smallsetminus (E_0 \cup E_1 \cup E_2), \cr 
            E_3 \coloneqq \tilde E_3 \cap \support \delta \omega, \text{ and} \cr
            P_3 \coloneqq \bigl\{ p \in (\support \omega)^+ \smallsetminus P_1 \cup P_2 \colon \support \partial p \cap (\tilde E_2 \smallsetminus E_2) \neq \emptyset \bigr\}.
        \end{cases}
    \end{equation*}
    (For illustrations of these sets, see Figure~\ref{figure: final proof examples}.)
    \begin{figure}[tp]
        \centering 
        \begin{subfigure}[t]{0.3\textwidth}\centering
            \begin{tikzpicture}[scale=0.6]
        
        \draw[line width=0.1mm,ForestGreen!60!black] (0,4) -- (0,0) -- (7,0) -- (7,4) -- cycle;

        \foreach \x in {0,0.5,...,2}
            \foreach \y in {0}
                \fill[BlueViolet] (\x+0.1,\y+0.1) rectangle ++(0.3,0.3);
                
        \foreach \x in {0}
            \foreach \y in {0.5,1}
                \fill[BlueViolet] (\x+0.1,\y+0.1) rectangle ++(0.3,0.3);
                
        \foreach \x in {0.5}
            \foreach \y in {0.5}
                \fill[BlueViolet] (\x+0.1,\y+0.1) rectangle ++(0.3,0.3);

         \draw[thick] (0,0) -- (2.5,0) -- (2.5,0.5) -- (1,0.5) -- (1,1) -- (0.5,1) -- (0.5,1.5) -- (0,1.5) -- cycle; 
    \end{tikzpicture}
    \caption{The set \( (\support \omega)^+ \) (blue squares) and the set \( (\support \delta \omega)^+ \) (solid line).}
    \label{subfig: final proof example Ai}
    \end{subfigure}
    \hfil
        \begin{subfigure}[t]{0.3\textwidth}\centering
            \begin{tikzpicture}[scale=0.6]
        
        \draw[line width=0.1mm,ForestGreen!60!black] (0,4) -- (0,0) -- (7,0) -- (7,4) -- cycle;

        \foreach \x in {0,0.5,...,2}
            \foreach \y in {0}
                \fill[BlueViolet] (\x+0.1,\y+0.1) rectangle ++(0.3,0.3);
                
        \foreach \x in {0}
            \foreach \y in {0.5,1}
                \fill[BlueViolet] (\x+0.1,\y+0.1) rectangle ++(0.3,0.3);  
          
         \draw[thick] (0,1.5) -- (0,0) -- (2.5,0);
         \draw[dashed] (0,1.5) -- (0.5,1.5);
         \draw[dashed] (2.5,0) -- (2.5,0.5);
         \draw[dotted] (2.5,0.5) -- (0.5,0.5) -- (0.5,1.5);
        \end{tikzpicture}
    \caption{The set \( E_0 \) (solid line), the set \( P_1 \) (blue squares), the set \( E_1 \) (dashed line), and the set \( \tilde E_2 \) (dotted lines). }
    \label{subfig: final proof example Aii}
    \end{subfigure}
    \hfil
    \begin{subfigure}[t]{0.3\textwidth}\centering
        \begin{tikzpicture}[scale=0.6]
        
        \draw[line width=0.1mm,ForestGreen!60!black] (0,4) -- (0,0) -- (7,0) -- (7,4) -- cycle;

        \foreach \x in {0.5}
            \foreach \y in {0.5}
                \fill[BlueViolet] (\x+0.1,\y+0.1) rectangle ++(0.3,0.3);
          
         \draw[thick] (1,0.5) -- (2.5,0.5);
         \draw[thick] (0.5,1) -- (0.5,1.5);
         \draw[dashed] (0.5,1) -- (1,1) -- (1,0.5);
    \end{tikzpicture}
    \caption{The set \( P_2 \) (blue square), the set \( E_2 \) (solid lines), and the set \( \tilde E_3 = E_3 \) (dashed line).}
    \label{subfig: final proof example Aiii}
    \end{subfigure}
        \caption{In the above figures, we illustrate the sets introduced in the proof of Lemma~\ref{lemma: general technical lemma forms} for a form \( \omega \in \Omega^2(B_N,\mathbb{Z}_2) \) and closed path \( \gamma \) (green line above). Note that in this case, we have \( P_3 = \emptyset. \)}
        \label{figure: final proof examples}
    \end{figure}
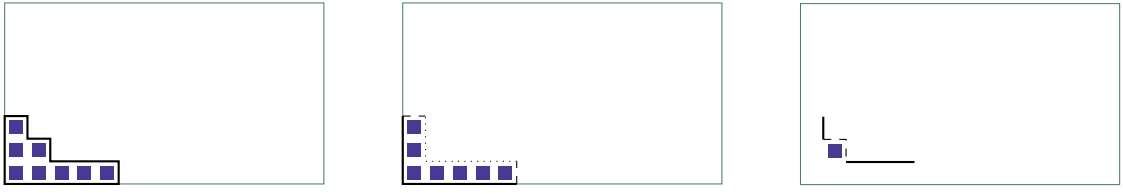
    Note that \( E_0 \sqcup E_1 \sqcup E_2 \sqcup E_3 \subseteq (\support \delta \omega)^+ \) and \( P_1 \sqcup P_2 \sqcup P_3 \subseteq (\support \omega)^+, \) and hence
    \begin{equation}\label{eq: part 0}
        \bigl| (\support \omega)^+ \bigr| + \bigl| (\support \delta \omega)^+ \bigr| \geq |E_0| + |E_1| + |E_2| + |E_3| + |P_1| +  |P_2| + |P_3|.
    \end{equation}
    
    We now give lower bounds on the right-hand side of~\eqref{eq: part 0},  which together with~\eqref{eq: part 0} will imply that~\eqref{eq: general technical lemma forms} holds.
    By definition, we have
    \begin{equation}\label{eq: part 1}
        |E_0| = |\support \delta \omega \cap \support \gamma|.
    \end{equation}
    Using the definition of \( |P_{\omega,\gamma,c}|, \) we have
    \begin{equation}\label{eq: part 2}
        |P_1| \geq |E_0| - |P_{\omega,\gamma,c}|,
    \end{equation}
    and since \( \delta \delta \omega = 0, \) we have 
    \begin{equation}\label{eq: part 3}
        |E_1| \geq |V^{\gamma,\omega}|.
    \end{equation}

    \begin{sublemma}
        We have 
        \begin{equation}\label{eq: part 4}
            |\tilde E_2| \geq 
            |E_0| - 2|P_{\omega,\gamma,c}|.
        \end{equation} 
    \end{sublemma}
    
    \begin{subproof}
        Define
        \begin{equation*}
            E_0' \coloneqq \{ e \in E_0 \colon \support \hat \partial e \cap P_{\omega, \gamma,c} = \emptyset \}.
        \end{equation*}
        Using the definition of \( P_{\omega,\gamma,c}, \) we see that
        \begin{equation*}
            |E_0'| = |E_0| - 2|P_{\omega,\gamma,c}|.
        \end{equation*} 
        Using this identity, we will prove that~\eqref{eq: part 4} holds by constructing an injective map \( \tau \) from \( E_0' \) to \( \tilde E_2. \)
        To this end, let \( e \in E_0' \) be arbitrary. Since \( e \in E_0' \subseteq E_0, \) the set \( \support \hat \partial e \cap P_1 \smallsetminus P_{\omega,\gamma,c} \) is non-empty, and hence either the set \( \support \hat \partial e \cap P_1 \smallsetminus \mathcal{P}_{\gamma,c} \) or the set \( \support \hat \partial e \cap P_1 \cap \mathcal{P}_{\gamma,c} \smallsetminus P_{\omega,\gamma,c} \) must be  non-empty.
        We deal with these two cases separately.
        \begin{enumerate}
            \item Assume that the set \(  \support \hat \partial e \cap  P_1 \smallsetminus \mathcal{P}_{\gamma,c}\) is non-empty. Fix any plaquette \( p_e \) in this set. Let \( e' \) be the unique edge in \( \support \partial p_e \smallsetminus \{ e \} \) that is parallel with \( e \) (see Figure~\ref{subfig: final proof example Bi}).
            
            \item Assume that the set \( \support \hat \partial e \cap  P_1 \cap \mathcal{P}_{\gamma,c} \smallsetminus P_{\omega,\gamma,c} \) is non-empty. Since \( \ell_1 \geq 2, ,\) this set will then contain exactly one plaquette, which we denote \( \hat p_e.\)   Since \( \hat p_e \in \mathcal{P}_{\gamma,c}, \) \( e \in \support \partial \hat p_e \cap \support \gamma, \) \( \gamma \) is a rectangular loop and \( \ell_1 \geq 2,\) the set \( \support \partial \hat p_e \cap \gamma \smallsetminus \{ e \} \) contains exactly one edge \( \hat e.\) Since \( \hat p_e \in P_1 \smallsetminus P_{\omega,\gamma,c} \) and \(  e \in \support \delta \omega, \) we must have \( \delta \omega (\hat e) = 0. \) Since \( \hat p_e \in \support \hat \partial \hat e \cap P_1, \) the set \( \support \hat \partial \hat e \cap \support \omega \smallsetminus \{ \hat p_e \} \) must be non-empty. Let \( p_e \) be any plaquette in this set, and note that by definition, we have \( p_e \in P_1\smallsetminus \mathcal{P}_{\gamma,c}. \) Let \( e' \) be the unique  edge in \( \support \partial p_e \smallsetminus \{ \hat e \} \) that is parallel to \( \hat e \) (see Figure~\ref{subfig: final proof example Bii}).
        \end{enumerate}
        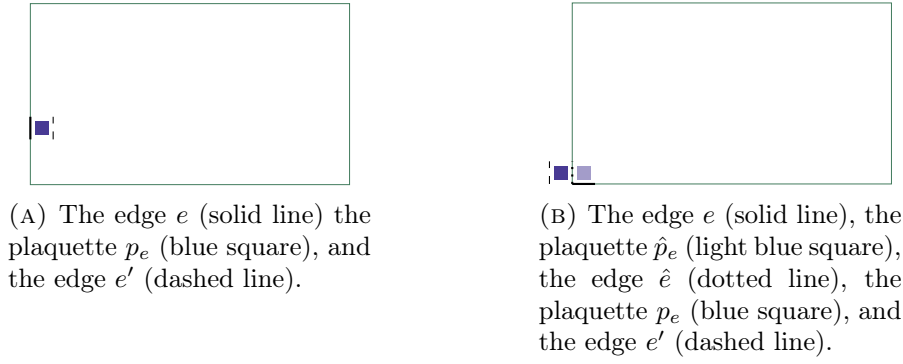
\begin{figure}[tp]
            \centering
            \begin{subfigure}[t]{0.3\textwidth}\centering
            \begin{tikzpicture}[scale=0.6]
        
        \draw[line width=0.1mm,ForestGreen!60!black] (0,4) -- (0,0) -- (7,0) -- (7,4) -- cycle;

        \foreach \x in {0}
            \foreach \y in {1}
                \fill[BlueViolet] (\x+0.1,\y+0.1) rectangle ++(0.3,0.3);

         \draw[thick] (0,1) -- (0,1.5);
         \draw[dashed] (0.5,1) -- (0.5,1.5);
    \end{tikzpicture}
    \caption{The edge \( e \) (solid line) the plaquette \( p_e \) (blue square), and the edge \( e' \) (dashed line).}
    \label{subfig: final proof example Bi}
    \end{subfigure}
    \hfil
        \begin{subfigure}[t]{0.3\textwidth}\centering
            \begin{tikzpicture}[scale=0.6]
        
        \draw[line width=0.1mm,ForestGreen!60!black] (0,4) -- (0,0) -- (7,0) -- (7,4) -- cycle;

        \foreach \x in {0}
            \foreach \y in {0}
                \fill[BlueViolet, opacity=.5] (\x+0.1,\y+0.1) rectangle ++(0.3,0.3);
                
        \foreach \x in {-0.5}
            \foreach \y in {0}
                \fill[BlueViolet] (\x+0.1,\y+0.1) rectangle ++(0.3,0.3);

         \draw[thick] (0,0) -- (0.5,0);
         \draw[thick, dotted] (0,0) -- (0,0.5);
         \draw[dashed] (-0.5,0) -- (-0.5,0.5);
        \end{tikzpicture}
    \caption{The edge \( e \) (solid line), the plaquette \( \hat p_e \) (light blue square), the edge \( \hat e \) (dotted line), the plaquette \( p_e \) (blue square), and the edge \( e' \) (dashed line).}
    \label{subfig: final proof example Bii}
    \end{subfigure}
            \caption{An illustration for the proof of~\eqref{eq: part 4}.}
            \label{figure: subproof example iii}
        \end{figure}
        Note that by construction, \( e' \) is parallel to an edge in \( \gamma \) and satisfies \( e' \in \support \partial p_e \) with \( p_e \in P_1. \) Since \( \ell_1 \geq 2, \) this implies that \( e' \notin E_0. \) 
        Next, since \( p_e \notin \mathcal{P}_{\gamma,c} \) and  \( e' \) is the edge in \( \support \partial p_e \) that is parallel to the unique edge in \( \support \partial p_e \cap \gamma ,\) we must have \( e' \notin E_1. \)
        Finally, note that since \( \ell_1 \geq 3 \) and \( p_e \in P_1, \) we have \( \support \hat \partial e' \cap P_1 = \{ p_e \}, \) and hence \( e' \in \support \delta (\omega|_{P_1}). \)
        Combining the above observations, we conclude that \( e' \in \tilde E_2. \)
        Define \( \tau(e) \coloneqq e'. \) 
        Since \( \tau \) is a mapping from \( E_0' \) to \( \tilde E_2, \) the desired conclusion will follow if we can show that \(\tau \) is injective.
        By construction, this follows immediately from observing that the mapping \( e \mapsto p_e \) is injective. This concludes the proof.
    \end{subproof} 
    Define
    \begin{equation*}
        \begin{split} 
            P_{\omega,\gamma,c,2} \coloneqq &\bigl\{ p \in P_2 \colon \exists p_1,p_2 \in  P_1, \, e_1,e_2 \in \tilde E_2\smallsetminus E_2 \colon 
            \\&\quad p_1 \neq p_2,\, e_1 \neq e_2,\, e_1 \in \support \partial p_1 \cap \support \partial p,\, e_2 \in \support \partial p_2 \cap \support \partial p ,\, 
            \\&\quad \omega(p) \partial p[e_1] + \omega(p_1) \partial p_1(e_1)
            =
            \omega(p) \partial p[e_2] + \omega(p_2) \partial p_2(e_2)
            =
            0
            \bigr\}.
         \end{split}
    \end{equation*} 
    \begin{sublemma}
        We have
        \begin{equation}\label{eq: part 5} 
            |E_2| + |P_2| \geq |\tilde E_2| - |P_{\omega,\gamma,c,2}|.
        \end{equation}
    \end{sublemma}
    
    \begin{subproof}
        Since \( \gamma \) is rectangular and \( \ell_1,\ell_2 \geq 4, \) it holds that
        \begin{equation}\label{eq: 1 or 2}
            |\support \partial p \cap \tilde E_2 \smallsetminus E_2 | \in \{1, 2\} \quad \text{ for all } p \in P_2 .
        \end{equation}
        Also, if $p \in P_{\omega,\gamma,c,2}$, then $|\support \partial p \cap \tilde E_2 \smallsetminus E_2| = 2$. 
        Hence
        \begin{equation}\label{eq: first eq in subproof}
            \begin{split}
                &|P_2| 
                = \pigl|\bigl\{ p \in P_2 \colon |\support \partial p \cap \tilde E_2 \smallsetminus E_2| =1 \bigr\} \pigr|
                +
                |P_2 \cap P_{\omega,\gamma,c,2}| 
                \\&\qquad\qquad +
                \pigl| \bigl\{ p \in P_2 \smallsetminus P_{\omega,\gamma,c,2} \colon  |\support \partial p \cap \tilde E_2 \smallsetminus E_2| = 2 \bigr\} \pigr|.
            \end{split}
        \end{equation}
        
        We now consider the first term on the right-hand side of~\eqref{eq: first eq in subproof}, and rewrite it as follows.
        \begin{equation}\label{eq: subproof part 00}
            \begin{split}
                &
                \pigl| \bigl\{ p \in P_2 \colon |\support \partial p \cap \tilde E_2 \smallsetminus E_2| =1
                \bigr\} \pigr|
                =
                \sum_{e \in \tilde E_2 \smallsetminus E_2}
                \pigl| \bigl\{ p \in P_2 \colon  \support \partial p \cap \tilde E_2 \smallsetminus E_2 = \{ e \} \bigr\} \pigr|
                \\&\qquad
                = 
                |\tilde E_2 \smallsetminus E_2|
                -
                \sum_{e \in \tilde E_2 \smallsetminus E_2}
                \Bigl( 1 - \pigl| \bigl\{ p \in P_2 \colon  \support \partial p \cap \tilde E_2 \smallsetminus E_2 = \{ e \} \bigr\} \pigr| \Bigr)
                \\&\qquad = 
                |\tilde E_2 \smallsetminus E_2|
                -
                \sum_{e \in \tilde E_2 \smallsetminus E_2}
                \Bigl( 1 - 
                \pigl| \bigl\{ p \in P_2 \cap \support \hat \partial e\colon  \support \partial p \cap \tilde E_2 \smallsetminus E_2 = \{ e \} \bigr\} \pigr|
                 \Bigr).
            \end{split}
        \end{equation}
        If \( e \in \tilde E_2 \smallsetminus E_2,\) then the set \( P_2 \cap \support \hat \partial e\) is non-empty, and if \( p \in P_2 \cap \support \hat \partial e, \) then, by definition, \( e \in \support \partial p \cap \tilde E_2 \smallsetminus E_2.\) Consequently, using~\eqref{eq: 1 or 2}, it follows that for any \(e \in \tilde E_2 \smallsetminus E_2 \) we have 
        \begin{equation}\label{eq: subproof part 01}
            \begin{split}
            &1 - 
            \pigl| \bigl\{ p \in P_2 \cap \support \hat \partial e\colon  \support \partial p \cap \tilde E_2 \smallsetminus E_2 = \{ e \} \bigr\} \pigr|
            \\&\qquad\leq 
            \mathbb{1}\bigl(\forall p \in \support \hat \partial e \cap P_2 \colon |\support \partial p \cap \tilde E_2 \smallsetminus E_2 | = 2\bigr).
            \end{split}
        \end{equation} 
        Combining~\eqref{eq: subproof part 00} and~\eqref{eq: subproof part 01}, we obtain
        \begin{equation}\label{eq: subproof part 0}
            \begin{split}
                &
                \pigl| \bigl\{ p \in P_2 \colon |\support \partial p \cap \tilde E_2 \smallsetminus E_2| =1
                \bigr\} \pigr|
                \\&\qquad\geq
                |\tilde E_2 \smallsetminus E_2| - \sum_{e \in \tilde E_2 \smallsetminus E_2}
                \mathbb{1}\bigl(\forall p \in \support \hat \partial e \cap P_2 \colon |\support \partial p \cap \tilde E_2 \smallsetminus E_2 | = 2\bigr) .
            \end{split}
        \end{equation}
         Now fix any \( e \in \tilde E_2 \smallsetminus E_2. \) Since \( \gamma \) is rectangular,  if   \( |\support \partial p \cap \tilde E_2 \smallsetminus E_2| = 2 \) for all \( p \in \support \hat \partial e \cap P_2, \)  then \( |\support \hat \partial e \cap P_2| = 1. \) Hence
        \begin{equation}\label{eq: subproof part 1}
            \begin{split}
                &\mathbb{1}\bigl(\forall p \in \support \hat \partial e \cap P_2 \colon |\support \partial p \cap \tilde E_2 \smallsetminus E_2 | = 2\bigr)
                \\&\qquad= 
                \mathbb{1}\bigl(\exists p \in P_2 \colon \support \hat \partial e \cap P_2 = \{ p \} \text{ and } |\support \partial p \cap \tilde E_2 \smallsetminus E_2 | = 2\bigr)
                \\&\qquad= 
                \mathbb{1}\bigl(\exists p \in P_{\omega,\gamma,c,2} \colon \support \hat \partial e \cap P_2 = \{ p \} \text{ and } |\support \partial p \cap \tilde E_2 \smallsetminus E_2 | = 2 \bigr)
                \\&\qquad\qquad +
                \mathbb{1}\bigl(\exists p \in P_2 \smallsetminus P_{\omega,\gamma,c,2} \colon \support \hat \partial e \cap P_2 = \{ p \} \text{ and } |\support \partial p \cap \tilde E_2 \smallsetminus E_2 | = 2\bigr).
            \end{split}
        \end{equation}
        If \( p \in P_{\omega,\gamma,c,2}, \) then, by definition, there are distinct \( e_1,e_2 \in \support \partial p \cap \tilde E_2 \smallsetminus E_2 \) such that \( {\support \hat \partial e_1 \cap P_2 = \support \hat \partial e_2 \cap P_2 = \{p \}.} \) Consequently, we have
        \begin{equation}\label{eq: subproof part 2}
        	\begin{split}
            	&\sum_{e \in \tilde E_2 \smallsetminus E_2}
            \mathbb{1}\bigl(\exists p \in P_{\omega,\gamma,c,2} \colon \support \hat \partial e \cap P_2 = \{ p \}\text{ and } |\support \partial p \cap \tilde E_2 \smallsetminus E_2 | = 2\bigr)
            \\&\qquad=
            2\sum_{p \in P_2} \mathbb{1}(p \in P_{\omega,\gamma,c,2}).
            \end{split}
        \end{equation}
        Next, note that
        \begin{equation}\label{eq: second eq in subproof}
            \begin{split}
                &\sum_{e \in \tilde E_2 \smallsetminus E_2} \mathbb{1}\bigl(\exists p \in P_2 \smallsetminus P_{\omega,\gamma,c,2} \colon \support \hat \partial e \cap P_2 = \{ p \} \text{ and } |\support \partial p \cap \tilde E_2 \smallsetminus E_2 | = 2\bigr)
                \\&\qquad=
                \sum_{p \in P_2} \mathbb{1}(p \notin P_{\omega,\gamma,c,2},\, |\support \partial p \cap \tilde E_2 \smallsetminus E_2 | = 2)
                \sum_{e \in \tilde E_2 \smallsetminus E_2} \mathbb{1}\bigl(  \support \hat \partial e \cap P_2 = \{ p \}  \bigr).
            \end{split}
        \end{equation}
        Now fix some \( p \in P_2 \smallsetminus P_{\omega,\gamma,c,2} \) with \( |\support \partial p \cap \tilde E_2 \smallsetminus E_2| = 2, \)
        and assume that \( e \in \tilde E_2 \smallsetminus E_2 \) is such that \( \support \hat \partial e \cap P_2 = \{ p \}. \) Let \( e' \) be defined by \( \{ e,e' \} = \support \partial p \cap \tilde E_2 \smallsetminus E_2 . \) Then there are distinct plaquettes \( \hat p\) and \( \hat p' \) such that \( \{ \hat p \} =  \support \hat \partial e \cap P_1 \) and   \( \{\hat p' \} =  \support \hat \partial e' \cap  P_1, \) and by definition, we have  \( e \in \support \partial \hat p \cap  \support \partial p \) and \( e' \in \support \partial \hat p' \cap  \support \partial p. \)  Since \( \support \hat \partial e \cap P_2 = \{ p \}, \) we must have, using also~\eqref{deltaomegae}, 
    \begin{equation*}
        \omega(p) \partial p[e] + \omega(\hat p) \partial \hat p[e] = \delta \omega(e) = 0,
    \end{equation*}
    and hence, since \( p \notin P_{\omega,\gamma,c,2}, \) we have
    \begin{equation*}
        \omega(p) \partial p[e'] + \omega(\hat p) \partial \hat p'[e'] \neq 0.
    \end{equation*}
    Since  \( e' \notin E_2, \) it follows that \( |\support \hat \partial e' \cap P_2| \geq 2, \) and hence
    \begin{equation*}
        \pigl| \bigl\{ e \in \tilde E_2 \smallsetminus E_2 \colon  \support \hat \partial e \cap P_2 = \{ p \}  \bigr\} \pigr| \leq 1.
    \end{equation*} 
    Together with~\eqref{eq: second eq in subproof}, this shows that
    \begin{equation}\label{eq: subproof part 3}
        \begin{split}
           &\sum_{e \in \tilde E_2 \smallsetminus E_2} \mathbb{1}\bigl(\exists p \in P_2 \smallsetminus P_{\omega,\gamma,c,2} \colon \support \hat \partial e \cap P_2 = \{ p \} \text{ and } |\support \partial p \cap \tilde E_2 \smallsetminus E_2 | = 2\bigr)
           \\&\qquad= 
           \bigl|\bigl\{ p \in P_2 \smallsetminus P_{\omega,\gamma,c,2} \colon  |\support \partial p \cap \tilde E_2 \smallsetminus E_2 | = 2 \bigr\} \pigr|.
        \end{split}
    \end{equation}
    Combining~\eqref{eq: first eq in subproof}, ~\eqref{eq: subproof part 0},~\eqref{eq: subproof part 1},~\eqref{eq: subproof part 2}, and~\eqref{eq: subproof part 3}, we finally obtain
    \begin{equation*}
            \begin{split}
                &|P_2| 
                \geq 
                |\tilde E_2 \smallsetminus E_2| 
                - 
                |P_2 \cap P_{\omega,\gamma,c,2}|.
            \end{split}
        \end{equation*}
        From this the desired conclusion immediately follows.
    \end{subproof}

  Combining~\eqref{eq: part 0},~\eqref{eq: part 1},~\eqref{eq: part 3},~\eqref{eq: part 4}, and~\eqref{eq: part 5}, we obtain   
    \begin{equation}
        \bigl| (\support \omega)^+ \bigr| + \bigl| (\support \delta \omega)^+ \bigr| \geq 
        3\bigl| \support \delta \omega  \cap \support \gamma\bigr| +  |V^{\gamma,\omega}| - 3|P_{\omega,\gamma,c}| + M,
    \end{equation}
    where
    $$M \coloneqq |E_3| + |P_3| + |P_1|  - |E_0| + |P_{\omega,\gamma,c}| - |P_{\omega,\gamma,c,2}|.$$
    The inequality~\eqref{eq: general technical lemma forms} will follow if we can show that $M \geq 0$. 
    If \( |P_{\omega,\gamma,c,2}| = 0, \) then~\eqref{eq: part 2} implies that $M \geq 0$. We can thus assume that \( |P_{\omega,\gamma,c,2}| \geq 1. \) Since \( \gamma \) is a rectangular loop, we always have  \( |P_{\omega,\gamma,c,2}| \leq 4. \) To see that $M \geq 0$ also in this case, we will show the following.
    \begin{itemize}
        \item[$(\star)$] For each \( p \in P_{\omega,\gamma,c,2}, \) there is an edge \( e = e_p \in C_1(B_N)^+ \) at distance at most $1$ from \( p \) with \( e \in \tilde E_3. \) \label{item: star 1}
    \end{itemize}
    To see that this implies that \( M \geq 0, \) we make the following observations.
    First, note that since \( \ell_1,\ell_2 \geq 8, \) if \( \dist(e,p) \leq 1 \) and \( \dist(e',p') \leq 1 \) for distinct \( p,p' \in P_{\omega,\gamma,c,2}, \) then \( e \neq e', \) and \( \support \hat \partial e \cap \support \hat \partial e' = \emptyset. \)
    Next, note that if \( e_p \in \tilde E_3, \) then either \( e_p \in E_3\) or the set \( \support \hat \partial e_3 \cap P_3\) is non-empty.
    Consequently, provided that \( (\star) \) is correct, we obtain
    \begin{equation}\label{eq: part 6}
        |P_{\omega,\gamma,c,2}| \leq | E_3| + |P_3| \leq | E_3| + |P_3| + \pigl( |P_1| - \bigl(|E_0| - |P_{\omega,\gamma,c}| \bigr) \pigr),
    \end{equation}
    which implies that $M \geq 0$ as desired. 
    It thus remains only to prove the following claim.
    \begin{sublemma}
        \( (\star) \) holds.
    \end{sublemma}
    
    \begin{subproof}
        Let \( p \in P_{\omega,\gamma,c,2} \) be arbitrary.
        Since \( \gamma \) is rectangular and \( \ell_1 \geq 4,\) we have \( |\support \partial p \smallsetminus \tilde E_2| = 2, \) and we can hence let \( e_1,e_2 \in C_1(B_N)^+\) be defined by \( \{ e_1,e_2 \} \coloneqq \support \partial p \cap \tilde E_2 \) (see Figure~\ref{subfig: final proof A}). Note that by definition, \( e_1\) and \( e_2\) are distinct.

        We now divide into the two cases \( \{ e_1,e_2 \} \cap  \tilde E_3 \neq \emptyset \) and \( \{ e_1,e_2 \} \cap  \tilde E_3 = \emptyset. \)
    
        If \( \{ e_1,e_2 \} \cap  \tilde E_3 \neq \emptyset, \) then~\( (\star) \) holds, and hence we are done. 
        
        Now assume that \( \{ e_1,e_2 \} \cap \tilde E_3 = \emptyset\). Then there must exist \( p_1,p_2 \in P_2 \smallsetminus \{ p \} \) with \( e_1 \in \support \partial p_1 \) and \( e_2 \in \support \partial p_2 \) (see Figure~\ref{subfig: final proof B}).
        Since \( p_1 \in P_2, \) there must exist \( e_0 \in \support \gamma \) and \( p_0 \in P_1 \cap \support \hat \partial e_0 \) such that \( \support \partial p_0 \cap \support \partial p_1 \neq \emptyset. \) Let \( e_0' \) be defined by  \( \{ e_0' \} = \support \partial p_0  \cap \support \partial p_1,\) and let \( e_0'' \) be the unique edge in \( \support \partial p_1  \smallsetminus \{ e_0' \} \) that is parallel with \( e_0' \) (see Figure~\ref{subfig: final proof C}). 
        If \( e_0'' \in \tilde E_3, \) then~\( (\star )\) holds.
        If \( e_0'' \notin \tilde E_3, \) then the set \(P_2 \cap \support \hat \partial e_0'' \smallsetminus \{ p_1 \}\) must be non-empty. 
        But this is geometrically impossible, and thus concludes the proof.
    \end{subproof}
\end{proof}

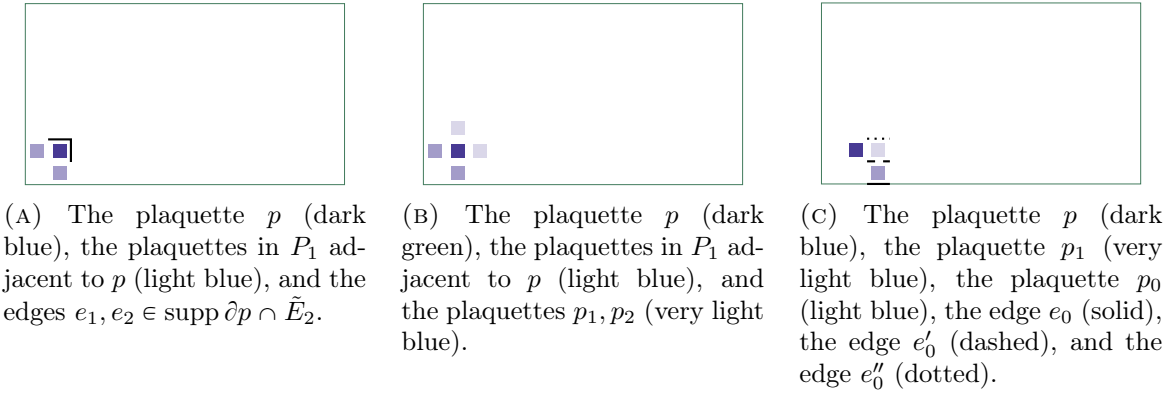
\begin{figure}[tp]
        \centering 
        \begin{subfigure}[t]{0.3\textwidth}\centering
    \begin{tikzpicture}[scale=0.6]
        
        \draw[line width=0.1mm,ForestGreen!60!black] (0,4) -- (0,0) -- (7,0) -- (7,4) -- cycle;

        \foreach \x in {0.5}
            \foreach \y in {0.5}
                \fill[BlueViolet] (\x+0.1,\y+0.1) rectangle ++(0.3,0.3);
        \foreach \x in {0}
            \foreach \y in {0.5}
                \fill[BlueViolet, opacity=.5] (\x+0.1,\y+0.1) rectangle ++(0.3,0.3);
        \foreach \x in {0.5}
            \foreach \y in {0}
                \fill[BlueViolet, opacity=.5] (\x+0.1,\y+0.1) rectangle ++(0.3,0.3);

         \draw[thick] (1,0.5) -- (1,1) -- (0.5,1);
    \end{tikzpicture}
    \caption{The plaquette \( p  \) (dark blue), the plaquettes in \( P_1 \) adjacent to \( p \) (light blue), and the edges \( e_1 ,e_2 \in \support \partial p \cap \tilde E_2. \)}
    \label{subfig: final proof A}
    \end{subfigure}
    \hfil
        \begin{subfigure}[t]{0.3\textwidth}\centering
    \begin{tikzpicture}[scale=0.6]
        
        \draw[line width=0.1mm,ForestGreen!60!black] (0,4) -- (0,0) -- (7,0) -- (7,4) -- cycle;

        \foreach \x in {0.5}
            \foreach \y in {0.5}
                \fill[BlueViolet] (\x+0.1,\y+0.1) rectangle ++(0.3,0.3);
        \foreach \x in {0}
            \foreach \y in {0.5}
                \fill[BlueViolet, opacity=0.5] (\x+0.1,\y+0.1) rectangle ++(0.3,0.3);
        \foreach \x in {0.5}
            \foreach \y in {0}
                \fill[BlueViolet, opacity=0.5] (\x+0.1,\y+0.1) rectangle ++(0.3,0.3);
                
        \foreach \x in {0.5}
            \foreach \y in {1}
                \fill[BlueViolet, opacity=0.2] (\x+0.1,\y+0.1) rectangle ++(0.3,0.3);
                
        \foreach \x in {1}
            \foreach \y in {0.5}
                \fill[BlueViolet, opacity=0.2] (\x+0.1,\y+0.1) rectangle ++(0.3,0.3);
                

    \end{tikzpicture}
    \caption{The plaquette \( p \) (dark green), the plaquettes in \( P_1 \) adjacent to \( p \) (light blue), and the plaquettes \( p_1,p_2 \) (very light blue).}
    \label{subfig: final proof B}
    \end{subfigure}
    \hfil
    \begin{subfigure}[t]{0.3\textwidth}\centering
    \begin{tikzpicture}[scale=0.6]
        
        \draw[line width=0.1mm,ForestGreen!60!black] (0,4) -- (0,0) -- (7,0) -- (7,4) -- cycle;

        \foreach \x in {0.5}
            \foreach \y in {0.5}
                \fill[BlueViolet] (\x+0.1,\y+0.1) rectangle ++(0.3,0.3);
                
                
        \foreach \x in {1}
            \foreach \y in {0.5}
                \fill[BlueViolet, opacity=0.2] (\x+0.1,\y+0.1) rectangle ++(0.3,0.3);
                
        \foreach \x in {1}
            \foreach \y in {00}
                {
                \fill[BlueViolet, opacity=0.5] (\x+0.1,\y+0.1) rectangle ++(0.3,0.3);
                }
          
         \draw[thick] (1,0) -- (1.5,0);
         \draw[thick, dashed] (1,0.5) -- (1.5,0.5);
         \draw[thick, dotted] (1,1) -- (1.5,1);
    \end{tikzpicture}
    \caption{The plaquette \( p   \) (dark blue), the plaquette \( p_1\) (very light blue), the plaquette \( p_0 \) (light blue), the edge \( e_0 \) (solid), the edge \( e_0' \) (dashed), and the edge \( e_0'' \) (dotted).}
    \label{subfig: final proof C}
    \end{subfigure}

    \caption{In the figure above, we illustrate the last step of the proof of Lemma~\ref{lemma: general technical lemma forms} for a given plaquette  \( p \in P_{\omega,\gamma,c,2} \subseteq P_2. \) } \label{figure: general technical lemma forms}
\end{figure}

The following lemma is a simple consequence of Lemma~\ref{lemma: general technical lemma forms}.
\begin{lemma}\label{lemma: basic inequality 6}
    
    Let \( \gamma \) be a path along the boundary of a rectangle with side lengths \( \ell_2 \geq \ell_1 \geq 8. \) Let \( \omega \in \Omega^2(B_N,\mathbb{Z}_n) \) 
    and assume that 
    \( \bigl| \support \delta \omega  \cap \support \gamma\bigr| <  |\support \gamma_R|.\)  
    Then
    \begin{equation*}
        \bigl|(\support \omega)^+\bigr| + \bigl| (\support \delta \omega)^+\bigr| + 3|P_{\omega,\gamma,c}| \geq 3\bigl|\support \delta \omega \cap \support \gamma\bigr|+2\|\omega^\gamma \|.
    \end{equation*}
\end{lemma}

\begin{proof}
    Let \( k \coloneqq \|\omega^\gamma \| , \)
    let \( P_1,P_2, \dots, P_k \) be the connected components of \( (\support \omega^\gamma)^+, \) and for \( j \in \{ 1,2, \dots, k \}, \) define \( \omega_j \coloneqq \omega|_{P_j}. \)
    Then, by definition, for each \( j \in \{ 1,2, \dots, k \}, \) we have
    \( |\support \delta \omega_j \cap \support \gamma| \geq 1. \) Since, by Lemma~\ref{lemma: split into components}~\ref{item: split into components iii}, 
    \begin{equation*}
         |\support \delta \omega \cap \support \gamma| = \sum_{j=1}^k |\support \delta \omega_j \cap \support \gamma|,
    \end{equation*}
    and, by assumption, we have \( \bigl| \support \delta \omega  \cap \support \gamma\bigr| < |\support \gamma_R| \) it follows that 
    \( \bigl|\support \delta  \omega_j \cap \support \gamma\bigr| < |\support \gamma_R|\) for all \( j \in \{ 1,2, \dots, k \}. \)
    Now fix any \( j \in \{ 1,2, \dots, k \}. \) 
    Since \[ \bigl|\support \delta  \omega_j \cap \support \gamma\bigr| < |\support \gamma_R|,\] we must have \( |V^{\gamma, \omega_j}| \geq 2.\)
    Applying Lemma~\ref{lemma: general technical lemma forms}, it follows that
    \begin{equation*}
        \bigl|(\support  \omega_j)^+\bigr| + \bigl|(\support \delta  \omega_j)^+ \bigr| + 3 |P_{\omega_j,\gamma,c}| \geq 3|\support \delta \omega_j \cap \support \gamma | + 2.
    \end{equation*} 
    Summing over all \( j \in \{ 1,2, \dots, k \}\) and using Lemma~\ref{lemma: split into components}, the desired conclusion follows.
\end{proof}

\begin{lemma}\label{lemma: Em1 motivation}
    Let \( \gamma \) be a path along the boundary of a rectangle with side lengths \( \ell_2 \geq \ell_1 \geq 8. \)
    Let \( \omega \in \Omega^2(B_N,\mathbb{Z}_n)  \) satisfy \( \omega^\gamma \neq 0, \) and assume that 
    \begin{equation*}
        \bigl|(\support \omega)^+\bigr| + \bigl| (\support \delta \omega)^+\bigr| + 2|P_{\omega,\gamma,c}| \leq 3|\support \delta \omega \cap \support \gamma |.
    \end{equation*}
    Then  \( \bigl|( \support \omega)^+ \bigr| \geq \ell_1 \) and \( |\support \delta \omega \cap \support \gamma_R| \geq \ell_1+2. \) 
\end{lemma}

\begin{proof}
    By Lemma~\ref{lemma: general technical lemma forms}, we have 
    \begin{equation*} 
        \bigl|(\support \omega)^+\bigr| + \bigl| (\support \delta  \omega)^+\bigr|+ 3|P_{ \omega,\gamma,c}|
        \geq 
        3|\support \delta  \omega \cap \support \gamma | + |V^{\gamma, \omega}| .
    \end{equation*}
    Since, by assumption, we also have 
    \begin{equation*} 
        \bigl|(\support \omega)^+\bigr| + \bigl| (\support \delta \omega)^+\bigr| + 3|P_{\omega,\gamma,c}|
        \leq 
        3|\support \delta \omega \cap \support \gamma | + |P_{\omega,\gamma,c}|,
    \end{equation*}
    it follows that \( |V^{\gamma,\omega}| \leq |P_{\omega,\gamma,c}|. \)

    If \( |V^{\gamma,\omega}| = 0, \) then, since \( \omega^\gamma \neq 0\) we must have \( \support \gamma_R \subseteq  \support \delta \omega^\gamma, \) and hence \( \bigl|(\support \omega)^+\bigr| \geq \ell_1.\)
    
    If \( |V^{\gamma,\omega}| \neq 0,\) then \( |V^{\gamma,\omega}| \geq 2, \) and thus (using that \( |V^{\gamma,\omega}| \leq |P_{\omega,\gamma,c}| \)), we have \(|P_{\omega,\gamma,c}| \geq 2. \)
    Assume that this is the case.
    For each  \( p \in P_{\omega,\gamma,c}, \) we have \( \support \partial p \cap \support \gamma \subseteq \support \delta \omega.\) Moreover, the restriction of \( \gamma_R \) to the set
    \begin{equation*}
        \support \gamma_R \smallsetminus \bigcup_{p \in P_{\omega,\gamma,c}} \support \partial p 
    \end{equation*}
    is the sum of \( |P_{\omega,\gamma,c}| \) paths with disjoint support.
    If not all of the edges in the support of such a path \( \gamma' \) are in \( \support \delta \omega\), then \( \pigl| \bigl\{ v \in V^{\gamma,\omega} \colon \support \hat \partial v \cap \support \gamma' \neq \emptyset \bigr\} \pigr| \geq 2 \) (that is, the path contains at least two elements in \( V^{\omega,\gamma} \)). Since \( |V^{\gamma,\omega}| \leq |P_{\omega,\gamma,c}| , \) for at least one of these paths \( \gamma' \) we must have \( \support \gamma' \subseteq \support \delta \omega.\) Noting that \( |\support \gamma'| \geq \ell_1-2, \) we obtain
    \begin{equation*}
        |\support \delta \omega \cap \support \gamma_R| 
        \geq
        |\support \gamma'| + 2|P_{\omega,\gamma_R,c}|   . 
    \end{equation*}
    Using Lemma~\ref{lemma: basic inequality 1 forms}, it follows that
    \begin{equation*}
        \begin{split}
            &\bigl| (\support \omega)^+ \bigr| \geq |\support \delta \omega \cap \support \gamma_R| - |P_{\omega,\gamma_R,c}| 
            \geq 
            |\support \gamma'| + |P_{\omega,\gamma_R,c}|
            \\&\qquad\geq 
            |\support \gamma'| + |P_{\omega,\gamma,c}|
             \geq
            (\ell_1-2) + 2 = \ell_1.
        \end{split}
    \end{equation*}
    This concludes the proof.
\end{proof}

\section{The contribution of bad events}\label{section: bad events}

The main purpose of this section is to prove the following proposition which gives an upper bound on the contribution to the expectation of the Wilson line/loop observable from \( \omega \in \Omega^2(B_N,\mathbb{Z}_n) \smallsetminus \mathcal{E} \), where
\begin{equation}\label{eq: def E2 forms ii}
    \begin{split} 
        \mathcal{E} \coloneqq \bigl\{ \omega \in \Omega^2(B_N,\mathbb{Z}_n) \colon
        \bigl|(\support \omega^{\gamma,2})^+\bigr|  = \| \omega^\gamma\| \text{ and }  |P_{\omega,\gamma,c}| = 0  &\bigr\}.
    \end{split}
\end{equation}
(Recall that $\omega^\gamma$ and $\omega^{\gamma,2}$ were defined in Section \ref{connectedsubsec} and $P_{\omega,\gamma,c}$ was defined in \eqref{Pomegagammacdef}.)
In words, \( \mathcal{E}\) is the event that each connected component of \( (\support \omega^{\gamma,2})^+ \) consists of exactly one plaquette, and that no plaquette in  \( P_{\omega,\gamma,c}\) is in \( (\support \omega^{\gamma,2})^+. \)

\begin{proposition}\label{proposition: useful upper bound forms ii}
    Let \( \gamma \) be a path along the boundary of a rectangle with side lengths \( \ell_2 \geq \ell_1 \geq 7. \)
    Assume that~\eqref{assumption: 1} holds.
    Then
    \begin{equation*}
        \begin{split}
            &
            \mathbb{E}_\varphi    \pigl[\widehat{L_\gamma}(\omega) \cdot \mathbb{1}( \omega \notin \mathcal{E} ) \pigr] 
            \leq     C^{(1)}_{\gamma,\beta,\kappa,m}\, \xi_\kappa^{| \gamma|}    \zeta_\beta,
        \end{split}
    \end{equation*}  
    where
    \begin{equation}\label{eq: def C''''}
        C^{(1)}_{\gamma,\beta,\kappa,m} \coloneqq C^{(1')}_{\gamma,\beta,\kappa,m}
        +
        C^{(1'')}_{\gamma,\beta,\kappa,m}
        +
        C^{(1''')}_{\gamma,\beta,\kappa,m}
        +
        C^{(1'''')}_{\gamma,\beta,\kappa,m},
    \end{equation}
    and \( C^{(1')}_{\gamma,\beta,\kappa,m}, \) \( C^{(1'')}_{\gamma,\beta,\kappa,m} ,\) \(  C^{(1''')}_{\gamma,\beta,\kappa,m}, \) and  \( C^{(1'''')}_{\gamma,\beta,\kappa,m} \) are defined in~\eqref{eq: def C},~\eqref{eq: def C'},~\eqref{eq: def C''}, and~\eqref{eq: def C'''} respectively, and \( \xi_\kappa   \) and \( \zeta_\beta \) are defined by~\eqref{eq: zeta def}.
\end{proposition}

Before we give a proof of Proposition~\ref{proposition: useful upper bound forms ii}, we define auxiliary events \( \mathcal{E}' ,\) \( \mathcal{E}'', \) \( \mathcal{E}''', \) and \( \mathcal{E}'''', \) and state and prove a few helpful lemmas.
We define
\begin{align*}
 &   \mathcal{E}' \coloneqq \bigl\{ \omega \in \Omega^2(B_N,\mathbb{Z}_n) \colon \bigl|(\support \omega^\gamma)^+ \bigr| \geq |\support \delta \omega \cap \support \gamma|-|P_{\omega,\gamma,c}|+1\bigr\},
    \\
&    \mathcal{E}'' \coloneqq \bigl\{ \omega \in \Omega^2(B_N,\mathbb{Z}_n) \colon |\support \delta \omega \cap \support \gamma| - |P_{\omega,\gamma,c}| \geq \| \omega^\gamma\| + 1 \bigr\},
    \\
&    \mathcal{E}''' \coloneqq \bigl\{ \omega \in \Omega^2(B_N,\mathbb{Z}_n) \colon |P_{\omega,\gamma,c}| \neq 0 \bigr\},
    \\
&    \mathcal{E}'''' \coloneqq \bigl\{ \omega \in \Omega^2(B_N,\mathbb{Z}_n) \colon \omega^{\gamma,2} \neq \omega^\gamma \bigr\}.
\end{align*}  
The main motivation for introducing these events is the following equality, which is established in the proof of Proposition~\ref{proposition: useful upper bound forms ii}.
\begin{equation*}
    \Omega^2(B_N,\mathbb{Z}_n) \smallsetminus \mathcal{E} = \mathcal{E}' \cup \mathcal{E}'' \cup \mathcal{E}''' \cup \mathcal{E}''''.
\end{equation*}
%

\begin{lemma}\label{lemma: upper bound on expectation of two events}
    Let \( \gamma \) be a path along the boundary of a rectangle with side lengths \( \ell_2 \geq \ell_1 \geq 7. \) Assume that~\eqref{assumption: 3} holds.
    Then the following four inequalities hold.
    \begin{equation}\label{eq: first inequality}
        \begin{split}
            &\mathbb{E}_\varphi  \pigl[\widehat{L_\gamma}(\omega) \cdot \mathbb{1}(\omega \in \mathcal{E}')\pigr] 
            \leq
            C^{(1')}_{\gamma,\beta,\kappa,m}\,\xi_\kappa^{| \gamma|}
            \zeta_\beta,
        \end{split}
    \end{equation}
    \begin{equation}\label{eq: second inequality}
        \begin{split}
            &\mathbb{E}_\varphi  \pigl[\widehat{L_\gamma}(\omega) \cdot \mathbb{1}(\omega \in \mathcal{E}''\smallsetminus \mathcal{E}')\pigr] 
            \leq
            C^{(1'')}_{\gamma,\beta,\kappa,m}\,\xi_\kappa^{| \gamma|} \zeta_\beta,
        \end{split}
    \end{equation}
    \begin{equation}\label{eq: third inequality}
        \begin{split}
            &\mathbb{E}_\varphi  \pigl[\widehat{L_\gamma}(\omega) \cdot \mathbb{1}\bigl(\omega \in \mathcal{E}'''\smallsetminus (\mathcal{E}' \cup \mathcal{E}'')\bigr)\pigr] 
            \leq
            C^{(1''')}_{\gamma,\beta,\kappa,m}\,\xi_\kappa^{| \gamma|} \zeta_\beta,
        \end{split}
    \end{equation}
    \begin{equation}\label{eq: fourth inequality}
        \begin{split}
            &\mathbb{E}_\varphi  \pigl[\widehat{L_\gamma}(\omega) \cdot \mathbb{1}\bigl(\omega \in \mathcal{E}''''\smallsetminus (\mathcal{E}' \cup \mathcal{E}'' \cup \mathcal{E}''')\bigr)\pigr] 
            \leq
            C^{(1'''')}_{\gamma,\beta,\kappa,m}\, \xi_\kappa^{| \gamma|} \zeta_\beta,
        \end{split}
    \end{equation}
    where 
    \begin{equation}\label{eq: def C}
        \begin{split}
        C^{(1')}_{\gamma,\beta,\kappa,m}\, \coloneqq
        & \;
        (16m)^2 \cdot \frac{\bigl( 1 + (16m)^2 \zeta_\beta\bigr)^{| \gamma|}\bigl(1+ (16m)^2 \zeta_\beta\xi_\kappa^{-2 } \bigr)^{|\mathcal{P}_{\gamma,c}|}  -1 }{\xi_\kappa(1-\xi_\kappa) \bigl( 1-(16m)^2 \zeta_\beta\xi_\kappa^{-1}\bigr) }
            \\&\qquad+ 
            (16m)^2 \cdot \frac{\bigl( 1 + (16m)^4 \zeta_\beta^2 \xi_\kappa^{-1}\bigr)^{| \gamma|} \pigl( 1 + (16m)^{2} \zeta_\beta \xi_\kappa^{-2}\pigr)^{|\mathcal{P}_{\gamma,c}|}  -1}{(1-\xi_\kappa)(1 - (16m)^2 \zeta_\beta ) }
            ,
         \end{split}
    \end{equation} 
    \begin{equation}\label{eq: def C'}\begin{split}
        C^{(1'')}_{\gamma,\beta,\kappa,m}\, \coloneqq  
        & \;
        \frac{ (16m)^4 }{1-\xi_\kappa}\biggl(  
        | \gamma|  \zeta_\beta  
        +
        2
        |\mathcal{P}_{\gamma,c}|  \zeta_\beta \xi_\kappa^{-2} 
        \biggr) 
            \\&\qquad\quad\times  \bigl(1 + (16m)^2\zeta_\beta \xi_\kappa^{-2} \bigr)^{|\mathcal{P}_{\gamma,c}| } \bigl( 1 + (16m)^2 \zeta_\beta\bigr)^{| \gamma|}
        ,
    \end{split}\end{equation} 
    \begin{equation}\label{eq: def C''}
    	\begin{split}
        	&C^{(1''')}_{\gamma,\beta,\kappa,m}\, \coloneqq 
        	\frac{
            \bigl( 1 + (16m)^2 \zeta_\beta\bigr)^{|\mathcal{P}_{\gamma,c}|}-1 }{\zeta_\beta} 
            \bigl(1 + (16m)^2 \zeta_\beta \xi_\kappa^2\bigr)^{| \gamma|} 
           \end{split}
    \end{equation} 
    and 
    \begin{equation}\label{eq: def C'''} 
    \begin{split}
        C^{(1'''')}_{\gamma,\beta,\kappa,m} &\coloneqq  
        \frac{(16m)^{4}\zeta_\beta  | \gamma|  \bigl( 1 + (16m)^{2} \zeta_\beta  \xi_\kappa^{2}  \bigr)^{| \gamma|}}{1-\xi_\kappa}
        \\&\qquad\quad\times\biggl(
        \frac{\xi_\kappa^{4}}{(1-(16m)^{2} \zeta_\beta \xi_\kappa^{-1})}
        +
        \frac{(16m)^{8} \zeta_\beta^{4} }{(1-(16m)^{2} \zeta_\beta )} 
        \biggr) .
            \end{split}
    \end{equation}
\end{lemma}

Before we give a proof of Lemma~\ref{lemma: upper bound on expectation of two events}, we state and prove a few useful lemmas, and introduce some useful notation. For \( i,j \geq 0, \) we define
\begin{equation}\label{eq: Aij}
        A_{j,i} \coloneqq \bigl\{ E \subseteq \support \gamma \colon \exists \omega \in \Omega^2(B_N,\mathbb{Z}_n) \colon \support \delta \omega \cap \support \gamma = E ,\, |E | = j ,\, |P_{\omega,\gamma,c}|=i \bigr\},
    \end{equation}
    \begin{equation}\label{eq: Aji'}
        A_{j,i}' \coloneqq \bigl\{ E \subseteq \support \gamma \colon \exists \omega \in \mathcal{E}''\smallsetminus\mathcal{E}' \colon \support \delta \omega \cap \support \gamma = E ,\, |E| = j ,\, |P_{\omega,\gamma,c}|=i \bigr\},
    \end{equation}
    and
    \begin{equation}\label{eq: Aj0''}  
        \begin{split}
            A_{j,0}'' \coloneqq &\bigl\{ \hat E \subseteq \support \gamma \colon \exists \omega \in \mathcal{E}''''\smallsetminus(\mathcal{E}' \cup \mathcal{E}'' \cup \mathcal{E}''') \colon 
            \\&\qquad \exists e \in \support \gamma \colon \hat E = \support \delta \omega \cap \support \gamma \sqcup \{ e \},\, |\hat E| = j+1 ,\, \\&\qquad|P_{\omega,\gamma,c}|=0,\, |\support \omega \cap \support \hat \partial e| \geq 2 \bigr\}.
        \end{split}
    \end{equation}

\begin{lemma}\label{lemma: Aji bound}
    Let \( j \in \bigl\{ 0,1, \dots, | \gamma| \bigr\} \) and \( i \in \bigl\{ 0,1, \dots, |\mathcal{P}_{\gamma,c}| \bigr\}. \) Then 
    \begin{equation*}
        \begin{split}
            &|A_{j,i}|  \leq \binom{| \gamma|}{j-2i}\binom{|\mathcal{P}_{\gamma,c}|}{i}.
        \end{split}
    \end{equation*}
\end{lemma}

\begin{proof}
    Let \( E \in A_{j,i}. \) Then \( |E| = j, \) and \( E \subseteq \support \gamma. \)
    Let \( \omega \in \Omega^2(B_N,\mathbb{Z}_n) \) be such that \( \support \delta \omega \cap \support \gamma = E \) and \( |P_{\omega,\gamma,c}|=i. \) Then there must exist \( i \) plaquettes \(p \in  \mathcal{P}_{\gamma,c} \) such that \( |\support \partial p \cap E| = |\support \partial p \cap \support \delta \omega \cap \support \gamma| = 2. \) Note that there are \( \binom{|\mathcal{P}_{\gamma,c}|}{i} \) ways to choose \( i \) plaquettes in \( \mathcal{P}_{\gamma,c}. \) At the same time, given \( P_{\omega,\gamma,c},\) the number of ways to choose the \( j-2i \) edges in \( E  \) that are not in the boundary of any of the plaquettes in \( P_{\omega,\gamma,c}\) is at most \( \binom{| \gamma|}{j-2i} . \) Combining these observations, we obtain the desired conclusion.
\end{proof}

\begin{lemma}\label{lemma: Aji' bound}
    Let \( j \in \bigl\{ 0,1, \dots, |\support \gamma| \bigr\} \) and \( i \in \bigl\{ 0,1, \dots, |\mathcal{P}_{\gamma,c}| \bigr\}. \) Then 
    \begin{equation*}
        \begin{split}
            &\pigl|A_{j,i}'\pigr| 
            \leq 
            \binom{|\mathcal{P}_{\gamma,c}|}{i}
            \binom{| \gamma|}{j-2i-1}
            (j-1).
        \end{split}
    \end{equation*}
\end{lemma}

\begin{proof}
    Let \( E \in A_{j,i}'. \) Then \( |E| = j, \) and \( E \subseteq \support \gamma. \)
    Let \( \omega \in \mathcal{E}'' \smallsetminus \mathcal{E}' \) be such that \( \support \delta \omega \cap \support \gamma = E \) and \( |P_{\omega,\gamma,c}| = i. \)
    Then there must exist \( i \) plaquettes \(p \in  \mathcal{P}_{\gamma,c} \) such that \( |\support \partial p \cap E| = |\support \partial p \cap \support \delta \omega \cap \support \gamma| = 2, \) and hence \( 2i \) edges in \( E \) that are in the boundary of such a plaquette. Let \( E' \) be the set of these edges.
    There are \( \binom{|\mathcal{P}_{\gamma,c}|}{i} \) ways to choose \( i \) plaquettes in \( \mathcal{P}_{\gamma,c}. \) Equivalently, there are \( \binom{|\mathcal{P}_{\gamma,c}|}{i} \) ways to choose the set \( E'.\)
    Since \( \omega \in \mathcal{E}''\smallsetminus \mathcal{E}',\) at least one edge \( e \in E \smallsetminus E' \) is adjacent to some other edge in \( E. \) Since the number of ways to choose \( E \smallsetminus (E' \cup \{ e \})\) is at most \(\binom{| \gamma|}{j-2i-1}\) and the number of ways to choose \( e\) given \( E \smallsetminus( E' \cup \{ e \})\) is at most \( (j-1),\) it follows that the number of ways to choose the edges in \( E \smallsetminus E' \) is at most \( \binom{| \gamma|}{j-2i-1} \cdot (j-1).\)
    Combining the two upper bounds given above, we obtain the desired conclusion. 
\end{proof}

\begin{lemma}\label{lemma: Aji'' bound}
    Let \( j \in \{ 0,1, \dots, | \gamma| \} \) and \( i \in \bigl\{ 0,1, \dots, |\mathcal{P}_{\gamma,c}| \bigr\}. \) Then
    \(
            {|A_{j,0}''| 
            \leq \binom{| \gamma|}{j+1} .}
        \)   
\end{lemma}

\begin{proof}
    Let \( \hat E \in A_{j,0}''. \) Then \( |\hat E| = j+1 \) and \( \hat E \subseteq \support \gamma. \) Since there are exactly \( \binom{| \gamma|}{j+1} \) such sets, we obtain the desired conclusion. 
\end{proof}

\begin{lemma}\label{lemma: letter inequalities}
    Let \( \omega \in \Omega^2(B_N,\mathbb{Z}_n), \) and define
    \begin{equation}\label{eq: letter definitions}
        \begin{cases}
            j \coloneqq |\support \delta \omega \cap \support \gamma | \cr 
            i \coloneqq |P_{\omega,\gamma,c}| \cr 
            k \coloneqq \bigl| (\support \omega^\gamma)^+ \bigr| \cr 
            k' \coloneqq \bigl| (\support \delta \omega^\gamma)^+ \bigr|. 
        \end{cases}
    \end{equation}  
    Then the following statements hold.
    \begin{enumerate}[label=\upshape(\roman*)]
        \item If \( \omega \in \mathcal{E}', \) then \( k' \geq j, \) \( k \geq j-i+1, \)  \( k+ k' \geq 3j-3i, \) and \( \support \delta\omega \cap \support \gamma \in A_{j,i}. \)
        \label{item: letter lemma i}
        
        \item If \( \omega \in \mathcal{E}''\smallsetminus \mathcal{E}', \) then \( k' \geq j, \) \( k = j-i,\) \(k' \geq 2k, \) \( j \geq \max(2i+1,2), \) and \( \support \delta \omega \cap \support \gamma  \in A_{j,i}'. \)
        \label{item: letter lemma ii}
        
        \item If \( \omega \in \mathcal{E}'''\smallsetminus (\mathcal{E}' \cup \mathcal{E}''), \) then \( k' \geq j, \) \( k = j-i,\) \(k' = 4k, \) \( i \geq 1, \) and \( \support \delta \omega \cap \support \gamma \in A_{j,i}. \)
        \label{item: letter lemma iii}
        
        \item If \( \omega \in \mathcal{E}''''\smallsetminus (\mathcal{E}' \cup \mathcal{E}'' \cup \mathcal{E}'''), \) then \( k' \geq j, \) \( k = j,\) \(k' = 4k, \) and \( i = 0,\) and \( \support \delta \omega \cap \support \gamma  \in A_{j,0}. \)
        Moreover, there is \( e \in \support \gamma \smallsetminus \support \delta \omega \) such that, with \( E \coloneqq (\support \delta \omega \cap \support \gamma) \sqcup \{ e \} \in A_{j,0}'', \) we have
        \( \bigl| (\support \omega^E)^+\bigr|\geq j+2, \)
        \( \bigl| (\support \delta \omega^E)^+\bigr|\geq 4j, \) and 
        \( \bigl| (\support \omega^E)^+\bigr| + \bigl| (\support \delta \omega^E)^+\bigr|\geq 5j +2. \)
        \label{item: letter lemma iv}
    \end{enumerate}  
\end{lemma}

\begin{proof}
    Since, by definition, we have \( \support \delta \omega \cap \support \gamma \subseteq (\support  \delta\omega^\gamma)^+ , \) it follows that  \( k' \geq j. \) Moreover, by definition, we have \( \support \delta \omega \cap \support \gamma \in A_{j,i}. \) 
    
    If \( \omega \in \mathcal{E'}, \) then  \( \bigl|(\support \omega^\gamma)^+ \bigr| \geq |\support \delta \omega \cap \support \gamma| -  |P_{\omega,\gamma,c}| +1, \) or equivalently, \( k \geq j-i+1\). Also, by Lemma~\ref{lemma: general technical lemma forms}, we have \( k + k' \geq 3j-3i. \)
    This completes the proof of~\ref{item: letter lemma i}.

    \begin{sublemma}\label{sublemma: 1}
        If \( \omega \notin \mathcal{E}', \) then \( k = j-i\) and \(k' \geq 2k . \)
    \end{sublemma}
    \begin{subproof}
        Let \( \omega \notin \mathcal{E}'. \) Then, by definition, we have \( \bigl|(\support \omega^\gamma)^+ \bigr| \leq |\support \delta \omega \cap \support \gamma| -  |P_{\omega,\gamma,c}| . \) Recalling Lemma~\ref{lemma: basic inequality 1 forms}, it follows that \( \bigl|(\support \omega^\gamma)^+ \bigr| = |\support \delta \omega \cap \support \gamma| -  |P_{\omega,\gamma,c}| , \) or equivalently, \( k = j-i\). Using Lemma~\ref{lemma: general technical lemma forms}, it follows that \( k + k' \geq 3j-3i = 3k, \) or equivalently, \(k' \geq 2k . \) This concludes the proof.
    \end{subproof}
    
    If \( \omega \in \mathcal{E}'', \) then \( |\support \delta \omega \cap \support \gamma| - |P_{\omega,\gamma,c}| \geq \| \omega^\gamma\|+1.\)
    Consequently, if 
    \( \omega \in \mathcal{E}''\smallsetminus \mathcal{E}', \) then \( \bigl|(\support \omega^\gamma)^+ \bigr| \geq \| \omega^\gamma\|+1. \) In other words, some connected component of \( (\support \omega^\gamma)^+ \) must contain at least two plaquettes.
    Since \( \bigl| (\support \omega^\gamma)^+\bigr| = |\support \delta \omega \cap \support \gamma| - |P_{\omega,\gamma,c}|, \) there must exist \( \{ e,e' \} \subseteq \support \gamma \) such that if \( \{ p,p'\} \subseteq (\support \omega)^+ \) are defined by \( \{ p \} = \support \hat \partial e \cap \support \omega \) and \( \{ p' \} = \support \hat \partial e' \cap \support \omega, \) then \( p \sim p'. \) Since \( p \neq p', \) \( e \) and \( e' \) cannot both be adjacent to some common corner of \( \gamma. \) This implies in particular that \( j \geq 2i+1.\) 
    Finally, note that if 
    \( \omega \in  \mathcal{E}'' \smallsetminus \mathcal{E}', \) then \( \support \delta \omega \cap \support \gamma \in A_{j,i}'. \) This concludes the proof of~\ref{item: letter lemma ii}.

    \begin{sublemma}\label{sublemma: 2}
        If \( \omega \notin \mathcal{E}'' \cup \mathcal{E}', \) then \( k' = 4k . \) 
    \end{sublemma}

    \begin{subproof}
        Let \( \omega \notin \mathcal{E}'' \cup \mathcal{E}'. \) Then \( (\support \omega^\gamma)^+ \) consists of connected components, each of size one, and hence \( \bigl|(\support \delta \omega^\gamma)^+\bigr| = 4\bigl|(\support  \omega^\gamma)^+\bigr|, \) or equivalently \( k' = 4k . \) 
    \end{subproof}

    If \( \omega \in \mathcal{E}''', \) then \( |P_{\omega,\gamma,c}| \geq 1,  \) or equivalently \( i \geq 1. \) This concludes the proof of~\ref{item: letter lemma iii}.

    \begin{sublemma}\label{sublemma: 3}
        If \( \omega \notin \mathcal{E}''', \) then \( i = 0. \) 
    \end{sublemma}
    
    \begin{subproof}
        Let \( \omega \notin \mathcal{E}'''.\) Then \( |P_{\omega,\gamma,c}| =0,  \) and hence \( i = 0. \) 
    \end{subproof}

    Now assume that \( \omega \in \mathcal{E}'''' \smallsetminus (\mathcal{E}'  \cup \mathcal{E}''  \cup \mathcal{E}''' ). \)
    By combining Claim~\ref{sublemma: 1} and Claim~\ref{sublemma: 3}, it follows that \( k = \bigl| (\support \omega^\gamma)^+ \bigr|  = j, \) and by Claim~\ref{sublemma: 2}, we have \( k' = \bigl| (\support \delta\omega^\gamma)^+ \bigr| = 4j. \)
    Since \( \omega \in \mathcal{E}'''' \smallsetminus (\mathcal{E}'  \cup \mathcal{E}''  \cup \mathcal{E}''' ),\) by definition, there is \( e \in \support \gamma \smallsetminus \support \delta \omega \) such that \( |\support \hat \partial e \cap \support \omega | \geq 2, \) and \(  E \coloneqq \support \delta \omega \cap \support \gamma \sqcup \{ e \}\in A_{j,0}''. \) 
    Consequently, we must have 
    \begin{equation*}
        \bigl| (\support \omega^E)^+ \bigr| = \bigl| (\support \omega^\gamma)^+\bigr| + 2 = j+2
    \end{equation*}
    and 
    \begin{equation*}
        \bigl| (\support \delta\omega^E)^+ \bigr| \geq \bigl| (\support \delta \omega^\gamma)^+\bigr| =4j, 
    \end{equation*}
    and hence 
    \begin{equation*}
        \bigl| (\support \omega^E)^+ \bigr|+\bigl| (\support \delta\omega^E)^+ \bigr| \geq 
        (j+2)+4j
        = 5j+2.
    \end{equation*}
    This concludes the proof of~\ref{item: letter lemma iv}.
\end{proof}

\begin{proof}[Proof of Lemma~\ref{lemma: upper bound on expectation of two events}]
  For any event \( \mathcal{E}_0 \subseteq \Omega^2(B_N,\mathbb{Z}_n) ,\) we have 
    \begin{equation}\label{eq: upper bound on expectation of two events 1}
        \begin{split}
            &\mathbb{E}_\varphi  \biggl[\widehat{L_\gamma}(\omega) \cdot \mathbb{1}(\omega \in \mathcal{E}_0)\biggr] 
            \\&\qquad
            \leq 
            \sum_{i=0}^{|\mathcal{P}_{\gamma,c}|} \sum_{j = \max(1,2i)}^{| \gamma|}
            \mathbb{E}_\varphi  \biggl[\widehat{L_\gamma}(\omega)
            \cdot \mathbb{1}(\omega \in \mathcal{E}_0,\, |\support \delta \omega \cap \support \gamma | = j,\, |P_{\omega,\gamma,c}|=i)\biggr].
        \end{split}
    \end{equation} 
    Let us show that~\eqref{eq: first inequality} holds. Fix any \( j \in \{1,2, \dots, | \gamma| \} \) and \( i \in \pigl\{ 0, \dots, \min\bigl(j, |\mathcal{P}_{\gamma,c}|\bigr) \pigr\}. \)
    If \( \omega \in \mathcal{E}' \) is such that \( |\support \delta \omega \cap \support \gamma | = j \) and \( |P_{\omega,\gamma,c}|=i, \) then, by~\ref{item: letter lemma i} of Lemma~\ref{lemma: letter inequalities}, there is \( E \in A_{j,i} \) such that \( \support \delta \omega^\gamma \cap \support \gamma = E, \) and we have \( |(\support \omega^\gamma)^+| \geq j-i+1 \) and \( | (\support \delta \omega^\gamma)^+| \geq \max\pigl(j, 3j-3i-\bigl|(\support \omega^\gamma)^+\bigr|\pigr) . \) 
    Applying Lemma~\ref{lemma: flip a set forms iii} to $\omega^\gamma$ with \( E, \) \( k \geq j-i+1 \geq 1,\) \( k' \geq \max\pigl(j, 3j-3i-k\pigr), \) and \( k'' = j = |E|, \) we thus obtain
    \begin{align}\nonumber
            &\mathbb{E}_\varphi  \pigl[ \widehat{L_\gamma}(\omega)
            \cdot \mathbb{1}(\omega \in \mathcal{E}',\, |\support \delta \omega \cap \support \gamma | = j,\, |P_{\omega,\gamma,c}|=i)\pigr]
            \\[1ex] \nonumber
            &\qquad\leq 
            \smash{\sum_{E \in A_{j,i}} 
            \sum_{k = j-i+1}^\infty
            \sum_{k' = \max(j,3j-3i-k)}^\infty} \mathbb{E}_\varphi  \pigl[ \widehat{L_\gamma}(\omega)
            \cdot \mathbb{1}( \support \delta \omega \cap \support \gamma = E,\, |\support \omega^E| = 2k 
            \\ \nonumber
            &\hspace{19em} |\support \delta \omega^E|= 2k',\, |\support \delta \omega \cap \support \gamma| = j)\pigr]
            \\ \nonumber
            &\qquad\leq 
            |A_{j,i}| 
            \sum_{k = j-i+1}^\infty
            \sum_{k' = \max(j,3j-3i-k)}^\infty
            (16m)^{2k} \zeta_\beta^k \xi_\kappa^{|\gamma | + k' - 2j}
            \\\label{eq: upper bound on expectation of two events 2} &\qquad\leq
            \binom{|\gamma |}{j-2i}\binom{|\mathcal{P}_{\gamma,c}|}{i} 
            \sum_{k = j-i+1}^\infty
            \sum_{k' = \max(j,3j-3i-k)}^\infty
            (16m)^{2k} \zeta_\beta^k \xi_\kappa^{|\gamma | + k' - 2j},
    \end{align} 
    where the last inequality follows from applying Lemma~\ref{lemma: Aji bound}.
    Summing over all \( i \in \bigl\{ 0,1, \dots, |\mathcal{P}_{\gamma,c}| \bigr\}, \) and \( j \in \bigl\{ \max(1,2i), \dots,|\gamma | \bigr\} \) (see Claim~\ref{sublemma: upper bound on expectation of two events 1} in the Appendix), we see that
    \begin{align}\nonumber
            &
            B_1 \coloneqq \sum_{i = 0}^{|\mathcal{P}_{\gamma,c}|} \sum_{j=\max(1,2i)}^{|\gamma |} \binom{|\gamma |}{j-2i}\binom{|\mathcal{P}_{\gamma,c}|}{i} 
            \sum_{k = j-i+1}^\infty
            \sum_{k' = \max(j,3j-3i-k)}^\infty
            \!\!\!\!\!\!\!\! \!\!(16m)^{2k} \zeta_\beta^k \xi_\kappa^{|\gamma | + k' - 2j}
            \\ \nonumber &\qquad\leq
            \frac{(16m)^2\zeta_\beta\xi_\kappa^{-1}\xi_\kappa^{|\gamma |}}{(1 - \xi_\kappa)(1-(16m)^{2} \zeta_\beta\xi_\kappa^{-1})}
            \Bigl( \bigl( 1+(16m)^{2} \zeta_\beta\xi_\kappa^{-2} \bigr)^{|\mathcal{P}_{\gamma,c}|}
            \pigl( 1 + (16m)^{2}\zeta_\beta \pigr)^{|\gamma |} -1 \Bigr)
            \\ \label{eq: upper bound on expectation of two events 3} &\qquad\quad +
            \frac{(16m)^{2}\zeta_\beta\xi_\kappa^{|\gamma |}}{(1 - \xi_\kappa)(1 - (16m)^{2} \zeta_\beta)} 
            \Bigl( \bigl( 1 + (16m)^{2}\zeta_\beta  \xi_\kappa^{-2} \bigr)^{|\mathcal{P}_{\gamma,c}|} 
             \pigl( 1 + (16m)^{4} \zeta_\beta^{2} \xi_\kappa^{- 1}  \pigr)^{|\gamma |} -1 \Bigr). 
    \end{align}
    Combining~\eqref{eq: upper bound on expectation of two events 1},~\eqref{eq: upper bound on expectation of two events 2}~and~\eqref{eq: upper bound on expectation of two events 3}, we obtain~\eqref{eq: first inequality} as desired.

    We now show that~\eqref{eq: second inequality} holds. 
    To this end, fix any  \( i \in \bigl\{ 0, \dots, |\mathcal{P}_{\gamma,c}|  \bigr\} \) and  \( j \in \{\max(2i+1,2), \dots, |\gamma | \}. \)
    If \( \omega \in \mathcal{E}'' \smallsetminus \mathcal{E}'\) is such that \( |\support \delta \omega \cap \support \gamma | = j \) and \( |P_{\omega,\gamma,c}|=i, \) then, by~\ref{item: letter lemma ii} of Lemma~\ref{lemma: letter inequalities}, there is \( E \in A_{j,i}' \) such that \( \delta \omega^\gamma \cap \support \gamma = E, \)
    and we have \( |(\support \omega^\gamma)^+| = j-i, \) and \( | (\support \delta \omega^\gamma)^+| \geq \max\bigl(j,2(j-i)\bigr) . \) 
    Applying Lemma~\ref{lemma: flip a set forms iii} with \( E, \) \( k = j-i \geq 1, \)  \( k' \geq 2j-2i, \) and \( k''=j,\) we thus obtain
    \begin{equation*}
        \begin{split}
            &\mathbb{E}_\varphi  \pigl[ \widehat{L_\gamma}(\omega)
            \cdot \mathbb{1}(\omega \in \mathcal{E}'' \smallsetminus \mathcal{E}',\, |\support \delta \omega \cap \support \gamma | = j,\, |P_{\omega,\gamma,c}|=i)\pigr]
            \\[1ex]&\qquad\leq 
            \smash{\sum_{E \in A_{j,i}'} \sum_{k'=\max(j,2j-2i)}^\infty}
            \mathbb{E}_\varphi  \pigl[ \widehat{L_\gamma}(\omega)
            \cdot \mathbb{1}( \support \delta \omega \cap \support \gamma = E,\, |\support \omega^E| = 2(j-i) \\ &\hspace{19em} |\support \delta \omega^E|= 2k',\, |\support \delta \omega \cap \support \gamma| = j)\pigr]
            \\&\qquad\leq 
            |A_{j,i}'| \sum_{k'=\max(j,2j-2i)}^\infty
            (16m)^{2(j-i)} \zeta_\beta^{j-i} \xi_\kappa^{|\gamma | + k' - 2j}
            \\&\qquad\leq 
            (j-1) \binom{|\gamma |}{j-2i-1}\binom{|\mathcal{P}_{\gamma,c}|}{i} \sum_{k'=\max(j,2j-2i)}^\infty
            (16m)^{2(j-i)} \zeta_\beta^{j-i} \xi_\kappa^{|\gamma | + k' - 2j}
            ,
        \end{split}
    \end{equation*}   
    where the last inequality follows from applying Lemma~\ref{lemma: Aji' bound}.
    
    Summing over all \( i \in \bigl\{ 0,1, \dots, |\mathcal{P}_{\gamma,c}| \bigr\}, \) \( j \in \bigl\{ \max(2i+1,2), \dots,|\gamma | \bigr\}, \) and using that \( \sum_{j=0}^n \binom{n}{j} \cdot j p^j = np(1+p)^{n-1}\) (see Claim~\ref{sublemma: upper bound on expectation of two events 2} in the Appendix), we obtain
    \begin{align}
        &B_2 \coloneqq \sum_{i=0}^{|\mathcal{P}_{\gamma,c}|} \sum_{j = \max(2i+1,2)}^{|\gamma |} \!\!\!\!\!\!(j-1) \binom{|\gamma |}{j-2i-1}\binom{|\mathcal{P}_{\gamma,c}|}{i} 
        \sum_{k'=\max(j,2j-2i)}^\infty \!\!\!\!\!\!\!\!
        (16m)^{2(j-i)} \zeta_\beta^{j-i} \xi_\kappa^{|\gamma | + k' - 2j}\label{eq: def B2}
        \\\nonumber
        &\qquad\leq
        \frac{\xi_\kappa^{|\gamma |} \cdot (16m)^4 \zeta_\beta}{1-\xi_\kappa}\biggl(  
        |\gamma |  \zeta_\beta  
        +
        2
        |\mathcal{P}_{\gamma,c}|  \zeta_\beta \xi_\kappa^{-2} 
        \biggr)  \bigl(1 + (16m)^2\zeta_\beta \xi_\kappa^{-2} \bigr)^{|\mathcal{P}_{\gamma,c}| } 
         %
 \bigl( 1 + (16m)^2 \zeta_\beta\bigr)^{|\gamma |}, 
    \end{align} 
    which shows that~\eqref{eq: second inequality} holds as desired.

     We now show that~\eqref{eq: third inequality} holds. To this end, fix any integers \( i \in \bigl\{ 1, \dots, |\mathcal{P}_{\gamma,c}|  \bigr\} \) and  \( j \in \{2i, \dots, |\gamma | \} .\) If \( \omega \in \mathcal{E}''' \smallsetminus (\mathcal{E}' \cup \mathcal{E}'') \) is such that \( |\support \delta \omega \cap \support \gamma | = j \) and \( |P_{\omega,\gamma,c}|=i, \) then, by~\ref{item: letter lemma iii} of Lemma~\ref{lemma: letter inequalities},  there is \( E \in A_{j,i} \) such that \( \support \delta \omega \cap \support \gamma = E, \) and we have \( \bigl|(\support \delta \omega)^+\bigr| = 4\bigl| (\support \omega)^+ \bigr|  = 4(j-i). \) 
     Applying Lemma~\ref{lemma: flip a set forms iii} with \( E, \) \( k = j-i \geq 1, \) \( k' = 4(j-i), \) and \( k''=j,\) we thus obtain
    \begin{equation*}
        \begin{split}
            &\mathbb{E}_\varphi  \pigl[\widehat{L_\gamma}(\omega)
            \cdot \mathbb{1}(\omega \in \mathcal{E}''' \smallsetminus (\mathcal{E}' \cup \mathcal{E}''),\, |\support \delta \omega \cap \support \gamma | = j,\, |P_{\omega,\gamma,c}|=k)\pigr]
            \\&\qquad\leq 
            \smash{\sum_{E \in A_{j,i}}}
            \mathbb{E}_\varphi  \pigl[ \widehat{L_\gamma}(\omega)
            \cdot \mathbb{1}( \support \delta \omega \cap \support \gamma = E,\, |\support \omega^E| = 2(j-i) \\ &\hspace{13em} |\support \delta \omega^E|= 2\cdot 4(j-i),\, |\support \delta \omega \cap \support \gamma| = j)\pigr] 
            \\&\qquad\leq 
            \sum_{E \in A_{j,i}}
            (16m)^{2(j-i)} \zeta_\beta^{j-i} \xi_\kappa^{|\gamma | + 4(j-i)-2j}
            \leq 
            |A_{j,i}|
            (16m)^{2(j-i)} \zeta_\beta^{j-i} \xi_\kappa^{|\gamma | + 4(j-i)-2j}
            \\&\qquad\leq 
            \binom{|\gamma |}{j-2i} \binom{|\mathcal{P}_{\gamma,c}|}{i} 
            (16m)^{2(j-i)} \zeta_\beta^{j-i} \xi_\kappa^{|\gamma | + 4(j-i)-2j},
        \end{split}
    \end{equation*}  
    where the last inequality follows from applying Lemma~\ref{lemma: Aji bound}.
    Summing over all \( i \in \bigl\{ 1, \dots, |\mathcal{P}_{\gamma,c}| \bigr\}, \) and \( j \in \bigl\{ 2i, \dots,|\gamma | \bigr\}, \) we obtain~\eqref{eq: third inequality} as desired.

    We now show that~\eqref{eq: fourth inequality} holds.
    To this end, fix any \( j \in \{ 0,1, \dots, |\gamma | \}. \)
    If \( \omega \in \mathcal{E}'''' \smallsetminus (\mathcal{E}' \cup \mathcal{E}'' \cup \mathcal{E}''') \) is such that \( |\support \delta \omega \cap \support \gamma | = j ,\) then, by~\ref{item: letter lemma iv} of Lemma~\ref{lemma: letter inequalities}, there is  \( E \in A_{j,0}'' \) and \( e \in E \) such that \(\support \delta \omega \cap \support \gamma = E\smallsetminus \{ e \}, \) and we have \( \hat k \coloneqq \bigl|(\support  \omega^E)^+\bigr| \geq j+2, \) \( \hat k' \coloneqq \bigl|(\support \delta \omega^E)^+\bigr| \geq 4j, \) and \( k'' \coloneqq  |\support \delta \omega^E \cap \support \gamma| = j. \)
    Applying Lemma~\ref{lemma: flip a set forms iii} with \( E, \) \( \hat k \) \( \hat k', \) and \( \hat k'',\) we thus obtain
    \begin{equation*}
        \begin{split}
            &\mathbb{E}_\varphi  \pigl[\widehat{L_\gamma}(\omega)
            \cdot \mathbb{1}\bigl(\omega \in \mathcal{E}'''' \smallsetminus (\mathcal{E}' \cup \mathcal{E}'' \cup \mathcal{E}'''),\, |\support \delta \omega \cap \support \gamma | = j \bigr)\pigr]
            \\&\qquad\leq 
            \smash{\sum_{E \in A_{j,0}''} \sum_{e \in E} \sum_{\hat k = j+2}^\infty \sum_{\hat k' = 4j}^\infty}
            \mathbb{E}_\varphi  \pigl[ \widehat{L_\gamma}(\omega)
            \cdot \mathbb{1}( \support \delta \omega \cap \support \gamma = E\smallsetminus \{ e \},\, |\support \omega^E| = 2\hat k \\ &\hspace{13em} |\support \delta \omega^E|= 2 \hat k',\, |\support \delta \omega \cap \support \gamma| = j)\pigr]  
            \\&\qquad\leq 
            \sum_{E \in A_{j,0}''} \sum_{e \in E} \sum_{\hat k = j+2}^\infty \sum_{\hat k' = 4j}^\infty
            (16m)^{2\hat k} \zeta_\beta^{\hat k} \xi_\kappa^{|\gamma | + \hat k'-2j} 
            \leq 
            |A_{j,0}''|\cdot |E| \sum_{\hat k = j+2}^\infty \sum_{\hat k' = 4j}^\infty
            (16m)^{2\hat k} \zeta_\beta^{\hat k} \xi_\kappa^{|\gamma | + \hat k'-2j}
            \\&\qquad\leq 
            \binom{|\gamma |}{j+1}  (j+1) \sum_{\hat k = j+2}^\infty \sum_{\hat k' = 4j + \max(0,j+6-\hat k)}^\infty
            (16m)^{2\hat k} \zeta_\beta^{\hat k} \xi_\kappa^{|\gamma | + \hat k'-2j},
        \end{split}
    \end{equation*}   
    where the last inequality follows from applying Lemma~\ref{lemma: Aji'' bound}.
    Summing over all \( j \in \bigl\{ 0, \dots,|\gamma | \bigr\} \) (see Claim~\ref{sublemma: upper bound on expectation of two events 3} in the Appendix), we obtain
    \begin{align}
        &B_3 \coloneqq \sum_{j = 0}^{|\gamma |}
        \binom{|\gamma |}{j+1}  (j+1) \sum_{\hat k = j+2}^\infty \sum_{\hat k' = 4j + \max(0,j+6-\hat k)}^\infty \!\!\!\! (16m)^{2\hat k} \zeta_\beta^{\hat k} \xi_\kappa^{|\gamma | + \hat k'-2j}\label{eq: B3}
        \\&\qquad\leq\nonumber
        \frac{\xi_\kappa^{|\gamma |}(16m)^{4}\zeta_\beta^2  |\gamma |  \bigl( 1 + (16m)^{2} \zeta_\beta  \xi_\kappa^{2}  \bigr)^{|\gamma |}}{1-\xi_\kappa}
            \biggl(
        \frac{\xi_\kappa^{4}}{(1-(16m)^{2} \zeta_\beta \xi_\kappa^{-1})}
        +
        \frac{(16m)^{8} \zeta_\beta^{4} }{(1-(16m)^{2} \zeta_\beta )} 
        \biggr) 
    \end{align}
    and hence~\eqref{eq: fourth inequality} holds as desired. This concludes the proof.
\end{proof}

\begin{proof}[Proof of Proposition~\ref{proposition: useful upper bound forms ii}]
    If \( \omega \in \Omega^2(B_N,\mathbb{Z}_n)\) is such that \( \bigl| (\support \omega^{\gamma,2})^+ \bigr| = \| \omega^{\gamma} \|, \) then we must necessarily have \( \omega^{\gamma,2} = \omega^\gamma.\) Consequently, if  \( |P_{\omega,\gamma,c}| = 0,\) then \( |\support \delta \omega \cap \support \gamma| = \| \omega^\gamma \|.\) Recalling the definition of  \( \mathcal{E}\) from~\eqref{eq: def E2 forms ii}, we can thus write 
    \begin{equation*}\label{eq: def E2 forms ii0}
        \begin{split} 
            \mathcal{E} = \bigl\{ \omega \in \Omega^2(B_N,\mathbb{Z}_n) \colon
            \bigl|(\support \omega^\gamma)^+\bigr| = |\support \delta \omega \cap \support \gamma| - |P_{\omega,\gamma,c}|  = \| \omega^\gamma\|,\quad& 
            \\ |P_{\omega,\gamma,c}| = 0,\, \omega^{\gamma,2} = \omega^\gamma  &\bigr\}.
        \end{split}
    \end{equation*}

    For all \( \omega \in \Omega^2(B_N,\mathbb{Z}_n), \) we have 
    \begin{equation*}
        |\support \delta \omega \cap \support \gamma| - |P_{\omega,\gamma,c}|\geq \|\omega^\gamma\|.
    \end{equation*}
    Furthermore, it follows from Lemma~\ref{lemma: basic inequality 1 forms} that 
    \begin{equation*}
       \bigl|(\support \omega^\gamma)^+ \bigr|  \geq  \bigl|\support \delta \omega\cap \support \gamma \bigr| - |P_{\omega,\gamma,c}|.
    \end{equation*}
    %
    If \( \omega \in \Omega^2(B_N,\mathbb{Z}_n) \smallsetminus \mathcal{E}, \) then either \(  |P_{\omega,\gamma,c}| > 0, \) 
    or \(  |P_{\omega,\gamma,c}| = 0 \) and either \( \bigl| (\support \omega^\gamma)^+ \bigr| > |\support \delta \omega \cap \support \gamma|,  \)  \( |\support \delta \omega  \cap \support \gamma| > \| \omega^\gamma\| , \) or \( \omega^{\gamma,2} \neq \omega^\gamma. \) 
    In other words, if \( \omega \in \Omega^2(B_N,\mathbb{Z}_n) \smallsetminus \mathcal{E}, \) then at least one of the following hold. 
    \begin{enumerate}[label=(\roman*)]
        \item \( \bigl|(\support \omega^\gamma)^+ \bigr| \geq |\support \delta \omega \cap \support \gamma|  - |P_{\omega,\gamma,c}| +1 \) (implying that \( \omega \in \mathcal{E}'\)),
        
        \item \( |\support \delta \omega \cap \support \gamma| - |P_{\omega,\gamma,c}|  \geq \| \omega^\gamma\| +1 \) (implying that \( \omega \in \mathcal{E}''\)),
        
        \item \(|P_{\omega,\gamma,c}| \geq 1 \) (implying that \( \omega \in \mathcal{E}'''\)), or
        
        \item \(\omega^{\gamma,2} \neq \omega^\gamma \) (implying that \( \omega \in \mathcal{E}''''\)).
    \end{enumerate}
    Hence \( \Omega^2(B_N,\mathbb{Z}_n) \smallsetminus \mathcal{E} = \mathcal{E}' \cup \mathcal{E}'' \cup \mathcal{E}''' \cup \mathcal{E}''''. \) Using Lemma~\ref{lemma: upper bound on expectation of two events} and a union bound, the desired conclusion follows.
\end{proof}

Before we end this section, we prove the following lemma, which will be used later.

\begin{lemma}\label{lemma: upper bound for E2 forms}
    Let \( \beta,\kappa \geq 0, \) and let \( \gamma \) be a rectangular path. Assume that~\eqref{assumption: 1} holds. Then
    \begin{equation*} 
        \mathbb{P}_\varphi (\omega \notin \mathcal{E})  
        \leq  
        C^{(2ii)}_{\gamma,\beta,\kappa,m}  \zeta_\beta,
    \end{equation*}
    where \( \mathcal{P}_\gamma \) is defined in~\eqref{eq: Pgamma} and \( \mathcal{P}_{\gamma,c} \) is defined in~\eqref{eq: Pgammac}, and
    \begin{equation}\label{eq: def C2ii ii}
        C^{(2ii)}_{\gamma,\beta,\kappa,m} \coloneqq 
        \frac{|\mathcal{P}_{\gamma,c}| (8m)^{2}  \xi_\kappa^{4} }{1-(8m)^{2} \zeta_\beta \xi_\kappa^{-1} }
        +
        \frac{|\mathcal{P}_\gamma|(8m)^{4} \zeta_\beta \xi_\kappa^{4} }{1-(8m)^{2} \zeta_\beta \xi_\kappa^{-1} }.
    \end{equation}
\end{lemma}

\begin{proof}
We saw in the proof of Proposition~\ref{proposition: useful upper bound forms ii} that if \( \omega \notin \mathcal{E}, \) then either \(  |P_{\omega,\gamma,c}| > 0, \) or \(  |P_{\omega,\gamma,c}| = 0 \) and either \( \bigl| (\support \omega^\gamma)^+ \bigr| > |\support \delta \omega \cap \support \gamma|,  \)  \( |\support \delta \omega  \cap \support \gamma| > \| \omega^\gamma\| , \) or \( \omega^{\gamma,2} \neq \omega^\gamma. \) 
    Consequently, if \( \omega \notin \mathcal{E}, \) then either \( |P_{\omega,\gamma,c}| > 0, \) or some plaquette in \( \mathcal{P}_\gamma^+ \) is in the support of some \( \omega' \lhd \omega \) with connected support that satisfies \( \bigl|(\support \omega')^+\bigr| \geq 2. \)  
    Applying Lemma~\ref{lemma: flip a set forms} twice, first, for each \( p \in \mathcal{P}_{\gamma,c}, \) with \( P_0 = \{ p \},\) \( k = 1, \) and \( k' = 5, \) and then for each \( p \in \mathcal{P}_\gamma^+, \) with \( P_0 = \{ p \},\) \( k = 2 \) and \( k' = 6, \) and using a union bound, we obtain
    \begin{equation*}
        \begin{split}
             &\mathbb{P}_\varphi (\omega \notin \mathcal{E})
             \leq
             \sum_{p \in \mathcal{P}_{\gamma,c}} \mathbb{P}_\varphi(p \in \support \omega) + \sum_{p \in \mathcal{P}_\gamma}\pigl(\bigl|( \support \omega^{\support \partial p})^+\bigr| \geq 2 \pigr)
             \\&\qquad\leq 
             \frac{|\mathcal{P}_{\gamma,c}| (8m)^{2} \zeta_\beta \xi_\kappa^{4} }{1-(8m)^{2} \zeta_\beta \xi_\kappa^{-1} }
             +
              \frac{|\mathcal{P}_\gamma|(8m)^{4} \zeta_\beta^2 \xi_\kappa^{4} }{1-(8m)^{2} \zeta_\beta \xi_\kappa^{-1} }.
        \end{split} 
    \end{equation*}
    This concludes the proof.
\end{proof}

\section{Poisson approximation} 
\label{section: poisson}

With Proposition~\ref{proposition: useful upper bound forms ii} at hand, we know that the contribution to the Wilson line expectation from 2-forms \( \omega \in \Omega^2(B_N,\mathbb{Z}_n) \) which do not satisfy the conditions \( |\support \delta \omega \cap \support \gamma | = \bigl|(\support \omega^\gamma)^+ \bigr| = \|\omega^\gamma \|,\)  \(  |P_{\omega,\gamma,c}| =0, \) and \( \omega^{\gamma,2} = \omega^\gamma\) is very small. For this reason, in this section, we concentrate on calculating the contribution from configurations that do satisfy the above conditions.
The main result in this section is the following proposition. 

\begin{proposition}\label{proposition: last resampling lemma forms ii}
    Let \( \beta,\kappa \geq 0, \) and let \( \gamma \) be a rectangular path. Assume that~\eqref{assumption: 1} and~\eqref{assumption: 3} hold.
    Then
    \begin{equation*}
    \begin{split}
        &\Bigl| \mathbb{E}_\varphi    \pigl[\widehat{L_\gamma}(\omega)\cdot \mathbb{1}(  \mathcal{E} ) \pigr] - \xi_\kappa^{|\gamma |} \alpha(\beta,\kappa)^{|\mathcal{P}_\gamma|}\Bigr|
            \leq
            C^{(2)}_{\gamma,\beta,\kappa,m} \, \varphi_\kappa(1)^{|\gamma |}
            \alpha(\beta,\kappa)^{|\mathcal{P}_\gamma |} \zeta_\beta,
    \end{split}
\end{equation*}
where  \( \mathcal{P}_\gamma \) is defined in~\eqref{eq: Pgamma}, 
\begin{equation}
    \begin{split}
        &C^{(2)}_{\gamma,\beta,\kappa,m} \coloneqq 
            C^{(2i)}_{\gamma,\beta,\kappa,m}
            +  
            C^{(2ii)}_{\gamma,\beta,\kappa, m}
            + 
            C^{(2iii)}_{\gamma,\beta,\kappa, m},
        \end{split}
    \end{equation}
    \( C^{(2i)}_{\gamma,\beta,\kappa,m} \) is defined in~\eqref{eq: def C2i}, \( C^{(2ii)}_{\gamma,\beta,\kappa,m} \) is defined in~\eqref{eq: def C2ii ii}, and \( C^{(2iii)}_{\gamma,\beta,\kappa,m} \) is defined in~\eqref{eq: def C2iii ii}.
\end{proposition}

The next few lemmas will be useful in the proof of Proposition~\ref{proposition: last resampling lemma forms ii}. 
For \( j\in \mathbb{Z}_n, \) define
\begin{equation*}
    r_\kappa(j) \coloneqq \frac{\varphi_\kappa(  j +1)}{ \varphi_\kappa(j)\varphi_\kappa(1)}.
\end{equation*}
The reason that the function \( r_\kappa  \) is relevant for us is the following lemma.

\begin{lemma}\label{lemma: change of complicated term forms ii}
    Let \( \beta,\kappa \geq 0, \) and let \( \gamma \) be a rectangular path. Then
    \begin{equation*}
        \begin{split}
            &\mathbb{E}_\varphi    \pigl[\widehat{L_\gamma}(\omega) \cdot \mathbb{1}(  \mathcal{E} ) \pigr] 
            =
            \varphi_\kappa(1)^{|\gamma |}
            \mathbb{E}_\varphi    \Bigl[\prod_{p  \in \mathcal{P}_\gamma} r_\kappa \bigl( \omega(p) \bigr)\cdot \mathbb{1}(  \mathcal{E} ) \Bigr] .
        \end{split}
    \end{equation*}
\end{lemma}

\begin{proof}
    For any \( \omega \in \Omega^2(B_N,\mathbb{Z}_n), \) we have
    \begin{align*}
        & \widehat{L_\gamma}(\omega) = 
        \prod_{e \in \support \gamma} 
            \varphi_\kappa\bigl( \delta\omega(e) + \gamma[e] \bigr) \varphi_\kappa\bigl( \delta\omega  (e) \bigr)^{-1}
            =
            \prod_{e \in  \gamma} \varphi_\kappa\bigl( \delta \omega(e) + \gamma[e]\bigr) \varphi_\kappa\bigl( \delta \omega(e)\bigr)^{-1}
        \\&\qquad=
        \varphi_\kappa(1)^{|\gamma |}\prod_{e \in   \gamma} \frac{\varphi_\kappa\bigl(\delta \omega (e)+ \gamma[e]\bigr)}{ \varphi_\kappa\bigl( \delta \omega(e)\bigr)\varphi_\kappa(1)}
        =
        \varphi_\kappa(1)^{|\gamma |}\prod_{e \in   \gamma} r_\kappa \bigl( \delta \omega(e) \bigr).
    \end{align*}
    Here the second equality uses that, by Lemma~\ref{lemma: symmetry}, the function \( \varphi_\kappa \) is even.

    Fix \( \omega \in \Omega^2(B_N,\mathbb{Z}_n). \) If \(\delta \omega(e) = 0,\) then \( r_\kappa \bigl( \delta\omega(e) \bigr) = r_\kappa(0) = 1. \)  
    Using this observation, it follows that 
    \begin{equation*}
        \prod_{e \in   \gamma} r_\kappa \bigl( \delta \omega(e) \bigr)
        =
        \prod_{e \in \gamma \colon \delta \omega(e) \neq 0} r_\kappa \bigl( \delta \omega(e) \bigr).
    \end{equation*}
    Now note that if \( \omega \in \mathcal{E}, \)
    then \( |\support \delta \omega \cap \support \gamma| = \bigl| (\support \omega^\gamma)^+ \bigr| = \| \omega^\gamma \|,\) and \( |P_{\omega,\gamma,c}| = 0. \) Consequently, for any \( e \in \gamma \) such that \( e \in \support \delta \omega  \), we have \( \bigl|\support \hat \partial e \cap \support \omega\bigr| = 1,\) and hence there is a unique \( p \in   \hat \partial e \subseteq \mathcal{P}_\gamma \) such that  \( \delta \omega(e) = \omega(p) \) and \(  \omega(p') = 0 \) for all \( p' \in  \hat \partial e \smallsetminus \{ p \}. \) Moreover, if $\omega \in \mathcal{E}$, then \( \omega^{\gamma,2} = \omega^\gamma\), and hence $\support \hat{\partial} e \cap \support \omega = \emptyset$ whenever $e \in \gamma$ and $\delta \omega(e) = 0$.
    Since \( r_\kappa(0) = 1,\) we obtain, for all \( \omega \in \mathcal{E}, \)
    \begin{equation*}
        \begin{split}
            &\prod_{e \in \gamma \colon  \delta \omega (e) \neq 0}  r_\kappa \bigl( \delta \omega(e) \bigr) \bigr)
            =
            \prod_{p  \in \mathcal{P}_\gamma}  r_\kappa \bigl(  \omega(p) \bigr).
        \end{split}
    \end{equation*}
    From this the desired conclusion immediately follows.
\end{proof}

Recall the definitions of \( \mathcal{P}_\gamma \) and \( \mathcal{P}_{\gamma,c} \) from~\eqref{eq: Pgamma} and~\eqref{eq: Pgammac} respectively. Further, recall the definition of \( \nsim \) from Section~\ref{connectedsubsec}. Given \( \omega \in \Omega^2(B_N,\mathbb{Z}_n), \) we let
\begin{equation*}
    P_\omega \coloneqq     \bigl\{ p \in \mathcal{P}_\gamma \smallsetminus \pm \mathcal{P}_{\gamma,c} \colon p^+ \nsim (\support \omega)^+\smallsetminus (\mathcal{P}_\gamma^+
    \smallsetminus \pm \mathcal{P}_{\gamma,c} ) \bigr\} .
\end{equation*}
For any \( \omega \in \Omega^2(B_N,\mathbb{Z}_n), \) we can identify the set \( P_\omega \) by looking only at \( (\support \omega)^+\smallsetminus (\mathcal{P}_\gamma^+ \smallsetminus \pm \mathcal{P}_{\gamma,c}). \)
For \( j \in \mathbb{Z}_n, \) define 
\begin{equation*}
    \hat \lambda_j \coloneqq  \varphi_\beta(j) \varphi_\kappa(j)^4 \quad\text{and}\quad \lambda_j \coloneqq \hat\lambda_j/\sum_{j \in \mathbb{Z}_n} \hat \lambda_j. 
\end{equation*}
With this notation, we have
\begin{equation}\label{alphasumf} 
    \alpha(\beta,\kappa) = \sum_{j \in \mathbb{Z}_n} \lambda_j r_\kappa(j)
    \quad \text{and} \quad 
    \lambda_0 = \frac{1}{\sum_{j \in \mathbb{Z}_n} \varphi_\beta(j) \varphi_\kappa(j)^4}.
\end{equation}

\begin{lemma}\label{lemma: lambda inequality}
We have \( 0 \leq 1 - \lambda_0 \leq \zeta_\beta \xi_\kappa^4. \)
\end{lemma}

\begin{proof}
    Since \( \varphi_\beta(0) = \varphi_\kappa(0) = 1,\) we have
    \begin{equation*}
        1 - \lambda_0 = 1 - \frac{1}{\sum_{j \in \mathbb{Z}_n} \varphi_\beta(j) \varphi_\kappa(j)^4} 
        = \frac{\sum_{j \in \mathbb{Z}_n \smallsetminus \{ 0 \}} \varphi_\beta(j) \varphi_\kappa(j)^4}{1 + \sum_{j \in \mathbb{Z}_n \smallsetminus \{ 0 \}} \varphi_\beta(j) \varphi_\kappa(j)^4}.
    \end{equation*}
    Since \( \varphi_\beta(j)\varphi_\kappa(j) \geq 0\) for all \( j \in \mathbb{Z}_n,\) we immediately obtain \( 1 -\lambda_0 \geq 0.\) To obtain an upper bound, note also that, by definition, we have \( \sum_{j \in \mathbb{Z}_n \smallsetminus \{ 0 \}} \varphi_\beta(j) \varphi_\kappa(j)^4 \leq \zeta_\beta \xi_\kappa^4.\) Hence
    \begin{equation*}
        \frac{\sum_{j \in \mathbb{Z}_n \smallsetminus \{ 0 \}} \varphi_\beta(j) \varphi_\kappa(j)^4}{1 + \sum_{j \in \mathbb{Z}_n \smallsetminus \{ 0 \}} \varphi_\beta(j) \varphi_\kappa(j)^4}
        \leq 
        \sum_{j \in \mathbb{Z}_n \smallsetminus \{ 0 \}} \varphi_\beta(j) \varphi_\kappa(j)^4 \leq \zeta_\beta \xi_\kappa^4,
    \end{equation*}
    from which the desired conclusion  follows.
\end{proof}

Let \( \mu_{\omega, \lambda} \) be the probability measure on \( \Omega^2(B_N,\mathbb{Z}_n)\) defined by
\begin{equation*}
    \mu_{\omega,\lambda}(\omega')
    =
    \begin{cases}
        0 &\text{if } (\support \omega')^+ \nsubseteq P_\omega^+   \cr 
        \prod_{p \in P_\omega} \lambda_{\omega'(p)} &\text{else.}
    \end{cases} 
\end{equation*}
Let \( \mathbb{E}_{\omega,\lambda}\) be the corresponding expectation.  
For \( \omega \in \Omega^2(B_n,\mathbb{Z}_n),\) let 
\begin{equation*}
    \begin{split}
        &\mathcal{E}_1^\omega \coloneqq \bigl\{ \omega' \in \Omega^2(B_N,\mathbb{Z}_n) \colon (\support \omega')^+ \subseteq P_\omega  \text{ and } \omega' \in \mathcal{E} \bigr\}
        \\&\qquad=
        \bigl\{ \omega' \in \Omega^2(B_N,\mathbb{Z}_n) \colon (\support \omega')^+ \subseteq P_\omega  \text{ and } |(\support \omega')^+| = \| \omega' \| \bigr\}.
    \end{split}
\end{equation*}

\begin{lemma}\label{lemma: last lemma}
    Let \( \omega  \in \mathcal{E}.\)
    Then the following hold.
    \begin{enumerate}[label=\upshape (\arabic*)]
        \item For all \( p_1 \in \support \omega|_{P_\omega}\) and \( p_2 \in \support \omega|_{C_2(B_N)^+ \smallsetminus P_\omega^+}\) we have \( \support \partial p_1 \cap \support \partial p_2 = \emptyset.\) \label{item: last lemma 1}
        \item \label{item: last lemma 2} There is a bijection between the set of all \( \omega' \in \mathcal{E}\) such that \( P_{\omega'} = P_\omega\) and the set of pairs \( (\omega'',\omega''') \in \mathcal{E} \times \mathcal{E}\) 
        with
        \begin{enumerate}[label=\upshape (\roman*)]
            \item \( (\support \omega'')^+ \subseteq P_\omega^+, \) \label{item: last lemma 2i}
            \item \( P_{\omega'''} = P_{\omega}, \) and \label{item: last lemma 2ii}
            \item \( (\support \omega''')^+ \cap P_\omega^+ = \emptyset.\) \label{item: last lemma 2iii}
        \end{enumerate}
    \end{enumerate}
\end{lemma}

\begin{proof}
   ~\ref{item: last lemma 1} Fix any \( p_1 \in \support \omega|_{P_\omega}\) and \( p_2 \in \support \omega|_{C_2(B_N)^+ \smallsetminus P_\omega}. \)
    We can without loss of generality assume that \( p_1 \in P_\omega.\)
    Since \( p_2 \notin \pm P_\omega \) and \( p_2 \in \support \omega,\) we must have either \( p_2^+ \in (\support \omega)^+ \smallsetminus (\mathcal{P}_\gamma^+\smallsetminus \pm \mathcal{P}_{\gamma,c}) \) or \( p_2^+ \sim (\support \omega)^+ \smallsetminus (\mathcal{P}_\gamma^+\smallsetminus \pm \mathcal{P}_{\gamma,c}).\) In any of these cases, since \( p_1^+ \nsim (\support \omega)^+ \smallsetminus (\mathcal{P}_\gamma^+\smallsetminus \pm \mathcal{P}_{\gamma,c}),\) we cannot have \( p_1^+ \sim p_2^+,\) and hence we must have \( \support \partial p_1 \cap \support \partial p_2 = \emptyset.\) 

  ~\ref{item: last lemma 2} 
    %
    We first construct a map $F:\omega' \mapsto (\omega'', \omega''')$ from $X := \{\omega' \in \mathcal{E} : P_{\omega'} = P_\omega\}$ to the set $Y$ of pairs $(\omega'', \omega''') \in \mathcal{E} \times \mathcal{E}$ that satisfy~\ref{item: last lemma 2i}--\ref{item: last lemma 2iii}.
    Fix any \( \omega' \in \mathcal{E}\) with \(P_{\omega'} =P_\omega.\) Define \( \omega'' \coloneqq \omega'|_{P_\omega^+}\) and \( \omega ''' \coloneqq \omega' - \omega''.\) Then, by construction, 
   ~\ref{item: last lemma 2i}~and~\ref{item: last lemma 2iii} hold.
    Moreover, we have
    \begin{align*}
        &P_{\omega'''} = \bigl\{ p \in \mathcal{P}_\gamma \smallsetminus \pm \mathcal{P}_{\gamma,c} \colon p^+ \nsim (\support \omega''')^+\smallsetminus (\mathcal{P}_\gamma^+ \smallsetminus \pm \mathcal{P}_{\gamma,c} ) \bigr\} 
        \\&\qquad= 
        \bigl\{ p \in \mathcal{P}_\gamma \smallsetminus \pm \mathcal{P}_{\gamma,c} \colon p^+ \nsim (\support \omega'|_{C_2(B_N)^+ \smallsetminus P_\omega^+})^+ \smallsetminus (\mathcal{P}_\gamma^+ \smallsetminus \pm \mathcal{P}_{\gamma,c} ) \bigr\} 
        \\&\qquad= 
        \bigl\{ p \in \mathcal{P}_\gamma \smallsetminus \pm \mathcal{P}_{\gamma,c} \colon p^+ \nsim (\support \omega')^+\smallsetminus (\mathcal{P}_\gamma^+ \smallsetminus \pm \mathcal{P}_{\gamma,c} ) \bigr\} 
        = 
        P_{\omega'} = P_\omega.
    \end{align*}
    Hence~\ref{item: last lemma 2ii} holds.
    We now show that \( \omega'',\omega''' \in \mathcal{E}. \) To this end, note first that, by definition, the sets \( (\support \omega'')^+\) and \( (\support \omega''')^+\) are disjoint, and each is a union of connected components of \( (\support \omega')^+.\) In other words, we can write \( (\support \omega'')^+ = P_1 \sqcup \dots \sqcup P_j \) and \( (\support \omega''')^+ = P_{j+1} \sqcup \dots \sqcup P_m,\) where \( P_1, \dots, P_m\) are the connected components of \( (\support \omega')^+.\)
    Since \( \omega' \in \mathcal{E}\) by assumption, we have \( |P_{\omega',\gamma,c}| = 0. \) Using Lemma~\ref{lemma: split into components} it thus follows that \( |P_{\omega'|_{P_i},\gamma,c}| = 0 \) for all \( i \in \{ 1,2, \dots, m \}, \) and hence, again using Lemma~\ref{lemma: split into components}, we obtain 
    \begin{equation*}
        |P_{\omega'',\gamma,c}| = \sum_{i=1}^j |P_{\omega'|_{P_i},\gamma,c}| = 0 .
    \end{equation*}
    Next we note that since \( \omega' \in \mathcal{E},\) 
    all connected components of \( (\support ( \omega')^{\gamma,2})^+\) have size one. Since \( \omega''\) is a restriction of \( \omega' \) to a union of connected components of \( (\support \omega')^+,\) it follows that all connected components of \( (\support ( \omega'')^{\gamma,2})^+\) have size one. From this, it immediately follows that \( |(\support  (\omega'')^\gamma)^+\| = |\support \delta \omega''  \cap \support \gamma |,\) \( ( \omega'')^{\gamma,2} = (\omega'')^\gamma, \) and \( \| (\omega'')^\gamma\|= |(\support (\omega'')^\gamma)^+|.\)  Since \( |P_{\omega'',\gamma,c}| =0,\) this implies that \( \omega'' \in \mathcal{E}.\) Completely analogously, we also obtain \( \omega''' \in \mathcal{E}.\)
    This shows that $F$ is a well-defined map from $X$ into $Y$. 
    
    We will show that $F:X \to Y$ is a bijection by constructing its inverse $F^{-1}:Y \to X$ explicitly. 
Fix any \( \omega'', \omega''' \in \mathcal{E}\) which satisfy~\ref{item: last lemma 2i}--\ref{item: last lemma 2iii}. 
    Define \( \omega' \coloneqq \omega'' + \omega'''.\) Since \( (\support \omega'')^+ \subseteq P_\omega^+ \subseteq \mathcal{P}_\gamma^+ \smallsetminus \pm \mathcal{P}_{\gamma,c},\) we have
    \begin{align*}
        &P_{\omega'} = 
        \bigl\{ p \in \mathcal{P}_\gamma \smallsetminus \pm \mathcal{P}_{\gamma,c} \colon p^+ \nsim (\support \omega')^+\smallsetminus (\mathcal{P}_\gamma^+ \smallsetminus \pm \mathcal{P}_{\gamma,c} ) \bigr\} 
        \\&\qquad=
        \bigl\{ p \in \mathcal{P}_\gamma \smallsetminus \pm \mathcal{P}_{\gamma,c} \colon p^+ \nsim (\support (\omega'-\omega''))^+\smallsetminus (\mathcal{P}_\gamma^+ \smallsetminus \pm \mathcal{P}_{\gamma,c} ) \bigr\}  
        \\&\qquad=
        \bigl\{ p \in \mathcal{P}_\gamma \smallsetminus \pm \mathcal{P}_{\gamma,c} \colon p^+ \nsim \support (\omega''')^+\smallsetminus (\mathcal{P}_\gamma^+ \smallsetminus \pm \mathcal{P}_{\gamma,c} ) \bigr\} 
        =
        P_{\omega'''} = P_\omega.
    \end{align*}
    Thus, it only remains to show that \( \omega' \in \mathcal{E}.\) To this end, note that 
    \begin{itemize}
        \item \( (\support \omega'')^+\) and \( (\support \omega''')^+\) are disjoint,
        \item \( (\support \omega'')^+\) and \( (\support \omega''')^+\) are both unions of connected components of \( (\support \omega')^+,\) and
        \item \( \support \omega'' \cup  \support \omega''' = \support \omega'.\) 
    \end{itemize}
    Since \( (\omega')^\gamma\) and \( (\omega')^{\gamma,2}\) are also unions of connected components of \( \omega',\) using Lemma~\ref{lemma: split into components} and the assumption that \( \omega'',\omega''' \in \mathcal{E},\)  it follows that  \( \omega' \in \mathcal{E}.\)
    This concludes the proof.
\end{proof}

\begin{lemma}\label{lemma: P tilde equation forms ii}
    Let \( \beta,\kappa \geq 0 \) and let \( \gamma \) be a rectangular path. Then
    \begin{equation}\label{eq: P tilde equation forms ii}
        \begin{split}
            &
            \mathbb{E}_\varphi  \Bigl[ \, \prod_{p  \in \mathcal{P}_\gamma} r_\kappa \bigl(  \omega(p) \bigr) \cdot \mathbb{1}(\omega \in \mathcal{E})\Bigr] 
            =
            \mathbb{E}_\varphi  \Bigl[ \mathbb{E}_{\omega, \lambda} \pigl[\prod_{p  \in P_\omega} r_\kappa \bigl(  \omega'(p) \bigr)  \mid \omega' \in \mathcal{E}_1^\omega \pigr]\cdot \mathbb{1}(\omega \in \mathcal{E})\Bigr].
        \end{split}
    \end{equation}
\end{lemma}

\begin{proof} 
    If \( p \in \mathcal{P}_\gamma \smallsetminus P_\omega\) and \( \omega \in \mathcal{E},\) then \( \omega(p) = 0,\) and hence
    \begin{equation}\label{eq: P tilde equation forms ii 1}
        \begin{split}
            &\mathbb{E}_\varphi  \Bigl[ \prod_{p  \in \mathcal{P}_\gamma} r_\kappa \bigl(  \omega(p) \bigr) \cdot \mathbb{1}(\omega \in \mathcal{E})\Bigr] 
            =
            \mathbb{E}_\varphi  \Bigl[ \prod_{p  \in P_\omega} r_\kappa \bigl(  \omega(p) \bigr) \cdot \mathbb{1}(\omega \in \mathcal{E})\Bigr]
            \\&\qquad =
            \mathbb{E}_\varphi  \Bigl[ \mathbb{E}_\varphi  \pigl[ \prod_{p  \in P_\omega} r_\kappa \bigl(  \omega(p) \bigr) \mid P_\omega,\, \omega \in \mathcal{E} \pigr]\cdot \mathbb{1}(\omega \in \mathcal{E})\Bigr].
        \end{split}
    \end{equation} 
    Now note that given \( \omega \in \mathcal{E} \), by definition, we have
    \begin{equation*}
        \begin{split}
            &\mathbb{E}_\varphi  \Bigl[ \prod_{p  \in P_{\omega'}} r_\kappa \bigl(  \omega'(p) \bigr)\mid P_{\omega'} = P_\omega,\, \omega' \in \mathcal{E} \Bigr]
            =
            \frac{\sum_{\omega'\in \mathcal{E} \colon  P_{\omega'} = P_\omega} \varphi(\omega') \prod_{p  \in P_{\omega'}} r_\kappa \bigl(\omega'(p)\bigr)  }{\sum_{\omega'\in \mathcal{E} \colon  P_{\omega'} = P_\omega} \varphi(\omega') }.
        \end{split}
    \end{equation*}
    By definition, if \( \omega' \in \mathcal{E}\) is such that \( P_{\omega'} = P_\omega,\) then there are no plaquettes \( p \in \support \omega'|_{P_\omega} \) and \( p' \in \support \omega'|_{C_2(B_N)^+ \smallsetminus P_\omega^+} \) such that \( p \sim p',\)  and hence

    \begin{equation*}
        \varphi(\omega') = \varphi(\omega'|_{P_\omega}) \varphi(\omega'|_{C_2(B_N)^+ \smallsetminus P_\omega^+}).
    \end{equation*}
    In particular, using Lemma~\ref{lemma: last lemma}, it follows that 
    \begin{equation*}
        \begin{split}
            &\frac{\sum_{\omega'\in \mathcal{E} \colon  P_{\omega'} = P_\omega} \varphi(\omega') \prod_{p  \in P_{\omega'}} r_\kappa \bigl(\omega'(p)\bigr)  }{\sum_{\omega'\in \mathcal{E} \colon  P_{\omega'} = P_\omega} \varphi(\omega') }
            \\&\qquad =
            \frac{\sum_{\omega'\in \mathcal{E} \colon  P_{\omega'} = P_\omega} \varphi(\omega'|_{P_\omega})\varphi(\omega'|_{C_2(B_N)^+ \smallsetminus P_\omega^+})
            \prod_{p  \in P_{\omega'}} r_\kappa \bigl(\omega'(p)\bigr)  }{\sum_{\omega'\in \mathcal{E} \colon  P_{\omega'} = P_\omega} \varphi(\omega'|_{P_\omega})\varphi(\omega'|_{C_2(B_N)^+ \smallsetminus P_\omega^+}) }
            \\&\qquad =
            \frac{\Big(\sum_{\omega'\in \mathcal{E} \colon  (\support \omega')^+ \subseteq  P_\omega^+} \varphi(\omega' )
            \prod_{p  \in P_{\omega'}} r_\kappa \bigl(\omega'(p)\bigr)\Big) \Big( 
            \sum_{\omega'' \in \mathcal{E} \colon P_{\omega''} = P_\omega,\, \omega''|_{P_\omega}=0} \varphi(\omega'')\Big)
            }{\Big(\sum_{\omega'\in \mathcal{E} \colon  (\support \omega)^+ \subseteq  P_\omega^+} \varphi(\omega' ) \Big)\Big(
            \sum_{\omega'' \in \mathcal{E} \colon P_{\omega''} = P_\omega,\, \omega''|_{P_\omega}=0} \varphi(\omega'')\Big)
            }
            \\&\qquad =
            \frac{\sum_{\omega'\in \mathcal{E} \colon  (\support \omega')^+ \subseteq  P_\omega^+} \varphi(\omega' ) \prod_{p  \in P_{\omega'}}r_\kappa  \bigl(\omega'(p)\bigr) 
            }{\sum_{\omega'\in \mathcal{E} \colon  (\support \omega)^+ \subseteq  P_\omega^+} \varphi(\omega' ) 
            }.
        \end{split}
    \end{equation*}
    If \( \omega' \in \mathcal{E} \) and \( (\support \omega')^+ \subseteq P_\omega^+,\) the connected components of \( (\support \omega')^+ \) must all consist of exactly one plaquette, 
    and hence \( \varphi(\omega') = \prod_{p \in C_2(B_N)^+} \hat \lambda_{\omega'(p)}.\) This implies in particular that the previous expression is equal to
    \begin{equation*}
        \begin{split}
            &\frac{\sum_{\omega'\in \mathcal{E} \colon  (\support \omega')^+ \subseteq  P_\omega^+} \Big(\prod_{p  \in P_{\omega'}} r_\kappa \bigl(\omega'(p)\bigr) \Big) \prod_{p \in C_2(B_N)^+} \hat \lambda_{\omega'(p)}}{\sum_{\omega'\in \mathcal{E} \colon  (\support \omega)^+ \subseteq  P_\omega^+} \prod_{p \in C_2(B_N)^+} \hat \lambda_{\omega'(p)}  }
            \\&\qquad =
            \frac{\sum_{\omega'\in \mathcal{E} \colon  (\support \omega')^+ \subseteq  P_\omega^+} \Big(\prod_{p  \in P_{\omega'}} r_\kappa \bigl(\omega'(p)\bigr) \Big) \prod_{p \in P_\omega} \lambda_{\omega'(p)} }{\sum_{\omega'\in \mathcal{E} \colon  (\support \omega)^+ \subseteq  P_\omega^+} \prod_{p \in P_\omega} \lambda_{\omega'(p)}  }
            \\&\qquad =
            \frac{\sum_{\omega'\in \mathcal{E} \colon  (\support \omega')^+ \subseteq  P_\omega^+} \Big(\prod_{p  \in P_{\omega'}} r_\kappa \bigl(\omega'(p)\bigr) \Big) \mu_{\omega,\lambda}(\omega') }{\sum_{\omega'\in \mathcal{E} \colon  (\support \omega)^+ \subseteq  P_\omega^+} \mu_{\omega,\lambda}(\omega')  }
            \\&\qquad =
            \frac{\sum_{\omega'\in \mathcal{E}_1^\omega } \Big(\prod_{p  \in P_{\omega'}} r_\kappa \bigl(\omega'(p)\bigr) \Big) \mu_{\omega,\lambda}(\omega') }{\sum_{\omega'\in \mathcal{E}_1^\omega} \mu_{\omega,\lambda}(\omega')  }
            = \mathbb{E}_{\omega,\lambda}\bigl[\prod_{p  \in P_{\omega'}}r_\kappa \bigl(\omega'(p)\bigr)\mid \omega' \in \mathcal{E}_1^\omega \bigr].
        \end{split}
    \end{equation*}
    We conclude that
    \begin{equation}\label{eq: P tilde equation forms ii 2}
        \mathbb{E}_\varphi  \Bigl[ \prod_{p  \in P_{\omega'}} r_\kappa \bigl(  \omega'(p) \bigr)\mid P_{\omega'} = P_\omega,\, \omega' \in \mathcal{E} \Bigr] = \mathbb{E}_{\omega,\lambda}\bigl[\prod_{p  \in P_{\omega'}}r_\kappa \bigl(\omega'(p)\bigr)\mid \omega' \in \mathcal{E}_1^\omega \bigr].
    \end{equation}
    Combining~\eqref{eq: P tilde equation forms ii 1} and~\eqref{eq: P tilde equation forms ii 2}, we obtain~\eqref{eq: P tilde equation forms ii} as desired.
\end{proof}

\begin{lemma}\label{lemma: step 1 of end forms ii}
    Let \( \beta,\kappa \geq 0 , \) let \( \gamma \) be a rectangular path, and let \( \omega \in \Omega^2(B_N,\mathbb{Z}_n). \) Then 
    \begin{equation*}
        \begin{split}
            &\mathbb{E}_{\omega, \lambda}\Bigl[ \prod_{p  \in P_\omega} r_\kappa \bigl( \omega'(p) \bigr)  \Bigr] 
            =
            \alpha(\beta,\kappa)^{|P_\omega|} .
        \end{split}
    \end{equation*}
\end{lemma}

\begin{proof}
     Using the definition of \( \mu_{\omega,\lambda}, \) we obtain
    \begin{equation*}
        \begin{split}
            &\mathbb{E}_{\omega, \lambda}\Bigl[ \prod_{p  \in P_\omega} r_\kappa \bigl( \omega'(p) \bigr)  \Bigr] 
            = \Bigl( \sum_{j \in \mathbb{Z}_n} \lambda_j r_\kappa(j) \Bigr)^{|P_\omega|} .
        \end{split}
    \end{equation*}
    Recalling~\eqref{alphasumf}, the desired conclusion immediately follows.
\end{proof}

\begin{lemma}\label{lemma: I forgot the use of ii}
    Let \( \beta,\kappa \geq 0 , \) let \( \gamma \) be a rectangular path, and let \( \omega \in \Omega^2(B_N,\mathbb{Z}_n). \)
    Then
    \begin{equation*}
        \begin{split}
            &\biggl| \mathbb{E}_{\omega, \lambda}\Bigl[\, \prod_{p  \in P_\omega} r_\kappa \bigl( \omega'(p) \bigr) \Bigr] - \mathbb{E}_{\omega, \lambda}\Bigl[\,  \prod_{p  \in P_\omega} r_\kappa \bigl( \omega'(p) \bigr)  \mid \omega' \in \mathcal{E}_1^\omega \Bigr]  \biggr| 
            \leq    C^{(2i)}_{\gamma,\beta,\kappa,m}\, \alpha(\beta,\kappa)^{|P_\omega|} \zeta_\beta,
        \end{split}
    \end{equation*}  
    where
    \begin{equation}\label{eq: def C2i}
        C^{(2i)}_{\gamma,\beta,\kappa,m} \coloneqq (2m-1)|\mathcal{P}_\gamma | \zeta_\beta \xi_\kappa^4
            \frac{ 
            \bigl( \frac{1}{1 - \zeta_\beta\xi_\kappa^2} + \xi_\kappa^2   \bigr)^2   + \xi_\kappa^4 }{1 - (2m-1)|\mathcal{P}_\gamma|\zeta_\beta^2 \xi_\kappa^8}.
    \end{equation}
\end{lemma}

\begin{proof} 
    If \( p \in P_\omega, \) then the number of plaquettes \( p' \in P_\omega \smallsetminus \{ p \} \) with \( p' \sim p \) is at most \( (2(m-1)-1)+2 = 2m-1.\)
    Using the definition of \( \mu_{\omega,\lambda} \) and the fact that $\sum_{j \in \mathbb{Z}_n \setminus \{0\}} \lambda_j = 1 - \lambda_0$, it follows that  
    \begin{equation*}
        \begin{split}
            &\mu_{\omega, \lambda}\bigl( \omega' \notin \mathcal{E}_1 ^\omega\bigr) 
            \leq 
            \sum_{k = 2}^{|P_\omega |} \binom{|P_\omega |}{1} \binom{2m-1}{1} \binom{|P_\omega|-2}{k-2} (1-\lambda_0)^{k}\lambda_0^{|P_\omega|-k} 
            =
            (2m-1)|P_\omega |  (1-\lambda_0)^{2} .
        \end{split}
    \end{equation*} 
    Similarly, we have
    \begin{equation*}
        \begin{split}
            &\mathbb{E}_{\omega, \lambda}\Bigl[ \,  \prod_{p  \in P_\omega} r_\kappa \bigl( \omega'(p) \bigr)  \cdot \mathbb{1}(\omega' \notin \mathcal{E}_1^\omega )\Bigr]
            \leq 
            \sum_{k = 2}^{|P_\omega |} \binom{|P_\omega |}{1} \binom{2m-1}{1} \binom{|P_\omega|-2}{k-2} \Bigl( \sum_{j \in \mathbb{Z}_n\smallsetminus \{ 0 \}}\lambda_j r_\kappa(j) \Bigr)^{k}  \lambda_0^{|P_\omega|-k} 
            \\&\qquad= 
            (2m-1)|P_\omega |
            \Bigl( \sum_{j \in \mathbb{Z}_n }\lambda_j r_\kappa (j)   \Bigr)^{|P_\omega|-2}
            \Bigl( \sum_{j \in \mathbb{Z}_n \smallsetminus \{ 0 \}}\lambda_j r_\kappa(j) \Bigr)^2.
        \end{split}
    \end{equation*}   
    Combining the above equations with~Lemma~\ref{lemma: step 1 of end forms ii} and~\eqref{alphasumf}, we obtain
    \begin{equation*}
        \begin{split}
            &\biggl| \mathbb{E}_{\omega, \lambda}\Bigl[\,  \prod_{p  \in P_\omega} r_\kappa  \bigl( \omega'(p) \bigr)  \Bigr] - \mathbb{E}_{\omega, \lambda}\Bigl[ \, \prod_{p  \in P_\omega} r_\kappa  \bigl( \omega'(p) \bigr)   \mid \omega' \in \mathcal{E}_1^\omega \Bigr]  \biggr|
            \\&\qquad=
            \biggl| \frac{\mathbb{E}_{\omega, \lambda}\pigl[ \prod_{p  \in P_\omega} r_\kappa  \bigl( \omega'(p) \bigr) \cdot \mathbb{1}(\omega' \notin \mathcal{E}_1^\omega) \pigr] -\mathbb{E}_{\omega, \lambda}\pigl[  \prod_{p  \in P_\omega} r_\kappa  \bigl( \omega'(p) \bigr)  \pigr] \mu_{\omega,\lambda}( \omega' \notin \mathcal{E}_1) }{1 - \mu_{\omega,\lambda}(\omega' \notin\mathcal{E}_1)} \biggr|
            \\&\qquad\leq 
            (2m-1)|P_\omega | \Bigl( \sum_{j \in \mathbb{Z}_n }\lambda_j r_\kappa(j)   \Bigr)^{|P_\omega|} 
            \frac{ 
            \Bigl( \sum_{j \in \mathbb{Z}_n \smallsetminus \{ 0 \}}\lambda_j r_\kappa(j) \Bigr)^2 /  (\sum_{j \in \mathbb{Z}_n} \lambda_j r_\kappa(j))^{2} + (1-\lambda_0)^2 }{1 - (2m-1)|P_\omega|(1-\lambda_0)^2}.
        \end{split}
    \end{equation*} In view of the expression in~\eqref{alphasumf} for \(\alpha(\beta,\kappa), \) we get
    \begin{equation*}
        \begin{split}
            &\biggl| \mathbb{E}_{\omega, \lambda}\Bigl[\, \prod_{p  \in P_\omega} r_\kappa\bigl( \omega'(p) \bigr) \Bigr] - \mathbb{E}_{\omega, \lambda}\Bigl[\,  \prod_{p  \in P_\omega} r_\kappa\bigl( \omega'(p) \bigr)  \mid \omega' \in \mathcal{E}_1^\omega \Bigr]  \biggr| 
            \\&\qquad\leq    
            (2m-1)|P_\omega | \zeta_\beta^{-1}
            \frac{ 
            \bigl( \alpha(\beta,\kappa) - \lambda_0   \bigr)^2 /  \alpha(\beta,\kappa)^{2} + (1-\lambda_0)^2 }{1 - (2m-1)|P_\omega|(1-\lambda_0)^2}
            \cdot 
            \alpha(\beta,\kappa)^{|P_\omega|} \zeta_\beta.
        \end{split}
    \end{equation*}   
    Recalling that \( |P_\omega| \leq |\mathcal{P}_\gamma|, \) and noting that, by Lemma~\ref{lemma: Zn properties of alpha} and Lemma~\ref{lemma: lambda inequality}, 
    \begin{equation*}
        0 \leq  \frac{\alpha(\beta,\kappa)-\lambda_0}{\alpha(\beta,\kappa)} 
        =
        \frac{\alpha(\beta,\kappa)-1+1-\lambda_0}{\alpha(\beta,\kappa)} 
        \leq 
        (\alpha(\beta,\kappa)-1)+(1-\lambda_0)
        \leq 
        \frac{\zeta_\beta\xi_\kappa^2}{1 - \zeta_\beta\xi_\kappa^2} + \zeta_\beta\xi_\kappa^4,
    \end{equation*}
    we obtain the desired conclusion.
\end{proof}

\begin{lemma}\label{lemma: upper bound for tilde P forms ii}
    Let \( \beta,\kappa \geq 0 \) and let \( \gamma \) be a rectangular path. Assume that~\eqref{assumption: 1} holds.
    Then
    \begin{equation}\label{eq: upper bound for tilde P forms ii}
        \mathbb{E}_\varphi \bigl[ |\mathcal{P}_\gamma \smallsetminus (P_\omega \cup \pm \mathcal{P}_{\gamma,c}) | \bigr]   
        \leq  |\mathcal{P}_\gamma|\zeta_\beta \cdot \frac{3(2m-3)(8m)^{2}  \xi_\kappa^{4} }{1-(8m)^{2} \zeta_\beta \xi_\kappa^{-1} }.
    \end{equation} 
\end{lemma}

\begin{proof}
    First, note that
    \begin{equation}\label{EvarphisumPvarphi}
    \mathbb{E}_\varphi \bigl[ |\mathcal{P}_\gamma \smallsetminus (P_\omega \cup \pm \mathcal{P}_{\gamma,c}) | \bigr] 
        = 
        \sum_{p \in \mathcal{P}_\gamma \smallsetminus \pm \mathcal{P}_{\gamma,c}} \mathbb{P}_\varphi ( p \in \mathcal{P}_\gamma \smallsetminus P_\omega).
    \end{equation} 
    Fix any \( p \in \mathcal{P}_\gamma\smallsetminus \pm \mathcal{P}_{\gamma,c}.\)
    If \( \omega \in \Omega^2(B_N,\mathbb{Z}_n)\) and \( p \notin P_\omega,\) then \( p \sim (\support \omega)^+ \smallsetminus \mathcal{P}_\gamma.\) 
    Note that 
    \begin{equation}\label{eq: upper bound for tilde P forms ii 1}
        \bigl| \{ p' \in C_2(B_N)^+ \smallsetminus \mathcal{P}_\gamma^+ \colon p^+ \sim p' \} \bigr| \leq \bigl( |\support \partial p|-1\bigr) (2m-3) = 3(2m-3).
    \end{equation}
    Now fix any \( p' \in C_2(B_N)^+ \smallsetminus \mathcal{P}_\gamma^+ \) with \(  p' \sim p. \)
    If \( p' \in \support \omega,\) then  \( |(\support \omega^{\support \partial p'})^+| \geq 1\) and \( |(\support \omega^{\support \partial p'})^+| + |(\support \delta \omega^{\support \partial p'})^+| \geq 5.\) Applying Lemma~\ref{lemma: flip a set forms} with \( P_0 = \{ p' \} ,\) \( k = 1, \) and \( k' = 5 ,\) we see that
    \begin{equation}\label{eq: upper bound for tilde P forms ii 2}
        \mathbb{P}_\varphi \bigl( p' \in (\support \omega)^+ \bigr)
        \leq \frac{(8m)^{2} \zeta_\beta \xi_\kappa^{4} }{1-(8m)^{2} \zeta_\beta \xi_\kappa^{-1} }.
    \end{equation} 
    Combining~\eqref{eq: upper bound for tilde P forms ii 1}~and~\eqref{eq: upper bound for tilde P forms ii 2}, we get
    \begin{equation*}
       \mathbb{P}_\varphi(p \in \mathcal{P}_\gamma \smallsetminus P_\omega) \leq \sum_{\substack{p' \in C_2(B_N)^+ \smallsetminus \mathcal{P}_\gamma \colon\\ p^+ \sim p'}} \mathbb{P}_\varphi \bigl( p' \in (\support \omega)^+ \bigr) \leq  \frac{3(2m-3)(8m)^{2} \zeta_\beta \xi_\kappa^{4} }{1-(8m)^{2} \zeta_\beta \xi_\kappa^{-1} }.
    \end{equation*}
    Summing over all \( p \in \mathcal{P}_\gamma \smallsetminus (P_\omega \cup \pm \mathcal{P}_{\gamma,c})\) and using~\eqref{EvarphisumPvarphi}, we obtain~\eqref{eq: upper bound for tilde P forms ii} as desired.
\end{proof}

We are now ready to give a proof of Proposition~\ref{proposition: last resampling lemma forms ii}.
\begin{proof}[Proof of Proposition~\ref{proposition: last resampling lemma forms ii}]
     Note first that by combining Lemma~\ref{lemma: symmetry} and Lemma~\ref{lemma: order}, we see that \( \xi_\kappa = \varphi_\kappa(1).\) Using Lemma~\ref{lemma: change of complicated term forms ii}, we thus have
    \begin{equation}\label{step1}
        \begin{split}
            &\Bigl| \mathbb{E}_\varphi    \pigl[\widehat{L_\gamma}(\omega) \cdot \mathbb{1}(  \mathcal{E} ) \pigr] - \xi_\kappa^{|\gamma |} \alpha(\beta,\kappa)^{|\mathcal{P}_\gamma|}\Bigr|
            \\&\qquad 
            =
            \varphi_\kappa(1)^{|\gamma |}\biggl| 
            \mathbb{E}_\varphi    \Bigl[\prod_{p  \in \mathcal{P}_\gamma} r_\kappa \bigl( \omega(p) \bigr)\cdot \mathbb{1}(  \mathcal{E} ) \Bigr]  -  \alpha(\beta,\kappa)^{|\mathcal{P}_\gamma|}\biggr|
            .
        \end{split}
    \end{equation}
    Next, note that by combining Lemma~\ref{lemma: P tilde equation forms ii} with Lemma~\ref{lemma: I forgot the use of ii}, we obtain
    \begin{align}\nonumber
            &
            \biggl|\mathbb{E}_\varphi  \Bigl[ \, \prod_{p  \in \mathcal{P}_\gamma} r_\kappa \bigl(  \omega(p) \bigr) \cdot \mathbb{1}(\omega \in \mathcal{E})\Bigr] 
            -
            \mathbb{E}_\varphi  \Bigl[ \mathbb{E}_{ \omega, \lambda} \pigl[ \prod_{p  \in P_\omega} r_\kappa \bigl(  \omega'(p) \bigr)  \pigr]\cdot \mathbb{1}(\omega \in \mathcal{E})\Bigr] \biggr|
            \\\nonumber
            &\qquad=
            \biggl|\mathbb{E}_\varphi  \Bigl[ \mathbb{E}_{ \omega, \lambda} \pigl[ \prod_{p  \in P_\omega} r_\kappa \bigl(  \omega'(p) \bigr) \mid \omega' \in \mathcal{E}_1^\omega \pigr]\cdot \mathbb{1}(\omega \in \mathcal{E})\Bigr]
            - 
            \mathbb{E}_\varphi  \Bigl[ \mathbb{E}_{ \omega, \lambda} \pigl[ \prod_{p  \in P_\omega} r_\kappa \bigl(  \omega'(p) \bigr)  \pigr]\cdot \mathbb{1}(\omega \in \mathcal{E})\Bigr]
            \biggr|
            \\\label{step2}
            &\qquad\leq  
          C^{(2i)}_{\gamma,\beta,\kappa,m}\, \alpha(\beta,\kappa)^{|P_\omega|} \zeta_\beta.
    \end{align}  
    Now fix some \( \omega \in \Omega^2(B_N,\mathbb{Z}_n) .\) 
    By Lemma~\ref{lemma: step 1 of end forms ii}, we have 
    \begin{equation*}
        \begin{split}
            &\mathbb{E}_{\omega, \lambda}\Bigl[ \prod_{p  \in P_\omega} r_\kappa \bigl( \omega'(p) \bigr)  \Bigr] 
            =
            \alpha(\beta,\kappa)^{|P_\omega|}.
        \end{split}
    \end{equation*}
    Let $\Phi := \mathcal{P}_\gamma \setminus (\pm \mathcal{P}_{\gamma,c})$.
    Since \( \alpha(\beta,\kappa) \geq 1 \) by Lemma~\ref{lemma: Zn properties of alpha} and $P_\omega \subseteq \Phi$, it immediately follows that
    \begin{equation*}
        \begin{split}
            &\mathbb{E}_{\omega, \lambda}\Bigl[ \prod_{p  \in P_\omega} r_\kappa \bigl( \omega'(p) \bigr)  \Bigr] 
            \leq
            \alpha(\beta,\kappa)^{|\Phi|}. 
        \end{split}
    \end{equation*}
    At the same time, using again that \( P_\omega \subseteq \Phi, \) we have 
    \begin{equation*}
        \begin{split}
            & \alpha(\beta,\kappa)^{|P_\omega|}
            =  \alpha(\beta,\kappa)^{|\Phi|} \alpha(\beta,\kappa)^{-|\Phi \smallsetminus P_\omega|}
            =  \alpha(\beta,\kappa)^{|\mathcal{P}_\gamma|} \pigl(1-\bigl(1-\alpha(\beta,\kappa)^{-1}\bigr)\pigr)^{|\Phi \smallsetminus P_\omega|}
            \\&\qquad\geq \alpha(\beta,\kappa)^{|\Phi|} \pigl(1-|\Phi \smallsetminus P_\omega|\bigl(1-\alpha(\beta,\kappa)^{-1}\bigr) \pigr)
        \end{split}
    \end{equation*}
    and hence 
    \begin{equation*}
        0 \geq \mathbb{E}_{\omega, \lambda}\pigl[ \prod_{p  \in P_\omega} r_\kappa \bigl( \omega'(p) \bigr)  \pigr] - \alpha(\beta,\kappa)^{|\Phi|} \geq -\alpha(\beta,\kappa)^{|\Phi|} |\Phi \smallsetminus P_\omega| (1 - \alpha(\beta,\kappa)^{-1}). 
    \end{equation*}
    Since
    \begin{equation*}
        \begin{split}
            &0 \geq \mathbb{E}_\varphi \Bigl[\mathbb{E}_{\omega,\lambda}\pigl[\prod_{p  \in P_\omega} r_\kappa \bigl(  \omega'(p) \bigr) \pigr] \Bigr] -  \alpha(\beta,\kappa)^{|\Phi |}
            \geq
            \mathbb{E}_\varphi \Bigl[\mathbb{E}_{\omega,\lambda}\pig[\prod_{p  \in P_\omega} r_\kappa \bigl(  \omega'(p) \bigr) \pigr] \cdot \mathbb{1}(\omega \in \mathcal{E})\Bigr] -  \alpha(\beta,\kappa)^{|\Phi|}
            \\&\qquad=
            -\mathbb{E}_\varphi \Bigl[\mathbb{E}_{\omega,\lambda}\pigl[\prod_{p  \in P_\omega} r_\kappa \bigl(  \omega'(p) \bigr) \pigr] \cdot \mathbb{1}(\omega \notin \mathcal{E})\Bigr]
            + 
            \mathbb{E}_\varphi \Bigl[\mathbb{E}_{\omega,\lambda}\pigl[\prod_{p  \in P_\omega} r_\kappa \bigl(  \omega'(p) \bigr) \pigr] \Bigr]-  \alpha(\beta,\kappa)^{|\Phi|},
        \end{split}
    \end{equation*}
    it follows that
    \begin{equation*}
        \begin{split}
            &\biggl| \mathbb{E}_\varphi \Bigl[\mathbb{E}_{\omega,\lambda}\pigl[\prod_{p  \in P_\omega} r_\kappa \bigl(  \omega'(p) \bigr) \pigr] \cdot \mathbb{1}(\omega \in \mathcal{E})\Bigr] -  \alpha(\beta,\kappa)^{|\Phi|} \biggr|
            \\&\qquad \leq  \alpha(\beta,\kappa)^{|\Phi|} 
            \cdot 
            \mathbb{P}_\varphi(\omega \notin\mathcal{E}) 
            +\alpha(\beta,\kappa)^{|\Phi|} 
            \cdot 
            \bigl(1-\alpha(\beta,\kappa)^{-1}\bigr)
            \cdot 
            \mathbb{E}_\varphi \bigl[ |\Phi \smallsetminus P_\omega| \bigr].
        \end{split}
    \end{equation*} 
    Finally, using Lemma~\ref{lemma: Zn properties of alpha}, we note that
    \begin{equation*}
        \begin{split}
            &\bigl|\alpha(\beta,\kappa)^{|\Phi|} - \alpha(\beta,\kappa)^{|\mathcal{P}_\gamma|}\bigr| 
            = \alpha(\beta,\kappa)^{|\mathcal{P}_\gamma|} \zeta_\beta \cdot
            \pigl|1 - \alpha(\beta,\kappa)^{-|\mathcal{P}_{\gamma,c}|} \pigr| \zeta_\beta^{-1}
            \\&\qquad\leq
            \alpha(\beta,\kappa)^{|\mathcal{P}_\gamma|} \zeta_\beta \cdot
            \pigl|1 - (1-\zeta_\beta \xi_\kappa)^{|\mathcal{P}_{\gamma,c}|} \pigr| \zeta_\beta^{-1}
            \leq
            \alpha(\beta,\kappa)^{|\mathcal{P}_\gamma|} \zeta_\beta \cdot
            |\mathcal{P}_{\gamma,c}| \zeta_\beta \xi_\kappa \cdot \zeta_\beta^{-1}
            \\&\qquad
            \leq
            \alpha(\beta,\kappa)^{|\mathcal{P}_\gamma|} \zeta_\beta \cdot
             |\mathcal{P}_{\gamma,c}| \xi_\kappa .
        \end{split}
    \end{equation*}
    Using Lemma~\ref{lemma: upper bound for tilde P forms ii}, Lemma~\ref{lemma: upper bound for E2 forms}, and Lemma~\ref{lemma: Zn properties of alpha}, we thus obtain
    \begin{equation}\label{step3}
        \begin{split}
            &\biggl| \mathbb{E}_\varphi \Bigl[\mathbb{E}_{\omega,\lambda}\pigl[\prod_{p  \in P_\omega} r_\kappa \bigl(  \omega'(p) \bigr) \pigr] \cdot \mathbb{1}(\omega \in \mathcal{E})\Bigr] -  \alpha(\beta,\kappa)^{|\Phi|} \biggr|
            \\&\qquad \leq 
            \alpha(\beta,\kappa)^{|\mathcal{P}_\gamma |}  \zeta_\beta
            \cdot \biggl( 
            C^{(2ii)}_{\gamma,\beta,\kappa, m}
            + 
            \zeta_\beta \xi_\kappa^2 
            \cdot 
            |\mathcal{P}_\gamma|\cdot \frac{3(2m-3)(8m)^{2}  \xi_\kappa^{4} }{1-(8m)^{2} \zeta_\beta \xi_\kappa^{-1} } + |\mathcal{P}_{\gamma,c}| \xi_\kappa \biggr).
        \end{split}
    \end{equation}  
    Combining~\eqref{step1},~\eqref{step2},~and~\eqref{step3}, using the triangle inequality, and defining
    \begin{equation}\label{eq: def C2iii ii}
        C^{(2iii)}_{\gamma,\beta,\kappa,m} \coloneqq 
        \zeta_\beta \xi_\kappa^2 
            \cdot 
            |\mathcal{P}_\gamma|\cdot \frac{3(2m-3)(8m)^{2}  \xi_\kappa^{4} }{1-(8m)^{2} \zeta_\beta \xi_\kappa^{-1} } + |\mathcal{P}_{\gamma,c}| \xi_\kappa,
    \end{equation}
    we obtain the desired conclusion. 
\end{proof}

\section{Main result, and proof of Theorem~1.3}\label{sec: main result proof}

 Theorem~\ref{theorem: first theorem Z2} is an immediate consequence of the following more general theorem.
\begin{theorem}[Small $\beta$ behavior of the Wilson line expectation value]\label{theorem: first theorem}
     Let \( m \geq 2, \) let \( G = \mathbb{Z}_n, \) let \( \beta, \kappa \geq 0, \) and let  \( \gamma \) be a path along the boundary of a rectangle with side lengths \( \ell_1,\ell_2 \geq 7. \) Suppose that \eqref{assumption: 1} and \eqref{assumption: 3} hold.
    Then
    \begin{equation}\label{eq: first theorem}
        \begin{split}
            &\biggr| \langle L_\gamma \rangle_{\beta,\kappa,\infty} - \xi_\kappa^{|\gamma |} \alpha(\beta,\kappa)^{|\mathcal{P}_\gamma|} \biggr| 
            \leq   
            C^{(0)}_{\gamma,\beta,\kappa,m}\, \xi_\kappa^{|\gamma |} \alpha(\beta,\kappa)^{{|\mathcal{P}_\gamma|}}  \zeta_\beta,
        \end{split}
    \end{equation} 
    where $\langle L_\gamma \rangle_{\beta,\kappa, \infty}$ is the  limit defined by~\eqref{infinitevolumelimit}, \( \alpha(\beta,\kappa) \) is  defined in~\eqref{eq: def alpha}, and \( C^{(0)}_{\gamma,\beta,\kappa,m} \) is defined in~\eqref{eq: def last constant}.
\end{theorem}

\begin{remark}
    One can show that, if \( \kappa\) is fixed, then \( \smash{C_{\gamma,\beta,\kappa,m}^{(0)} = O(1) \bigl( 1 + O( \zeta_\beta |\gamma |) \bigr)^2 .}\) In other words, for the right-hand side of~\eqref{eq: first theorem} to be smaller than \( \xi_\kappa^{|\gamma |}  \alpha(\beta,\kappa)^{|\mathcal{P}_\gamma|},\) \( \zeta_\beta^3 |\gamma |^2 \) needs to be small. In particular, this holds either when \( \gamma\) is fixed and \( \beta \to 0,\) or when \( \zeta_\beta^{1+\varepsilon} |\gamma | \asymp 1\) for some \( \varepsilon \in (0,1/2),\) in which case the term \( \alpha(\beta,\kappa)^{|\mathcal{P}_\gamma|}\) is non-trivial.
\end{remark}

\begin{proof}[Proof of Theorem~\ref{theorem: first theorem}]
    By Proposition~\ref{proposition: high-temperature expansion ALHM 3}, we have 
    \begin{equation*} 
        \begin{split}
            &\mathbb{E}_{N,\beta,\kappa}\bigl[L_\gamma(\sigma)\bigr] 
            =
            \mathbb{E}_\varphi    \bigl[\widehat{L_\gamma}(\omega) \bigr] ,
        \end{split}
    \end{equation*} 
    and by Proposition~\ref{proposition: useful upper bound forms ii}, we have
    \begin{equation*}
        \begin{split}
            &
            \mathbb{E}_\varphi    \pigl[ \widehat{L_\gamma}(\omega) \cdot \mathbb{1}( \omega \notin \mathcal{E} ) \pigr] 
            \leq     C^{(1)}_{\gamma,\beta,\kappa,m}\, \xi_\kappa^{|\gamma |}   \zeta_\beta.
        \end{split}
    \end{equation*}  
    Combining these equations, we obtain
    \begin{equation}\label{eq: thm proof eq 1}
        \begin{split}
            \Bigl|\mathbb{E}_{N,\beta,\kappa}\bigl[L_\gamma(\sigma))\bigr] 
            -
            \mathbb{E}_\varphi    \pigl[ \widehat{L_\gamma}(\omega)\cdot \mathbb{1}(  \mathcal{E} ) \pigr] \Bigr|
            \leq 
            C^{(1)}_{\gamma,\beta,\kappa,m}\, \xi_\kappa^{|\gamma |}   \zeta_\beta.
        \end{split}
    \end{equation}
    On the other hand, by Proposition~\ref{proposition: last resampling lemma forms ii}, we have
    \begin{equation}\label{eq: thm proof eq 2}
        \begin{split}
            &\Bigl| \mathbb{E}_\varphi    \pigl[\widehat{L_\gamma}(\omega)\cdot \mathbb{1}(  \mathcal{E} ) \pigr] - \xi_\kappa^{|\gamma |} \alpha(\beta,\kappa)^{|\mathcal{P}_\gamma|}\Bigr|
            \leq
            C^{(2)}_{\gamma,\beta,\kappa,m} \xi_\kappa^{|\gamma |}
            \alpha(\beta,\kappa)^{|\mathcal{P}_\gamma |} \zeta_\beta.
        \end{split}
    \end{equation}
    Combining~\eqref{eq: thm proof eq 1} and~\eqref{eq: thm proof eq 2}, we see that
    \begin{equation*}
        \begin{split}
            \biggl|\mathbb{E}_{N,\beta,\kappa}\bigl[ L_\gamma(\sigma))\bigr] 
            -
            \xi_\kappa^{|\gamma |} \alpha(\beta,\kappa)^{|\mathcal{P}_\gamma|} \biggr|
            \leq 
            C^{(1)}_{\gamma,\beta,\kappa,m}\, \xi_\kappa^{|\gamma |}  \zeta_\beta
            +
            C^{(2)}_{\gamma,\beta,\kappa,m} \xi_\kappa^{|\gamma |}
            \alpha(\beta,\kappa)^{|\mathcal{P}_\gamma |} \zeta_\beta.
        \end{split}
    \end{equation*}
    Letting
    \begin{equation}\label{eq: def last constant}
        C^{(0)}_{\gamma,\beta,\kappa,m} \coloneqq C^{(1)}_{\gamma,\beta,\kappa,m} \alpha(\beta,\kappa)^{-|\mathcal{P}_\gamma |}
        +
        C^{(2)}_{\gamma,\beta,\kappa,m} ,
    \end{equation}
    and applying Corollary~\ref{corollary: unitary gauge} and Proposition~\ref{proposition: limit exists}, the desired conclusion follows.   
\end{proof}

\appendix

\section{The calculations in Lemma~\ref{lemma: upper bound on expectation of two events}}

In this section, for completeness, we collect the calculations that were omitted from the proof of Lemma~\ref{lemma: upper bound on expectation of two events}. The notation in this section uses the notation of this proof.

\begin{sublemma}\label{sublemma: upper bound on expectation of two events 1}
    Let \( B_1\) be defined by~\eqref{eq: upper bound on expectation of two events 3}. Then
    \begin{align*}
        &
        B_1
        \leq
        \frac{ (16m)^2\zeta_\beta\xi_\kappa^{-1}\xi_\kappa^{|\gamma |}}{(1 - \xi_\kappa)(1-(16m)^{2} \zeta_\beta\xi_\kappa^{-1})}
        \Bigl( \bigl( 1+(16m)^{2} \zeta_\beta\xi_\kappa^{-2} \bigr)^{|\mathcal{P}_{\gamma,c}|}
        \pigl( 1 + (16m)^{2}\zeta_\beta \pigr)^{|\gamma |} -1 \Bigr)
        \\&\qquad\qquad+
        \frac{(16m)^{2}\zeta_\beta\xi_\kappa^{|\gamma |}}{(1 - \xi_\kappa)(1 - (16m)^{2} \zeta_\beta)} \Bigl( \bigl( 1 + (16m)^{2}\zeta_\beta  \xi_\kappa^{-2} \bigr)^{|\mathcal{P}_{\gamma,c}|} 
        \pigl( 1 + (16m)^{4} \zeta_\beta^{2} \xi_\kappa^{- 1}  \pigr)^{|\gamma |} -1 \Bigr).
    \end{align*}
\end{sublemma}

\begin{subproof}
    Using the formula for a geometric sum, we obtain
    \begin{align*}
        &
        B_1
        =
        \frac{\xi_\kappa^{|\gamma |}}{1 - \xi_\kappa}\sum_{i = 0}^{|\mathcal{P}_{\gamma,c}|} \sum_{j=\max(1,2i)}^{|\gamma |} \binom{|\gamma |}{j-2i}\binom{|\mathcal{P}_{\gamma,c}|}{i} 
        \xi_\kappa^{- j} \sum_{k = j-i+1}^\infty (16m)^{2k} \zeta_\beta^k \xi_\kappa^{\max(0,2j-3i-k)}.
    \end{align*}
    Equivalently, \( B_1 = B_1' + B_1'',\) where
    \begin{align*}
        &B_1' \coloneqq \frac{\xi_\kappa^{|\gamma |}}{1 - \xi_\kappa}\sum_{i = 0}^{|\mathcal{P}_{\gamma,c}|} \sum_{j=\max(1,2i)}^{|\gamma |} \binom{|\gamma |}{j-2i}\binom{|\mathcal{P}_{\gamma,c}|}{i} 
        \xi_\kappa^{- j} \sum_{k = j-i+1}^{2j-3i}
        (16m)^{2k} \zeta_\beta^k \xi_\kappa^{2j-3i-k} 
        \\&\qquad=
        \frac{\xi_\kappa^{|\gamma |}}{1 - \xi_\kappa}\sum_{i = 0}^{|\mathcal{P}_{\gamma,c}|} \sum_{j=\max(1,2i)}^{|\gamma |} \binom{|\gamma |}{j-2i}\binom{|\mathcal{P}_{\gamma,c}|}{i} 
        \xi_\kappa^{ - j+2j-3i} 
        \sum_{k = j-i+1}^{2j-3i}
        (16m)^{2k} \zeta_\beta^k \xi_\kappa^{-k} 
    \end{align*}
    and 
    \begin{align*}
        &B_1'' \coloneqq\frac{\xi_\kappa^{|\gamma |}}{1 - \xi_\kappa}\sum_{i = 0}^{|\mathcal{P}_{\gamma,c}|} \sum_{j=\max(1,2i)}^{|\gamma |} \binom{|\gamma |}{j-2i}\binom{|\mathcal{P}_{\gamma,c}|}{i} 
        \xi_\kappa^{- j} 
        \sum_{k = \max(j-i+1, 2j-3i+1)}^\infty (16m)^{2k} \zeta_\beta^k \xi_\kappa^{0} 
        \\&\qquad
       = \frac{\xi_\kappa^{|\gamma |}}{1 - \xi_\kappa}\sum_{i = 0}^{|\mathcal{P}_{\gamma,c}|} \sum_{j=\max(1,2i)}^{|\gamma |} \binom{|\gamma |}{j-2i}\binom{|\mathcal{P}_{\gamma,c}|}{i} \xi_\kappa^{- j} 
        \sum_{k = 2j-3i+1}^\infty
        (16m)^{2k} \zeta_\beta^k .
    \end{align*}
    We now give upper bounds for \( B_1' \) and \( B_1''.\)
    Using the formula for a geometric sum, we find
    \begin{equation*}
        \begin{split}
            &B_1' =
            \frac{\xi_\kappa^{|\gamma |}}{(1 - \xi_\kappa)(1-(16m)^{2} \zeta_\beta\xi_\kappa^{-1})}
            \\&\qquad\qquad=
            \sum_{i = 0}^{|\mathcal{P}_{\gamma,c}|} \sum_{j=\max(1,2i)}^{|\gamma |} \binom{|\gamma |}{j-2i} \binom{|\mathcal{P}_{\gamma,c}|}{i}
            \xi_\kappa^{- j+2j-3i} 
            (16m)^{2(j-i+1)} \zeta_\beta^{j-i+1} \xi_\kappa^{-(j-i+1)},
        \end{split}
    \end{equation*}
    and, analogously, 
    \begin{equation*}
        \begin{split}
            &B_1'' =
            \frac{\xi_\kappa^{|\gamma |}}{(1 - \xi_\kappa)(1 - (16m)^{2} \zeta_\beta)}\sum_{i = 0}^{|\mathcal{P}_{\gamma,c}|} \sum_{j=\max(1,2i)}^{|\gamma |} \binom{|\gamma |}{j-2i}\binom{|\mathcal{P}_{\gamma,c}|}{i} 
            (16m)^{2(2j-3i+1)} \zeta_\beta^{2j-3i+1} \xi_\kappa^{- j}.
        \end{split}
    \end{equation*}
    Simplifying these expressions, we see that
    \begin{align*}
        &
        B_1' = \frac{(16m)^2 \zeta_\beta\xi_\kappa^{-1}\xi_\kappa^{|\gamma |}}{(1 - \xi_\kappa)(1-(16m)^{2} \zeta_\beta\xi_\kappa^{-1})}\sum_{i = 0}^{|\mathcal{P}_{\gamma,c}|} \sum_{j=\max(1,2i)}^{|\gamma |} \binom{|\gamma |}{j-2i}\binom{|\mathcal{P}_{\gamma,c}|}{i} 
        \xi_\kappa^{j-3i} 
        (16m)^{2(j-i)} \zeta_\beta^{j-i} \xi_\kappa^{-(j-i)}  
        \\&\qquad=
        \frac{(16m)^2 \zeta_\beta\xi_\kappa^{-1}\xi_\kappa^{|\gamma |}}{(1 - \xi_\kappa)(1-(16m)^{2} \zeta_\beta\xi_\kappa^{-1})}
        \\&\qquad\qquad\times
        \sum_{i = 0}^{|\mathcal{P}_{\gamma,c}|} \binom{|\mathcal{P}_{\gamma,c}|}{i} 
        (16m)^{2i} \zeta_\beta^{i}\xi_\kappa^{-2i}
        \sum_{j=\max(1,2i)}^{|\gamma |} \binom{|\gamma |}{j-2i}
        (16m)^{2(j-2i)} \zeta_\beta^{j-2i}
    \end{align*}
    and
    \begin{align*}
        &
        B_1'' 
        =
        \frac{(16m)^{2} \zeta_\beta\xi_\kappa^{|\gamma |}}{(1 - \xi_\kappa)(1 - (16m)^{2} \zeta_\beta)}\sum_{i = 0}^{|\mathcal{P}_{\gamma,c}|} \sum_{j=\max(1,2i)}^{|\gamma |} \binom{|\gamma |}{j-2i}\binom{|\mathcal{P}_{\gamma,c}|}{i} 
        (16m)^{2(2j-3i)} \zeta_\beta^{2j-3i} \xi_\kappa^{- j}
        \\&\qquad=
        \frac{(16m)^{2} \zeta_\beta\xi_\kappa^{|\gamma |}}{(1 - \xi_\kappa)(1 - (16m)^{2} \zeta_\beta)}\sum_{i = 0}^{|\mathcal{P}_{\gamma,c}|} \binom{|\mathcal{P}_{\gamma,c}|}{i} (16m)^{2i} \zeta_\beta^{i}  \xi_\kappa^{-2i} 
        \\&\qquad\qquad\quad \times \sum_{j=\max(1,2i)}^{|\gamma |} \binom{|\gamma |}{j-2i} (16m)^{4(j-2i)} \zeta_\beta^{2(j-2i)} \xi_\kappa^{- (j-2i)}.
    \end{align*}

    Using the identities \( \sum_{j=0}^\infty \binom{n}{j} p^j = (1+p)^n\) and \( \sum_{j=1}^\infty \binom{n}{j} p^j = (1+p)^n-1,\) we obtain the upper bounds
    \begin{align*}
        &B_1' \leq
        \frac{(16m)^2 \zeta_\beta\xi_\kappa^{-1}\xi_\kappa^{|\gamma |}}{(1 - \xi_\kappa)(1-(16m)^{2} \zeta_\beta\xi_\kappa^{-1})}\sum_{i = 0}^{|\mathcal{P}_{\gamma,c}|} \binom{|\mathcal{P}_{\gamma,c}|}{i} 
        (16m)^{2i} \zeta_\beta^{i} \xi_\kappa^{-2i}
        \Bigl(\pigl( 1 + (16m)^{2}\zeta_\beta \pigr)^{|\gamma |} - \mathbb{1}(i=0) \Bigr)
    \end{align*}
    and
    \begin{align*}
        &B_1'' 
        \leq \frac{(16m)^{2}\zeta_\beta\xi_\kappa^{|\gamma |}}{(1 - \xi_\kappa)(1 - (16m)^{2} \zeta_\beta)}\sum_{i = 0}^{|\mathcal{P}_{\gamma,c}|} \binom{|\mathcal{P}_{\gamma,c}|}{i} (16m)^{2i}\zeta_\beta^{i}  \xi_\kappa^{-2i} 
        \Bigl( \pigl( 1 + (16m)^{4} \zeta_\beta^{2}\xi_\kappa^{- 1}  \pigr)^{|\gamma |} - \mathbb{1}(i = 0)\Bigr).
    \end{align*}
    Finally, using again the identity \( \sum_{j=0}^\infty \binom{n}{j} p^j = (1+p)^n,\) we see that
    \begin{align*}
        &B_1' \leq \frac{(16m)^2\zeta_\beta\xi_\kappa^{-1}\xi_\kappa^{|\gamma |}}{(1 - \xi_\kappa)(1-(16m)^{2} \zeta_\beta\xi_\kappa^{-1})}
        \Bigl( \bigl( 1+(16m)^{2} \zeta_\beta\xi_\kappa^{-2} \bigr)^{|\mathcal{P}_{\gamma,c}|}
        \pigl( 1 + (16m)^{2}\zeta_\beta \pigr)^{|\gamma |} -1 \Bigr)
    \end{align*}
    and
    \begin{align*}
        &B_1'' \leq \frac{(16m)^{2}\zeta_\beta\xi_\kappa^{|\gamma |}}{(1 - \xi_\kappa)(1 - (16m)^{2} \zeta_\beta)} \Bigl( \bigl( 1 + (16m)^{2}\zeta_\beta  \xi_\kappa^{-2} \bigr)^{|\mathcal{P}_{\gamma,c}|} 
        \pigl( 1 + (16m)^{4} \zeta_\beta^{2} \xi_\kappa^{- 1}  \pigr)^{|\gamma |} -1 \Bigr) .
    \end{align*}
    Since \( B_1 = B_1' + B_1'',\) this concludes the proof. 
\end{subproof}

\begin{sublemma}\label{sublemma: upper bound on expectation of two events 2}
    Let \( B_2\) be defined by~\eqref{eq: def B2}. Then
    \begin{align*}
        &B_2
        \leq
        \frac{\xi_\kappa^{|\gamma |} (16m)^4 \zeta_\beta}{1-\xi_\kappa}\biggl(  
        |\gamma |  \zeta_\beta  
        +
        2
        |\mathcal{P}_{\gamma,c}|  \zeta_\beta \xi_\kappa^{-2} 
        \biggr) \bigl(1 + (16m)^2\zeta_\beta \xi_\kappa^{-2} \bigr)^{|\mathcal{P}_{\gamma,c}| } \bigl( 1 + (16m)^2 \zeta_\beta\bigr)^{|\gamma |}.
    \end{align*}
\end{sublemma}

\begin{subproof}
    Note first that if \( j \geq 2i+1,\) then \( 2j-2i \geq j.\) Using this observation, we see that
    \begin{align*}
        &B_2
        =
        \sum_{i=0}^{|\mathcal{P}_{\gamma,c}|} \sum_{j = \max(2i+1,2)}^{|\gamma |} (j-1) \binom{|\gamma |}{j-2i-1}\binom{|\mathcal{P}_{\gamma,c}|}{i} \sum_{k'=2j-2i}^\infty
        (16m)^{2(j-i)} \zeta_\beta^{j-i} \xi_\kappa^{|\gamma | + k' - 2j}.
    \end{align*}
    Using the formula for geometric sums, it follows that \(B_2 = B_2' + B_2'',\) where
    \begin{equation*}
        \begin{split}
            &B_2' \coloneqq \frac{\xi_\kappa^{|\gamma |}}{1-\xi_\kappa}\sum_{i=0}^{|\mathcal{P}_{\gamma,c}|} \binom{|\mathcal{P}_{\gamma,c}|}{i} 
            (16m)^{2i+2}\zeta_\beta^{i+1} \xi_\kappa^{-2i}
            \\&\qquad\qquad\times \sum_{j = \max(2i+1,2)}^{|\gamma |} (j-2i-1) \binom{|\gamma |}{j-2i-1}
            (16m)^{2(j-2i-1)} \zeta_\beta^{j-2i-1} 
        \end{split}
    \end{equation*}
    and
    \begin{equation*}
        \begin{split}
            &B_2'' \coloneqq \frac{\xi_\kappa^{|\gamma |}}{1-\xi_\kappa}\sum_{i=0}^{|\mathcal{P}_{\gamma,c}|} 2i \binom{|\mathcal{P}_{\gamma,c}|}{i} 
            (16m)^{2i+2}\zeta_\beta^{i+1} \xi_\kappa^{-2i}
            \\&\qquad\qquad\times\sum_{j = \max(2i+1,2)}^{|\gamma |}  \binom{|\gamma |}{j-2i-1}
            (16m)^{2(j-2i-1)} \zeta_\beta^{j-2i-1}.
        \end{split}
    \end{equation*}
    We now rewrite these expressions. 
    Using the identity \( \sum_{j=0}^n \binom{n}{j} j p^j = np(1+p)^{n-1},\) valid for \( p > 0 \) and \( n \geq 0,\) we obtain
    \begin{equation}\label{eq: B2' final}
        B_2' = \frac{\xi_\kappa^{|\gamma |}(16m)^2  \zeta_\beta^{1}}{1-\xi_\kappa} 
        \bigl( 1 + (16m)^{2}\zeta_\beta\xi_\kappa^{-2} \bigr)^{|\mathcal{P}_{\gamma,c}|}
        \pigl( |\gamma | (16m)^{2} \zeta_\beta  \bigl( 1 + (16m)^{2} \zeta_\beta \bigr)^{|\gamma |-1}\pigr) 
    \end{equation}
    and 
    \begin{equation}\label{eq: B2'' final}
        B_2'' = \frac{2\xi_\kappa^{|\gamma |} (16m)^{2}\zeta_\beta^{1} }{1-\xi_\kappa}
        |\mathcal{P}_{\gamma,c}| (16m)^2\zeta_\beta \xi_\kappa^{-2} \bigl(1 + (16m)^2\zeta_\beta \xi_\kappa^{-2} \bigr)^{|\mathcal{P}_{\gamma,c}|-1 }
         \pigl(\bigl( 1 + (16m)^2 \zeta_\beta\bigr)^{|\gamma |}  \pigr) .
    \end{equation}
    Combining~\eqref{eq: B2' final}~and~\eqref{eq: B2'' final} and recalling that \( B_2 = B_2' + B_2'',\) the claim follows.
\end{subproof}

\begin{sublemma}\label{sublemma: upper bound on expectation of two events 3}
    Let \( B_3\) be defined by~\eqref{eq: B3}. Then
    \begin{align*}
        &B_3
        \leq 
        \frac{\xi_\kappa^{|\gamma |}(16m)^{4}\zeta_\beta^2  |\gamma |  \bigl( 1 + (16m)^{2} \zeta_\beta  \xi_\kappa^{2}  \bigr)^{|\gamma |}}{1-\xi_\kappa}
        \biggl(
        \frac{\xi_\kappa^{4}}{(1-(16m)^{2} \zeta_\beta \xi_\kappa^{-1})}
        +
        \frac{(16m)^{8} \zeta_\beta^{4} }{(1-(16m)^{2} \zeta_\beta )} 
        \biggr) .
    \end{align*}   
\end{sublemma}

\begin{subproof}
    We first use the formula for geometric sums to write
    \begin{align*}
        &
        B_3 =
        \frac{\xi_\kappa^{|\gamma |}}{1-\xi_\kappa}
        \sum_{j = 0}^{|\gamma |}
        \binom{|\gamma |}{j+1}  (j+1) \sum_{\hat k = j+2}^\infty 
        (16m)^{2\hat k} \zeta_\beta^{\hat k} \xi_\kappa^{4j + \max(0,j+6-\hat k)-2j}.
    \end{align*}
    We have \( B_3 = B_3' + B_3'',\) where
    \begin{equation*}
        B_3' \coloneqq \frac{\xi_\kappa^{|\gamma |}}{1-\xi_\kappa}
        \sum_{j = 0}^{|\gamma |}
        \binom{|\gamma |}{j+1}  (j+1) \xi_\kappa^{3j+6}
        \sum_{\hat k = j+2}^{j+5}
        (16m)^{2\hat k} \zeta_\beta^{\hat k} \xi_\kappa^{-\hat k}
    \end{equation*}
    and 
    \begin{equation*}
        B_3'' \coloneqq \frac{\xi_\kappa^{|\gamma |}}{1-\xi_\kappa}
        \sum_{j = 0}^{|\gamma |}
        \binom{|\gamma |}{j+1}  (j+1) \xi_\kappa^{2j}\sum_{\hat k = j+6}^\infty 
        (16m)^{2\hat k} \zeta_\beta^{\hat k}.
    \end{equation*}
    We now give upper bounds for \( B_3'\) and \(B_3'',\) which, when combined, imply the desired conclusion.
    We first give an upper bound for \( B_3'. \) Using the formula for geometric sums, we infer that
    \begin{equation*}
        \begin{split}
            &B_3' \leq \frac{\xi_\kappa^{|\gamma |}}{1-\xi_\kappa}
            \sum_{j = 0}^{|\gamma |}
            \binom{|\gamma |}{j+1}  (j+1) \xi_\kappa^{3j+6}
            \sum_{\hat k = j+2}^{\infty}
            (16m)^{2\hat k} \zeta_\beta^{\hat k} \xi_\kappa^{-\hat k}
            \\&\qquad=
            \frac{\xi_\kappa^{|\gamma |}(16m)^{2}\zeta_\beta \xi_\kappa^2 }{(1-\xi_\kappa)(1-(16m)^{2} \zeta_\beta \xi_\kappa^{-1})}
            \sum_{j = 0}^{|\gamma |}
            \binom{|\gamma |}{j+1}  (j+1) 
            (16m)^{2(j+1)} \zeta_\beta^{j+1}  \xi_\kappa^{2(j+1)}.
        \end{split}
    \end{equation*}
    Using the identity \( \sum_{j=0}^n \binom{n}{j} j p^j = np(1-p)^{n-1},\) valid for \( p \in (0,1)\) and \( n \geq 0,\) we obtain 
    \begin{equation}\label{eq: B3' final bound}
        \begin{split}
            &B_3' \leq \frac{\xi_\kappa^{|\gamma |}(16m)^{2}\zeta_\beta \xi_\kappa^2}{(1-\xi_\kappa)(1-(16m)^{2} \zeta_\beta \xi_\kappa^{-1})}
            \cdot |\gamma | (16m)^{2} \zeta_\beta  \xi_\kappa^{2} \bigl( 1 + (16m)^{2} \zeta_\beta  \xi_\kappa^{2} \bigr)^{|\gamma |}
            \\&\qquad=
            \frac{\xi_\kappa^{|\gamma |}(16m)^{4}\zeta_\beta^2  |\gamma |  \bigl( 1 + (16m)^{2} \zeta_\beta  \xi_\kappa^{2}  \bigr)^{|\gamma |}}{1-\xi_\kappa}
            \cdot 
            \frac{\xi_\kappa^{4}}{(1-(16m)^{2} \zeta_\beta \xi_\kappa^{-1})}.
        \end{split}
    \end{equation}
    We now derive an upper bound for \( B_3''.\) Using the formula for geometric sums, we see that
    \begin{equation*}
        \begin{split}
            &B_3'' = 
            \frac{\xi_\kappa^{|\gamma |}}{(1-\xi_\kappa)(1-(16m)^{2} \zeta_\beta )}
            \sum_{j = 0}^{|\gamma |}
            \binom{|\gamma |}{j+1}  (j+1) \xi_\kappa^{2j} (16m)^{2(j+6)} \zeta_\beta^{j+6}  
            \\&\qquad= 
            \frac{\xi_\kappa^{|\gamma |} (16m)^{10} \zeta_\beta^{5} \xi_\kappa^{-2}}{(1-\xi_\kappa)(1-(16m)^{2} \zeta_\beta )}
            \sum_{j = 0}^{|\gamma |}
            \binom{|\gamma |}{j+1}  (j+1) 
            (16m)^{2(j+1)} \zeta_\beta^{j+1} \xi_\kappa^{2(j+1)} .
        \end{split}
    \end{equation*}
    Using the identity \( \sum_{j=0}^n \binom{n}{j} j p^j = np(1-p)^{n-1},\) valid for \( p \in (0,1)\) and \( n \geq 0,\) we find that
    \begin{equation}\label{eq: B3'' final bound}
        \begin{split}
            &B_3'' \leq \frac{\xi_\kappa^{|\gamma |} (16m)^{10} \zeta_\beta^{5} \xi_\kappa^{-2}}{(1-\xi_\kappa)(1-(16m)^{2} \zeta_\beta )}
            \cdot |\gamma | (16m)^{2} \zeta_\beta  \xi_\kappa^{2} \bigl( 1 + (16m)^{2} \zeta_\beta  \xi_\kappa^{2}  \bigr)^{|\gamma |}
            \\&\qquad=
            \frac{\xi_\kappa^{|\gamma |}(16m)^{4}\zeta_\beta^2  |\gamma |  \bigl( 1 + (16m)^{2} \zeta_\beta  \xi_\kappa^{2}  \bigr)^{|\gamma |}}{1-\xi_\kappa}
            \cdot
            \frac{(16m)^{8} \zeta_\beta^{4} }{(1-(16m)^{2} \zeta_\beta )} .
        \end{split}
    \end{equation}
    Combining~\eqref{eq: B3' final bound} and~\eqref{eq: B3'' final bound}, and recalling that \( B_3 = B_3' + B_3'',\) we find the desired conclusion.  
\end{subproof}


\begin{thebibliography}{99}

    
    \bibitem{a2021} Adhikari, A., Wilson Loop Expectations for Non-Abelian Gauge Fields Coupled to a Higgs Boson at Low and High Disorder, Commun.\ Math.\ Phys. 405, 117 (2024).
    
    \bibitem{ac2022} Adhikari, A., Cao, S., Correlation decay for finite lattice gauge theories at weak coupling, preprint (2022).
    
    
    \bibitem{sc2019} Cao, S., Wilson loop expectations in lattice gauge theories with finite gauge groups, Commun.\ Math.\ Phys.\ 380, (2020), 1439--1505.

    \bibitem{c-survey} Chatterjee, S., Yang-Mills for Probabilists, in Springer Proceedings in Mathematics and Statistics Vol 283 (2019).

     \bibitem{c2021} Chatterjee, S., A probabilistic mechanism for quark confinement. To appear in Commun.\ Math.\ Phys.\ (2021).

    \bibitem{c2019} Chatterjee, S., Wilson loop expectations in Ising lattice gauge theory, Commun.\ Math.\ Phys.\ 377, (2020), 307--340.
    
    \bibitem{e1975} Elitzur, S. Impossibility of spontaneously breaking local symmetries. Phys. Rev. D. 12 (12): 3978–3982. (1975).
    
    \bibitem{f2021} Forsstr\"om, M. P., Decay of correlations in finite abelian lattice gauge theories, Commun. in Math. Phys. (2022). 
    
    \bibitem{f2021b} Forsstr\"om, M. P., Wilson lines in the abelian lattice Higgs model, preprint (2021).
    
    \bibitem{flv2020} Forsstr\"om, M. P., Lenells, J., Viklund, F., Wilson loops in finite abelian lattice gauge theories, To appear in \emph{Annales de l'Institut Henri Poincar\'e (B) Probabilit\'es et Statistiques} (2022).
 
    \bibitem{flv2021} Forsstr\"om, M. P., Lenells, J., Viklund, F., Wilson loops in the abelian lattice Higgs model, To appear in \emph{Probability and Mathematical Physics} (2022).
    
    \bibitem{frolich-spencer82} Fr\"ohlich, J., Spencer, T., \textit{Massless phases and symmetry restoration in abelian gauge theories and spin systems}, Commun.\ Math.\ Phys.\ 83, 411--454 (1982).

    
     \bibitem{fs1979} Fradkin, E., Shenker, S. H. , Phase diagrams of lattice gauge theories with Higgs fields, Phys.\ Rev.\ D 19(12), (1979), 3682--3697.

    \bibitem{os1978} Osterwalder, K. , Seiler, E., Gauge field theories on a lattice, Annals of Physics  110 (2)  (1978).

    \bibitem{seiler1982} Seiler, E., Gauge Theories as a Problem of Constructive Quantum Field Theory and Statistical Mechanics, Lecture Notes in Physics (LNP, volume 159) (1982).


     \bibitem{gs2021} Garban, C., Supelveda, A., Improved spin-wave estimate for Wilson loops in \( U(1) \) lattice gauge gauge theory, preprint (2021).
    
    \bibitem{jsj1980} Jongeward, G. A., Stack, J. D., Jayaprakash, C., Monte Carlo calculations on \( \mathbb{Z}_2 \) gauge-Higgs theories, Phys.\ Rev.\ D 21(12), (1980), 3360--3368.

    \bibitem{ks1984} Kanaya, K., Sugiyama, Y., Meanfield Study of \( \mathbb{Z}_2 \) Higgs Model with Radial Excitations and Mode Correlation Problem, Prog.\ Theor.\ Phys., 72(6), (1984), 1158--1175.
    
    \bibitem{mms1979} Marra, R., Miracle-Sole, S., On the Statistical Mechanics of the Gauge Invariant Ising Model, Commun. Math. Phys. 67,  (1979), 233--240.

     \bibitem{s1988} Shrock, E., Lattice Higgs models, Nucl.\ Phys.\ B (Proc.\ Suppl.) 4 (1988), 373--389.

     \bibitem{w1971} Wegner, F. J., Duality in Generalized Ising Models and Phase Transitions without Local Order Parameters, J.\ Math.\ Phys.\ 12(10) (1971), 2259--2272.
    
    \bibitem{w1974} Wilson, K., \textit{Confinement of quarks}, Phys.\ Rev.\ D 10, (1974), 2445--2459. 
\end{thebibliography}
\end{document}